\documentclass[11pt]{article}

\usepackage{amsmath,amssymb,amsthm}
\usepackage{booktabs}
\usepackage{placeins}
\usepackage{enumitem}
\usepackage{geometry}
\usepackage[hidelinks]{hyperref}
\usepackage{microtype}
\numberwithin{equation}{section}
\newtheorem{theorem}{Theorem}[section]
\newtheorem{lemma}[theorem]{Lemma}
\newtheorem{proposition}[theorem]{Proposition}
\newtheorem{corollary}[theorem]{Corollary}
\theoremstyle{definition}
\newtheorem{definition}[theorem]{Definition}
\newtheorem{example}[theorem]{Example}
\theoremstyle{remark}
\newtheorem{remark}[theorem]{Remark}

\newcommand{\Z}{\mathbb{Z}}
\newcommand{\Q}{\mathbb{Q}}
\newcommand{\R}{\mathbb{R}}
\newcommand{\C}{\mathbb{C}}
\newcommand{\F}{\mathbb{F}}
\newcommand{\Qp}{\mathbb{Q}_p}
\newcommand{\Zp}{\mathbb{Z}_p}
\newcommand{\Li}{\operatorname{Li}}
\newcommand{\li}{\operatorname{li}}
\newcommand{\Res}{\operatorname{Res}}
\newcommand{\ord}{\operatorname{ord}}
\newcommand{\vol}{\operatorname{vol}}
\newcommand{\Hom}{\operatorname{Hom}}
\newcommand{\CycPres}{P_{\mathrm{cyc}}}
\newcommand{\CycPresM}[1]{P_{\mathrm{cyc},#1}}
\newcommand{\Ami}{\mathcal A}
\newcommand{\LA}{\mathrm{LA}}
\newcommand{\Dla}{\mathcal D^{\mathrm{la}}}
\newcommand{\qpoch}[2]{(#1;#2)_\infty}
\newcommand{\mob}{\mu_{\mathrm{Mob}}}
\newcommand{\teich}{\omega}
\newcommand{\ang}[1]{\langle #1\rangle}
\newcommand{\one}{1}
\newcommand{\starx}{\mathbin{\star_{\!\times}}}
\newcommand{\widebar}[1]{\overline{#1}}
\newcommand{\Iw}{\Lambda_{\mathcal O}}
\newcommand{\IwK}{\Lambda_K}
\newcommand{\IwPs}{\mathfrak I}
\newcommand{\Zhalf}{\mathcal Z_{1/2}}
\newcommand{\Zclass}{\mathcal Z_p}
\newcommand{\Regdep}{D_p^{[p]}}
\newcommand{\bd}{\mathrm{bd}}
\newcommand{\valp}{v_p}
\newcommand{\Disc}{\mathbb D}
\newcommand{\Regord}{\operatorname{Reg}^{\mathrm{ord}}}

\title{Bloch-regulator Principal Parts of Cyclotomic Iwasawa Pseudomeasures}
\author{HONGHUAI FANG AND ZEKUN CHEN}
\date{}
\hypersetup{
  pdfauthor={Honghuai Fang and Zekun Chen},
  pdftitle={Bloch-regulator principal parts of cyclotomic Iwasawa pseudomeasures},
  pdfsubject={Bloch-regulator principal parts, local K3 realization, and cyclotomic refinements},
  pdfkeywords={Habiro ring, Bloch group, Mellin transform, Iwasawa algebra, p-adic L-function, Coleman regulator}
}

\makeatletter
\renewcommand{\footnoterule}{%
  \kern-.4\p@
  \hrule\@width 5pc
  \kern11\p@
  \kern-\footnotesep
}
\newcommand{\titlepagebibliographicfootnotes}{%
  \insert\footins{%
    \normalfont\footnotesize
    \interlinepenalty\interfootnotelinepenalty
    \splittopskip\footnotesep
    \splitmaxdepth\dp\strutbox
    \floatingpenalty\@MM
    \hsize\columnwidth
    \@parboxrestore
    \parindent1.5em
    \indent
    {\itshape \textup{2020} Mathematics Subject Classification.}\enspace
    Primary 11S80, 11S70; Secondary 11R23, 19F27, 11M06.\par
    \indent
    {\itshape Key words and phrases.}\enspace
    Bloch group, Coleman regulator, Iwasawa pseudomeasure,
    \(p\)-adic \(L\)-function, Habiro ring, cyclotomic presentation.\par
  }%
}
\makeatother

\begin{document}
\maketitle
\titlepagebibliographicfootnotes

\begin{abstract}
Let \(p\) be an odd prime and let \(K/\Qp\) be a finite unramified
extension.  From a finite presentation by roots of unity of order prime to
\(p\), we construct a localized Iwasawa pseudomeasure on
\(\Zp^\times\).  Although the pseudomeasure depends on the chosen
presentation, its image modulo bounded measures depends only on the
associated Bloch class: it is the Frobenius-depleted Coleman regulator of
that class multiplied by a universal half-shifted zeta principal part.
Consequently, every nonexceptional weight component is bounded, while the
exceptional component has at most a simple pole with explicitly determined
residue.  For \(p>3\) and \(\Zp\)-valued coefficients, vanishing of the
principal part is equivalent to vanishing of the corresponding class in
\(K_3(K;\Zp)\).

Cyclotomic refinements preserve the Bloch class and act on the associated
pseudomeasures by explicit Iwasawa multipliers.  Normalized finite linear
combinations of refinements interpolate arbitrary finite jets of the
bounded weight-space data, subject only to the normalization at the
exceptional point.  We also establish a half-shifted complex Mellin
factorization, compare the construction with the GSWZ germ family, and
derive, at simple degree-one places above primes \(p>3\), a
finite-polylogarithm criterion for the local \(K_3\)-class of the knot
\(5_2\).
\end{abstract}

\begingroup
\setlength{\parskip}{0pt}
\tableofcontents
\endgroup
\clearpage

\section{Introduction and main results}\label{sec:intro}

\subsection{The cyclotomic Iwasawa pseudomeasure}

The Bloch group provides a common framework for weight-two algebraic
\(K\)-theory, dilogarithmic regulators, and ideal hyperbolic geometry.  Its
five-term relation is the algebraic form of the basic functional equation
of the dilogarithm, while Suslin's exact sequence relates the Bloch group of
a field to its indecomposable \(K_3\) \cite{Suslin}.  Complex and \(p\)-adic
regulators attach dilogarithmic invariants to Bloch classes
\cite{Coleman,BesserDeJeuSyntomic}; the same five-term calculus governs the
shape parameters of ideal tetrahedra \cite{Purcell}.

The constructions studied here begin before passage to the Bloch group.
Their input is a finite cyclotomic presentation
\[
 \xi=\sum_z a_z[z]_{\mathrm{cyc}},
 \qquad
 \eta_\xi=\operatorname{cl}(\xi),
\]
supported on roots of unity.  The presentation \(\xi\) retains the chosen
roots and coefficients, whereas \(\eta_\xi\) is its coefficient-extended
Bloch class.  Pochhammer germs and their associated distributions are
attached to \(\xi\); the Coleman regulator factors through
\(\eta_\xi\).  This distinction is also present in the Habiro-ring
construction of Garoufalidis--Scholze--Wheeler--Zagier (GSWZ), where a
root-of-unity presentation determines local logarithmic germs whose
singular terms are controlled by regulators \cite{GSWZ}.

Fix an odd prime \(p\), and let \(K/\Qp\) be a finite unramified extension
containing the support and coefficients of \(\xi\).  Put
\[
 G=\Zp^\times,
 \qquad
 \IwK(G)=\mathcal O_K[[G]][1/p].
\]
For a support root \(z\ne1\) of order prime to \(p\), let
\[
 \mu_z(a+p^m\Zp)=\frac{z^a}{1-z^{p^m}},
 \qquad
 \nu_z^\times=x^{-1}\one_G\mu_z,
 \qquad
 \nu_\xi^\times=\sum_z a_z\nu_z^\times.
\]
For every odd integer \(c>1\), prime to \(p\), set
\(r_c=\delta_1-c\delta_c\), and let \(\Sigma\) be the multiplicative set
generated by the \(r_c\).  The half-shift measure
\(\lambda_{1/2,c}^\times\) constructed in
Subsection~\ref{subsec:half-regularizer} defines
\begin{equation}\label{eq:intro-Iwasawa-pseudomeasure}
 \IwPs_{\xi,p}
 :=r_c^{-1}\starx
   \bigl(\lambda_{1/2,c}^\times\starx\nu_\xi^\times\bigr)
 \in\Sigma^{-1}\IwK(G).
\end{equation}
The cross relation proved in Theorem~\ref{thm:cross-relation} shows that
this element is independent of \(c\).  Thus
\(\IwPs_{\xi,p}\) is a well-defined cyclotomic Iwasawa pseudomeasure
attached to the chosen presentation.  It is linear in \(\xi\).

The same regularization applied without the presentation-dependent factor
defines the half-shift zeta pseudomeasure
\[
 \Zhalf=r_c^{-1}\starx\lambda_{1/2,c}^\times.
\]
Associativity of multiplicative convolution gives the canonical
factorization
\begin{equation}\label{eq:intro-canonical-factorization}
 \IwPs_{\xi,p}=\Zhalf\starx\nu_\xi^\times.
\end{equation}
Here \(\nu_\xi^\times\) is a bounded measure containing the dependence on
the chosen presentation, whereas \(\Zhalf\) is universal and contains all
possible singularities.  We study the singularities of
\(\IwPs_{\xi,p}\), the dependence of its principal part on the Bloch class
\(\eta_\xi\), and the variation of its bounded components under
cyclotomic refinement.

The analytic origin of the construction is the fixed-center half-shifted
Pochhammer expansion.  For a prime-to-\(p\) root of unity \(z\ne1\), its
coefficientwise \(p\)-adic Laurent specialization has the form
\[
 \frac{\Li_{2,p}(z)}{Y}
 +\sum_{n\ge0}
   \frac{B_{n+1}(1/2)}{(n+1)!}\Li_{1-n,p}(z)Y^n.
\]
The polar coefficient is Coleman's dilogarithm, while the regular
coefficients are lower polylogarithmic values attached to the chosen
symbol.  Coleman--Koblitz measures realize these values as moments, and
the Amice transform converts the corresponding power series into locally
analytic distributions.  Restriction from \(\Zp\) to \(\Zp^\times\)
removes the \(p\)-divisible part of the moments and produces, at the
exceptional negative weight, the Frobenius-depleted value
\[
 D_p(z)-p^{-2}D_p(z^p).
\]
The half-shift kernel is not itself bounded.  The power-difference
regularization indexed by \(c\) cancels its pole and corresponds on
\(G\) to convolution by \(r_c\); localization then restores the singular
object in the form \eqref{eq:intro-Iwasawa-pseudomeasure}.  Thus the
passage from the Pochhammer germ to \(\IwPs_{\xi,p}\) separates a
regulator-controlled singular term from presentation-dependent bounded
data.

A root-of-unity symbol need not define an integral Bloch element, since its
boundary may be torsion.  We therefore work with flat coefficient rings in
which the support orders are invertible.  Locally, the support orders are
assumed prime to \(p\); after finite unramified extension the roots become
Teichm\"uller units, and the presentation determines a well-defined
coefficient-extended Bloch class.  The precise admissibility convention is
given in Section~\ref{sec:presentations}.

\subsection[Principal parts and local K3]{Principal parts and local \(K_3\)}

Let
\[
 \mathrm{PP}_K(G)=\Sigma^{-1}\IwK(G)/\IwK(G),
\]
and write \(\operatorname{pp}\) for the quotient map.  This quotient
forgets bounded measures and retains exactly the singular contribution of
a pseudomeasure.  If \(\phi\) is the Frobenius lift on \(K\), define
\[
 \Phi_{\mathrm{cyc}}(a[z])=a[z^p],
 \qquad
 D_{p,K}^{[p]}(\eta)
 =D_{p,K}(\eta)-p^{-2}D_{p,K}(\Phi_{\mathrm{cyc}}\eta).
\]
The local \(K\)-theory convention is fixed in
Subsection~\ref{subsec:local-K-convention}.

The decomposition
\(G=\mu_{p-1}\times(1+p\Zp)\) gives \(p-1\) weight components.  The
universal denominator is a unit on all but one of these components; on the
exceptional component it has a single simple zero.  The main theorem
identifies the resulting principal part with the depleted Coleman
regulator of the Bloch class.

\begin{theorem}[Principal-part theorem]
\label{thm:intro-main-principal-part}
Let \(\xi\) be a finite cyclotomic presentation supported on
prime-to-\(p\) roots in \(K\), and put
\(\eta_\xi=\operatorname{cl}_{K,K}(\xi)\in B(K)_K\).  Then:
\begin{enumerate}[label=(\roman*)]
\item Every component with \(j\ne-1\pmod{p-1}\) is a bounded measure.
The exceptional component has at most a simple pole at \(s=-1\), and
\[
 \Res_{s=-1}\mathcal M_{\IwPs_{\xi,p},-1}(s)
 =(1-p^{-1})D_{p,K}^{[p]}(\eta_\xi).
\]
The pseudomeasure \(\IwPs_{\xi,p}\) is bounded if and only if
\(D_{p,K}^{[p]}(\eta_\xi)=0\).

\item Its complete principal part is
\begin{equation}\label{eq:intro-main-principal-part}
 \operatorname{pp}(\IwPs_{\xi,p})
 =D_{p,K}^{[p]}(\eta_\xi)\operatorname{pp}(\Zhalf).
\end{equation}
In particular, two presentations of the same coefficient-extended Bloch
class give pseudomeasures whose difference is bounded.
\end{enumerate}
\end{theorem}

Equivalently, there is a bounded measure \(B_\xi\in\IwK(G)\) such that
\begin{equation}\label{eq:intro-bounded-remainder}
 \IwPs_{\xi,p}
 =D_{p,K}^{[p]}(\eta_\xi)\Zhalf+B_\xi.
\end{equation}
Thus the universal factor determines the location and order of the
possible singularity, the depleted Coleman regulator determines its
coefficient, and the chosen presentation affects only the bounded
remainder.  Theorem~\ref{thm:half-zeta-identification} further identifies
\[
 \Zhalf=(\delta_1-\delta_{2^{-1}})\starx\Zclass,
\]
where \(\Zclass\) is the classical Iwasawa zeta pseudomeasure.  Hence the
unique possible pole of \(\IwPs_{\xi,p}\) is inherited from the classical
Iwasawa zeta factor.  The principal-part formula is stronger than the
residue identity: it determines the complete class modulo bounded
measures, not merely the first Laurent coefficient.

For integral coefficients, this principal part realizes the corresponding
local \(K_3\)-class.

\begin{corollary}[Local \(K_3\)-realization]
\label{cor:intro-local-K3-realization}
Suppose \(p>3\), let \(\kappa_{K,p}\) be the comparison map of
Subsection~\ref{subsec:local-K-convention}, and assume that \(\xi\) has
coefficients in \(\Zp\).  Then
\[
 \operatorname{pp}(\IwPs_{\xi,p})=0
 \quad\Longleftrightarrow\quad
 \kappa_{K,p}(\eta_\xi)=0
 \quad\text{in }K_3(K;\Zp).
\]
Moreover,
\[
 K_3(K;\Zp)
 \overset{\sim}{\longrightarrow}
 \mathcal O_K\operatorname{pp}(\Zhalf),
 \qquad
 x\longmapsto
 (1-p^{-2}\phi)D_p(x)\operatorname{pp}(\Zhalf),
\]
is a \(\Zp\)-linear isomorphism.
\end{corollary}

The proof constructs the regularized half-shift measure, establishes the
cross relation, evaluates the bounded factor at the exceptional point by
negative moments, and applies division in the bounded-denominator Iwasawa
algebra.  The local \(K_3\)-statement then follows from the local regulator
theorem of \cite{GSWZ} and the Bloch--Suslin comparison.

\subsection{Cyclotomic refinements and finite-jet variation}

The principal-part theorem isolates the part of \(\IwPs_{\xi,p}\) that is
intrinsic to the Bloch class.  To measure the remaining presentation
dependence, let \(m\) be prime to \(p\) and define
\[
 R_m[z]_{\mathrm{cyc}}
 =m\sum_{u^m=z}[u]_{\mathrm{cyc}}.
\]
The Bloch distribution relation preserves the class, while the analytic
construction transforms by an explicit Iwasawa multiplier.  After finite
unramified extension,
\begin{equation}\label{eq:intro-refinement-covariance}
 \operatorname{cl}(R_m\xi)=\operatorname{cl}(\xi),
 \qquad
 \IwPs_{R_m\xi,p}
 =(m\delta_m)\starx\IwPs_{\xi,p}.
\end{equation}
Finite linear combinations of the operators \(R_m\) control the Taylor
jets of the component multipliers.

\begin{theorem}[Finite-jet interpolation by refinements]
\label{thm:intro-refinement-interpolation}
Let \(\xi\in\CycPres(K;\mathcal O_K)\).  For
\[
 e=0,\ldots,p-2,
 \qquad
 j_e\equiv e-1\pmod{p-1},
 \qquad
 Y_e=(1+p)^{s-(e-1)}-1,
\]
fix \(N\ge0\) and arbitrary coefficients
\[
 v_{e,q}\in\mathcal O_K
 \quad(0\le e\le p-2,\ 0\le q\le N),
 \qquad
 v_{0,0}=1.
\]
Then, after finite unramified extension, there is a normalized finite linear
refinement
\[
 C=\sum_m c_mR_m,
 \qquad
 \sum_m c_m=1,
\]
such that \(C\xi\) represents the base-changed Bloch class of \(\xi\) and
\[
 \mathcal M_{\IwPs_{C\xi,p},j_e}(s)
 =H_{C,e}(Y_e)\mathcal M_{\IwPs_{\xi,p},j_e}(s),
 \qquad
 [Y_e^q]H_{C,e}=v_{e,q}.
\]
\end{theorem}

Modulo \(p\), the interpolation matrix is the tensor product of the
power-evaluation matrix on \(\F_p^\times\) and a triangular
Mahler--Pascal matrix.  Consequently, the nonexceptional
\(\lambda\)-invariants can be increased independently without changing
their \(\mu\)-invariants.  If the exceptional residue is nonzero, every
finite collection of subsequent Laurent coefficients can be prescribed
after scalar extension, while the polar coefficient remains fixed.
Normalized refinements therefore preserve the principal part and allow
finite jets of the component multipliers to be prescribed subject to
\(H_{C,0}(0)=1\).

\subsection[Mellin factorizations, GSWZ comparison, and the knot 5-2]{Mellin factorizations, GSWZ comparison, and the knot \(5_2\)}

The separation in \eqref{eq:intro-canonical-factorization} has parallel
complex and \(p\)-adic Mellin forms.  For a primitive odd Dirichlet
character \(\chi\) modulo \(N\), the complex half-shifted Pochhammer
transform satisfies
\[
 L_{\xi_{\chi,N}}(s)
 =-G(\chi)(1-2^{-s})\zeta(s)L(s+1,\chi^{-1})
 \qquad(\Re(s)>1).
\]
On the \(p\)-adic side, the corresponding weight component factors as
\[
 \mathcal M_{\IwPs_{\chi,N,p},j}(s)
 =G(\chi)Z_{1/2,j}(s)
   L_p(1-s,\chi^{-1}\teich^j),
\]
with the reflected Kubota--Leopoldt normalization of
Subsection~\ref{subsec:KL-comparison}.  Imprimitive packets contribute an
additional explicit finite factor.  In both settings, the half shift gives
the zeta factor and the character packet gives the Dirichlet or
Kubota--Leopoldt factor.

Under the hypotheses of \cite{GSWZ}, the fixed-center logarithmic germ is
carried to the Iwasawa construction by the Amice transform.  At a center of
exact order \(m\), the power map \(q\mapsto q^c\) changes the order to
\(m/(m,c)\).  Cancellation of the common regulator pole forces the
coefficient \(c/(m,c)^2\), and the resulting regularization is compatible
with the logarithmic Frobenius difference.  This identifies
\(\IwPs_{\xi,p}\) as the unit-weight-space realization of the regulator
singularity in the GSWZ germ family.

The knot \(5_2\) provides a concrete local application.  Let
\[
 F=\Q(\alpha),
 \qquad
 \alpha^3-\alpha^2+1=0,
 \qquad
 \beta_{5_2}=2[1-\alpha^2]+[1-\alpha]\in B(F).
\]
At a simple degree-one place above a prime \(p>3\), let
\(a=1-\bar\alpha^2\in\F_p\).  Then \(a\ne-2\), and put
\[
 \li_{r,p}(X)=\sum_{k=1}^{p-1}k^{-r}X^k,
 \qquad
 W_p(X)=\li_{2,p}(X)
       +\frac{3\li_{1,p}(X)^2}{2(X+2)}.
\]
Besser's residue-disc expansion \cite{Besser} gives an explicit congruence
for the first \(p\)-adic digit of the Coleman regulator and yields
\[
 \beta^K_{5_2,p,\bar\alpha}
 \text{ generates }K_3(\Qp;\Zp)
 \quad\Longleftrightarrow\quad
 W_p(a)\ne0.
\]
Thus the local generator problem is reduced to a finite-polylogarithm
calculation over \(\F_p\).

\subsection{Organization and hypotheses}

Section~\ref{sec:presentations} defines admissible cyclotomic
presentations, character packets, coefficient extensions, and the local
\(K_3\) comparison.  Section~\ref{sec:local-germs} constructs the formal
Pochhammer germs and realizes their coefficients by Amice transforms and
Coleman--Koblitz measures.  Section~\ref{sec:pseudomeasures} constructs the
cyclotomic Iwasawa pseudomeasure, identifies the half-zeta factor, isolates
the exceptional component, proves the principal-part and local \(K_3\)
statements, and derives the character factorizations.
Section~\ref{sec:fourier-refinement} proves the cyclotomic distribution and
refinement results, establishes finite-jet interpolation, compares the
construction with complex Mellin transforms and GSWZ germs, and proves the
criterion for \(5_2\).  The appendices contain the character and Kummer
calculations, the analytic details of the complex Mellin formulas, the
higher-center Frobenius gluing argument, and the local calculations for
\(5_2\).

The fixed-center construction requires that \(p\) be odd and prime to the
support orders.  Unramifiedness is used for Frobenius descent.  The
assumption \(p>3\) enters the local Bloch--\(K_3\) comparison and the
application to \(5_2\).  The refinement statements allow finite
unramified base change in order to contain the newly introduced roots.

\section{Admissible cyclotomic presentations}\label{sec:presentations}

\subsection{Pre-Bloch groups, torsion boundaries, and scalar extension}

\begin{definition}\label{def:bloch}
For a field $F$, let $\mathcal P(F)$ be the free abelian group on symbols $[z]$, $z\in F\setminus\{0,1\}$, modulo the five-term relation
\[
[x]-[y]+\left[\frac yx\right]
-\left[\frac{1-x^{-1}}{1-y^{-1}}\right]
+\left[\frac{1-x}{1-y}\right]=0.
\]
Here $x,y\in F\setminus\{0,1\}$ and $x\ne y$, so every displayed argument again belongs to $F\setminus\{0,1\}$.  A direct expansion in the exterior square shows that the five-term relation is killed by $[z]\mapsto z\wedge(1-z)$; this is the standard Bloch--Suslin boundary calculation
\cite[Lem.~1.1, p.~181]{Suslin}.  Thus the boundary map
\[
\partial:\mathcal P(F)\longrightarrow \bigwedge^2F^\times,
\qquad
\partial[z]=z\wedge(1-z),
\]
is well defined, and the Bloch group is $B(F)=\ker\partial$.  For a
commutative $\Z$-algebra $R$, write
\[
B(F)_R=B(F)\otimes_\Z R.
\]
\end{definition}

\begin{remark}\label{rem:Bloch-group-convention}
Suslin formulates the boundary using the antisymmetric tensor quotient
\((F^\times\otimes F^\times)_\sigma\)
\cite[Sec.~1, p.~181]{Suslin}, whereas
Definition~\ref{def:bloch} uses the ordinary exterior square.  The natural
map from the former to the latter has \(2\)-primary kernel, so the two
Bloch-group conventions agree after inverting \(2\).  In particular, they
agree in every \(p>3\) comparison with local \(K_3\).  The
integral distribution identities are identities in the
pre-Bloch group itself and do not depend on this boundary convention.
\end{remark}

A root of unity need not define an integral Bloch element, since its boundary may be torsion.  This torsion must be annihilated before defining the class map.

\begin{lemma}[Admissible root-of-unity class map]\label{lem:admissible-class-map}
Let $S\subset\mu(F)\setminus\{1\}$ be finite, let $M$ be a common multiple of the orders of the elements of $S$, and let $R$ be flat over $\Z$ with $M\in R^\times$.  Then every symbol $[\zeta]$, $\zeta\in S$, lies in
\[
\ker\bigl(\partial\otimes 1_R:\mathcal P(F)\otimes R
\longrightarrow (\bigwedge^2F^\times)\otimes R\bigr).
\]
Flatness identifies this kernel with $B(F)_R$.  Consequently there is a canonical $R$-linear map
\[
\operatorname{cl}_{S,R}:
\bigoplus_{\zeta\in S}R[\zeta]_{\mathrm{cyc}}
\longrightarrow B(F)_R,
\qquad
[\zeta]_{\mathrm{cyc}}\longmapsto[\zeta]\otimes1.
\]
\end{lemma}

\begin{proof}
If $\zeta^M=1$, then bilinearity of the exterior product gives
\[
M\,\partial[\zeta]
=M\bigl(\zeta\wedge(1-\zeta)\bigr)
=\zeta^M\wedge(1-\zeta)
=1\wedge(1-\zeta)=0.
\]
Thus $\partial[\zeta]$ is annihilated by $M$.  Tensoring with a ring in which $M$ is invertible kills this boundary.  Since $R$ is flat, tensoring the exact sequence
\[
0\longrightarrow B(F)\longrightarrow\mathcal P(F)
\overset{\partial}{\longrightarrow}\bigwedge^2F^\times
\]
preserves exactness at the middle term.  Hence the kernel after scalar extension is $B(F)_R$.
\end{proof}

If every support order is prime to \(p\), then each such order is a unit in
\(\Zp\) and in every finite unramified extension.  The class maps commute
with field extension and
with flat scalar extension, since both routes send
\([\zeta]_{\mathrm{cyc}}\) to \([\zeta]\otimes1\).  Hence a chosen
cyclotomic presentation has compatible global, complex, and local Bloch
classes once the relevant embeddings are fixed.

\subsection{Presentation modules and anti-invariant parts}

\begin{definition}[Admissible presentation module]\label{def:admissible-presentation}
Under the hypotheses of Lemma~\ref{lem:admissible-class-map}, define
\[
P_{\mathrm{cyc},S}(F;R)
=
\bigoplus_{\zeta\in S}R[\zeta]_{\mathrm{cyc}}.
\]
An \emph{admissible chosen cyclotomic presentation} is a pair
\[
(\eta,\xi),
\qquad
\xi=\sum_{\zeta\in S}a_\zeta[\zeta]_{\mathrm{cyc}},
\qquad
\eta=\operatorname{cl}_{S,R}(\xi)\in B(F)_R.
\]
For a flat coefficient ring $R$, put
\[
 \mu_{\mathrm{adm}}(F;R)=
 \{\zeta\in\mu(F)\setminus\{1\}:\ord(\zeta)\in R^\times\},
 \qquad
 \CycPres(F;R)=
 \bigoplus_{\zeta\in\mu_{\mathrm{adm}}(F;R)}R[\zeta]_{\mathrm{cyc}}.
\]
Every element of this direct sum has finite support.  If $\xi$ has support $S$, write
\[
 \operatorname{cl}_{F,R}(\xi):=\operatorname{cl}_{S,R}(\xi).
\]
Functoriality makes this independent of enlarging the finite set $S$.
When the field and coefficient ring are clear, we abbreviate
$\operatorname{cl}_{F,R}$ to $\operatorname{cl}$.
\end{definition}

For an integer $M\ge2$, assume that a primitive root $\zeta_M\in F$ has been chosen, let $S_M\subset F$ be the set of primitive $M$-th roots, and write
\[
\CycPresM M(F;R)
=
\bigoplus_{a\in(\Z/M\Z)^\times}R[\zeta_M^a]_{\mathrm{cyc}}.
\]
The inversion involution is
\[
\iota[\zeta_M^a]_{\mathrm{cyc}}
=[\zeta_M^{-a}]_{\mathrm{cyc}}.
\]
If $2\in R^\times$, set
\[
\CycPresM M(F;R)^-
=
\{\xi:\iota\xi=-\xi\}.
\]
In the pre-Bloch group, the standard relation
\[
 2\bigl([z]+[z^{-1}]\bigr)=0
\]
implies $[z^{-1}]=-[z]$ after inverting $2$ \cite[Lem.~1.2, p.~181]{Suslin}.  For the
admissible roots under consideration, both symbols lie in the
coefficient-extended Bloch group by
Lemma~\ref{lem:admissible-class-map}; hence the class map factors through
the anti-invariant quotient.  The formal anti-invariant presentation remains
part of the input, because the Pochhammer and pseudomeasure constructions are
defined before passing to the Bloch quotient.

If \(2\in R^\times\), the projectors
\((1\pm\iota)/2\) give the decomposition
\[
 \xi=\xi^++\xi^-,
 \qquad
 \xi^\pm=\frac12(\xi\pm\iota\xi).
\]
\subsection[Character packets, local conventions, and K-theory]{Character packets, local conventions, and \(K\)-theory}

Let $\chi$ be a Dirichlet character modulo $N$, extended by zero away from integers prime to $N$, and put
\[
E_\chi=\Q(\chi),
\qquad
\chi^\vee=\chi^{-1}.
\]
The exact-order packet is
\begin{equation}\label{eq:character-packet}
\xi_{\chi,N}
=
\sum_{a\in(\Z/N\Z)^\times}
\chi(a)[\zeta_N^a]_{\mathrm{cyc}}
\in\CycPresM N(\Q(\zeta_N);E_\chi).
\end{equation}
If $\chi(-1)=-1$, the change of variables $b=-a$ gives
\[
 \iota\xi_{\chi,N}=-\xi_{\chi,N}.
\]
Thus the packet is anti-invariant, and so is its coefficient-extended
Bloch class
\[
\eta_{\chi,N}
=
\operatorname{cl}(\xi_{\chi,N}).
\]
For a primitive Dirichlet character $\chi$ modulo $N$, normalize its Gauss
sum by
\begin{equation}\label{eq:gauss-sum-normalization}
 G(\chi)=\sum_{a\in(\Z/N\Z)^\times}\chi(a)\zeta_N^a.
\end{equation}
If $\chi$ is induced from a primitive character $\chi_f$ of conductor $f$ and $N=fg$, the symbols in \eqref{eq:character-packet} still have exact order $N$, while the Fourier coefficients are controlled by $\chi_f$.  This distinction is the source of the finite imprimitive correction.

\subsubsection{Complex and local root conventions}\label{subsec:root-conventions}

We use two systems of roots of unity for two different analytic purposes.

\paragraph{Complex radial convention.}
For every complex root $\lambda$ of exact order $m$, choose a square root
\[
 \widetilde\lambda^2=\lambda.
\]
For $\lambda=1$, require $\widetilde\lambda=1$.  Square roots are chosen independently for distinct complex roots.  At the center $\lambda=1$, this convention makes the basic radial path agree with $q=e^{-2\hbar}$.

\paragraph{Local compatible convention.}
Inside a fixed algebraic closure of the local coefficient field, choose a compatible family $\{\zeta_m\}_{m\ge1}$ in which $\zeta_m$ has exact order $m$ and
\[
\zeta_{mn}=\zeta_m\zeta_n\quad((m,n)=1),
\qquad
\zeta_{p^r}^p=\zeta_{p^{r-1}}.
\]
The local coordinate at the $m$-th center is $q=\zeta_m+x$.  The complex radial square roots and the locally compatible roots are separate choices; only the latter enter Frobenius gluing.

\subsubsection{Local embeddings and Teichm\"uller decomposition}

Let $p$ be odd, let $K_p/\Qp$ be finite, and let $\mathcal O_{K_p}$ be its ring of integers.  We write $\mathcal O=\mathcal O_{K_p}$ throughout the local analytic sections.  We assume that the orders of all roots in the presentation are prime to $p$, so they embed as Teichm\"uller roots in an unramified extension after enlarging $K_p$ if necessary.  For $x\in\Zp^\times$, write
\[
x=\teich(x)\ang x,
\qquad
\teich(x)\in\mu_{p-1},
\qquad
\ang x\in1+p\Zp.
\]
For $s\in\Zp$, the function $\ang x^s=\exp(s\log\ang x)$ is locally analytic.  This is the multiplicative coordinate used in all unit Mellin transforms.

\begin{remark}\label{rem:notation}
The letter $\zeta$ denotes a Teichm\"uller root of order prime to $p$ in
the local sections and an ordinary complex root in the complex section.
The ambient field and chosen embedding are specified when needed.  Coleman
logarithms are normalized by $\log_p(p)=0$ and vanish on Teichm\"uller
roots of order prime to $p$.
\end{remark}

\subsubsection[Local K-theory convention]{Local \(K\)-theory convention}
\label{subsec:local-K-convention}

For an unramified finite extension \(L/\Qp\) with \(p>3\), put
\[
 K_3(L)_{\mathrm{ind}}
 :=\operatorname{coker}\!\left(K_3^M(L)\longrightarrow K_3(L)\right)
\]
and define the \(p\)-adic coefficient group by
\[
 K_3(L;\Zp):=\varprojlim_n K_3(L;\Z/p^n\Z).
\]
The étale comparison used in \cite[Thm.~9 and its proof]{GSWZ} identifies
this group with \(H^1(L,\Zp(2))\).  Also put
\[
 \widehat{B(L)}_p=\varprojlim_n B(L)/p^nB(L).
\]
Suslin's exact sequence \cite[Thm.~5.2, p.~197]{Suslin} is
\[
 0\longrightarrow
 \widetilde{\operatorname{Tor}}(\mu(L),\mu(L))
 \longrightarrow K_3(L)_{\mathrm{ind}}
 \longrightarrow B(L)\longrightarrow0.
\]
Here \(\widetilde{\operatorname{Tor}}(\mu(L),\mu(L))\) denotes Suslin's
canonical extension of \(\operatorname{Tor}_1^{\Z}(\mu(L),\mu(L))\) by
\(\Z/2\).  For unramified \(L/\Qp\), this finite group has order prime to
\(p\), since \(\mu(L)\) has order prime to \(p\).  Hence \(p\)-completion
identifies \(\widehat{B(L)}_p\) with the \(p\)-completion of
\(K_3(L)_{\mathrm{ind}}\).  The comparison in the proof of
\cite[Thm.~9]{GSWZ} then identifies this group with \(K_3(L;\Zp)\); the
Milnor and \(K_2\) terms make no contribution in this range.  Consequently,
the natural map \(B(L)\otimes\Zp\to\widehat{B(L)}_p\) defines
\begin{equation}\label{eq:local-Bloch-K3-comparison}
 \kappa_{L,p}:B(L)\otimes\Zp\longrightarrow K_3(L;\Zp).
\end{equation}
All local \(K_3\)-statements below concern the image under
\(\kappa_{L,p}\).  For a finite extension \(E/\Qp\), put
\[
 \kappa_{L,p,E}=\kappa_{L,p}\otimes1_E:
 B(L)\otimes_{\Z}E
 \longrightarrow
 K_3(L;\Zp)\otimes_{\Zp}E.
\]
These maps and the Coleman regulator are natural under unramified field
extension.

\section{Formal local germs and unit distributions}\label{sec:local-germs}

We write
\[
 (a;q)_\infty=\prod_{r\ge0}(1-aq^r).
\]
The \(q\)-Pochhammer expansion is used only to derive the coefficients of
the formal germ; the resulting \(p\)-adic germ is defined coefficientwise.

Fix an odd prime $p$.  Let $K/\Qp$ be a finite extension containing the
coefficient field and all prime-to-$p$ roots of unity in the presentation,
and let $\mathcal O_K$ be its valuation ring.  In this section we write
$\mathcal O=\mathcal O_K$.  Normalize the Iwasawa logarithm
by
\[
 \log_p(p)=0,
 \qquad
 \log_p(\teich(u))=0
 \quad(u\in\mathcal O^\times).
\]
We use Coleman's single-valued dilogarithm \cite{Coleman} in the convention
\begin{equation}\label{eq:Dp-normalization}
 D_p(z)=\Li_{2,p}(z)+\frac12\log_p(z)\log_p(1-z).
\end{equation}
Up to the global sign in the comparison map, this agrees with the
weight-two syntomic regulator formula of
\cite[Thm.~1.6(2) and Rem.~1.7]{BesserDeJeuSyntomic}.  For a prime-to-$p$
Teichm\"uller root $z$, one has $D_p(z)=\Li_{2,p}(z)$.

Suppose the support lies in a finite extension $F/\Qp$ and the coefficients
in a finite extension $E/\Qp$, both embedded in $K$.  Define
\[
 D_p^{\mathrm{tens}}=D_p\otimes1:
 B(F)\otimes_\Z E\longrightarrow F\otimes_{\Qp}E
\]
and let
\[
 \iota_{F,E}:F\otimes_{\Qp}E\longrightarrow K,
 \qquad
 a\otimes e\longmapsto ae.
\]
We write $D_{p,E}=\iota_{F,E}\circ D_p^{\mathrm{tens}}$.

When the embedded subfields are Frobenius stable, Frobenius acts diagonally
on the support and coefficient factors.  If the coefficients are fixed by
Frobenius, we suppress $E$ and write simply $D_p$.  We use the same symbol
for the Coleman regulator on $K_3(L;\Zp)$ and for its pullback to
$B(L)\otimes\Zp$ through $\kappa_{L,p}$.  On a cyclotomic presentation
$\xi=\sum_z a_z[z]_{\mathrm{cyc}}$, the coefficientwise evaluation is
\[
 \sum_z a_zD_p(z).
\]
Whenever $\xi$ represents a Bloch class, this sum equals the regulator of
$\operatorname{cl}(\xi)$ and hence of its image under $\kappa_{L,p}$.

\subsection{Formal Pochhammer germs}
\label{subsec:presentation-formal-branches}

Let $m$ be prime to $p$, put $T=e^Y-1$, and first take $z,Y\in\C$ with
$|z|<1$ and $\Re(Y)<0$.  The Bernoulli-polynomial generating function gives
\begin{equation}\label{eq:half-bernoulli-generating}
 \frac{ue^{u/2}}{e^u-1}
 =
 \sum_{r\ge0}B_r(1/2)\frac{u^r}{r!}.
\end{equation}
After division by $u$, this becomes
\begin{equation}\label{eq:half-kernel-full-expansion}
 \frac{e^{u/2}}{e^u-1}
 =\frac1u+
 \sum_{n\ge0}\frac{B_{n+1}(1/2)}{(n+1)!}u^n.
\end{equation}
The Pochhammer logarithm and the series in $r$ converge absolutely, so
\begin{align*}
 \frac1m\log(e^{mY/2}z;e^{mY})_\infty
 &=
 \frac1m\sum_{r\ge1}\frac{z^r}{r}
 \frac{e^{mrY/2}}{e^{mrY}-1}\\
 &=
 \frac1{m^2Y}\sum_{r\ge1}\frac{z^r}{r^2}
 +\sum_{n\ge0}m^{n-1}
 \frac{B_{n+1}(1/2)}{(n+1)!}
 \left(\sum_{r\ge1}z^rr^{n-1}\right)Y^n.
\end{align*}
Expanding at $Y=0$ gives an identity in $\Q[[z]]((Y))$: coefficientwise
in $z^r$, it is \eqref{eq:half-kernel-full-expansion} with $u=mrY$.
After base change to $K$, define its coefficientwise $p$-adic Laurent
specialization at a prime-to-$p$ root $z\ne1$ by
\begin{equation}\label{eq:higher-formal-coefficient-expansion}
 \frac{\Li_{2,p}(z)}{m^2Y}
 +
 \sum_{n\ge0}
 m^{n-1}
 \frac{B_{n+1}(1/2)}{(n+1)!}
 \Li_{1-n,p}(z)Y^n.
\end{equation}
Here the polar coefficient is Coleman's $\Li_{2,p}$,
$\Li_{1,p}(z)=-\log_p(1-z)$, and $\Li_{1-n,p}$ for $n\ge1$ is a rational
polylogarithm regular at $z\ne1$.  Equation
\eqref{eq:higher-formal-coefficient-expansion} therefore defines the
$p$-adic Laurent series coefficientwise.  Removing its polar term and its
vanishing constant term gives the formal germ below.

\begin{definition}[Presentation-labeled formal germs]
\label{def:presentation-formal-branches}
For every odd $p$, every $m$ prime to $p$, and every prime-to-$p$ root
$z\ne1$, define the unique $F_{[z],p,m}(T)\in K[[T]]$ by
\begin{equation}\label{eq:normalized-single-branch-log}
 F_{[z],p,m}(e^Y-1)
 =
 \sum_{n\ge1}
 m^{n-1}
 \frac{B_{n+1}(1/2)}{(n+1)!}
 \Li_{1-n,p}(z)Y^n.
\end{equation}
Put
\begin{equation}\label{eq:normalized-single-branch}
 U_{[z],p,m}(T)=\exp\bigl(F_{[z],p,m}(T)\bigr).
\end{equation}
Thus $F_{[z],p,m}(0)=0$ and $U_{[z],p,m}(0)=1$.
\end{definition}

For a finite presentation
\[
 \xi=\sum_z a_z[z]_{\mathrm{cyc}},
\]
define
\begin{align}
 F_{\xi,p,m}(T)
 &=
 \sum_z a_zF_{[z],p,m}(T),
 \label{eq:normalized-presentation-log}
 \\
 U_{\xi,p,m}(T)
 &=
 \exp\bigl(F_{\xi,p,m}(T)\bigr).
 \label{eq:normalized-presentation-branch}
\end{align}
The germs $F_{\xi,p,m}$ form a coefficientwise family indexed by
$(m,p)=1$.  Under the GSWZ hypotheses,
Proposition~\ref{prop:fixed-higher-compatibility} identifies the fixed-center
regularization of this family with the corresponding GSWZ germ.

\begin{remark}\label{rem:class-vs-presentation-local}
The normalized logarithm \eqref{eq:normalized-presentation-log} is linear in
the written presentation, whereas its polar coefficient factors through the
coefficient-extended local Bloch class and, in the unramified $p>3$ range,
through its image under $\kappa_{L,p,E}$.
\end{remark}

\subsubsection{The fixed-center germ}\label{subsec:fixed-local-germ}

For every odd prime and every prime-to-\(p\) root \(z\ne1\), write
\[
 F_{[z],p}:=F_{[z],p,1}.
\]
For a finite presentation \(\xi=\sum_z a_z[z]_{\mathrm{cyc}}\), put
\[
 F_{\xi,p}(T)=\sum_z a_zF_{[z],p}(T).
\]
Because \(e^Y-1=Y+O(Y^2)\), substitution \(T=e^Y-1\) identifies
\(K[[T]]\) with \(K[[Y]]\), and
\eqref{eq:normalized-single-branch-log} supplies the coefficients.  For
\(n\ge1\), the symbol \(\Li_{1-n,p}\) denotes the usual rational
polylogarithm evaluated at the Teichm\"uller root.

Put $Y=\log(1+T)$.  From
\eqref{eq:half-bernoulli-generating},
\[
 B_1(1/2)=0,
 \qquad
 B_r(1/2)=0\quad(r>1\text{ odd}),
 \qquad
 B_r(1/2)=(2^{1-r}-1)B_r.
\]

Setting $m=1$ in \eqref{eq:normalized-single-branch-log} and summing with
weights $a_z$ gives, for every finite presentation,
\begin{equation}\label{eq:presentation-local-germ}
 F_{\xi,p}(e^Y-1)
 =
 \sum_{n\ge1}
 \frac{B_{n+1}(1/2)}{(n+1)!}
 S_n(\xi)Y^n,
 \qquad
 S_n(\xi)=\sum_z a_z\Li_{1-n,p}(z).
\end{equation}
For $n\ge1$, define its $p$-depleted counterpart by
\begin{equation}\label{eq:Snp-depleted}
 S_n^{[p]}(\xi)
 =
 \sum_z a_z
 \left(\Li_{1-n,p}(z)-p^{n-1}\Li_{1-n,p}(z^p)\right).
\end{equation}

Since $B_{n+1}(1/2)=0$ for odd $n+1>1$, the series in \eqref{eq:presentation-local-germ} contains only odd powers of $Y$.

\subsection{Amice transforms and Coleman--Koblitz measures}
\label{subsec:amice-transforms}

Let \(\LA(\Zp,K)\) be the space of \(K\)-valued locally analytic
functions on \(\Zp\), and put
\[
 \Dla(\Zp,K)=\Hom_{K,\mathrm{cont}}(\LA(\Zp,K),K).
\]
For \(D\in\Dla(\Zp,K)\), set
\[
 \Ami_D(T)=D\bigl(x\mapsto(1+T)^x\bigr).
\]
Restriction to a compact open set \(U\) is denoted by \(\one_UD\), and
multiplication by a locally analytic function \(h\) by \(hD\).  For
\(a\in\Zp\), let \(\delta_a\) be the Dirac distribution
\(\delta_a(f)=f(a)\).  For \(u\in\Zp\), define multiplicative push-forward by
\[
 (u_*D)(f)=D(x\mapsto f(ux)).
\]
Its Amice transform satisfies
\begin{equation}\label{eq:Amice-pushforward}
 \Ami_{u_*D}(T)=\Ami_D((1+T)^u-1).
\end{equation}
Let
\[
 \Disc^-:=\{T\in\mathbf C_p:|T|_p<1\}
\]
be the open rigid unit disc, and let
\(\mathcal O_{\mathrm{rig}}(\Disc^-)\) denote its ring of $K$-valued
rigid analytic functions.

\begin{theorem}[Amice]\label{thm:exact-Amice-criteria}
Let \(F(T)=\sum_{n\ge0}b_nT^n\in K[[T]]\).
\begin{enumerate}[label=(\roman*)]
\item There is a unique locally analytic distribution \(D\) with
\(\Ami_D=F\) if and only if \(F\) converges on the open unit disc,
equivalently if \(|b_n|_pr^n\to0\) for every \(0<r<1\).
\item The distribution is a bounded \(K\)-valued measure if and only if
\(\inf_n v_p(b_n)>-\infty\).
\item It is an \(\mathcal O_K\)-valued measure if and only if
\(b_n\in\mathcal O_K\) for every \(n\).
\end{enumerate}
Moreover,
\begin{equation}\label{eq:Amice-Mahler-coefficients}
 b_n=D\!\left(\binom{x}{n}\right).
\end{equation}
\end{theorem}

This is the Amice transform theorem together with Mahler's orthonormal
basis criterion \cite{Amice,Robert,SchneiderTeitelbaum}.  Thus
\[
 \mathcal O_K[[T]]
 \subset \mathcal O_K[[T]][1/p]
 \subset \mathcal O_{\mathrm{rig}}(\Disc^-)
 \subset K[[T]]
\]
correspond respectively to integral measures, bounded \(K\)-valued
measures, locally analytic distributions, and formal germs.

\subsubsection[Coleman--Koblitz measures at prime-to-p roots]{Coleman--Koblitz measures at prime-to-$p$ roots}\label{subsec:koblitz-measures}

The local analytic input is the Coleman--Koblitz measure attached to each
root of unity; see \cite{Coleman,Koblitz} for the classical construction
and its polylogarithmic role.

\begin{definition}[Coleman--Koblitz measure]\label{def:koblitz-measure}
Let $z\ne1$ be a root of unity of order prime to $p$.  Define a distribution on compact open balls by
\begin{equation}\label{eq:koblitz-cylinder}
 \mu_z(a+p^m\Zp)
 =
 \frac{z^a}{1-z^{p^m}},
 \qquad
 0\le a<p^m.
\end{equation}
\end{definition}

\begin{theorem}\label{thm:koblitz-measure}
The values \eqref{eq:koblitz-cylinder} are compatible under refinement and define an $\mathcal O$-valued bounded measure on $\Zp$.  Its Amice transform is
\begin{equation}\label{eq:koblitz-amice}
 \Ami_{\mu_z}(T)
 =
 \frac1{1-z(1+T)}.
\end{equation}
Moreover,
\begin{equation}\label{eq:koblitz-pclass}
 \one_{p\Zp}\mu_z
 =
 p_*\mu_{z^p},
\end{equation}
where $p_*$ denotes push-forward under $x\mapsto px$.
\end{theorem}

\begin{proof}
Refinement of the ball $a+p^m\Zp$ gives
\begin{align*}
 \sum_{b=0}^{p-1}
 \mu_z(a+bp^m+p^{m+1}\Zp)
 &=
 \frac{z^a}{1-z^{p^{m+1}}}
 \sum_{b=0}^{p-1}z^{bp^m}\\
 &=
 \frac{z^a}{1-z^{p^m}}.
\end{align*}
Since the order of $z$ is prime to $p$, $1-z^{p^m}$ is a unit; hence the values are uniformly bounded and integral.

For $q=1+T$ with $|T|_p<1$, the Riemann sums are
\[
 \frac1{1-z^{p^m}}
 \sum_{a=0}^{p^m-1}(zq)^a
 =
 \frac{1-(zq)^{p^m}}
 {(1-zq)(1-z^{p^m})}.
\]
Although $z^{p^m}$ may be periodic, the exact estimate
\[
 \left|
 \frac{1-(zq)^{p^m}}{1-z^{p^m}}-1
 \right|_p
 =
 |1-q^{p^m}|_p
 \longrightarrow0
\]
holds uniformly for $|T|_p\le r<1$; here $z^{p^m}$ and
$1-z^{p^m}$ are units.  Thus the quotient tends uniformly on closed
subdiscs to $(1-zq)^{-1}$, proving \eqref{eq:koblitz-amice}.  Finally, on
a ball contained in $p\Zp$,
\[
 \mu_z(pa+p^{m+1}\Zp)
 =
 \frac{z^{pa}}{1-z^{p^{m+1}}}
 =
 \mu_{z^p}(a+p^m\Zp),
\]
which is \eqref{eq:koblitz-pclass}.  On a ball disjoint from $p\Zp$,
both sides of that restriction identity are zero.
\end{proof}

\subsection{Logarithmic distributions and character packets}\label{subsec:dirichlet-distribution}

For the same root $z$, define
\begin{equation}\label{eq:C-z-T}
 C_z(T)
 =
 \log\frac{1-z}{1-z(1+T)}.
\end{equation}
\begin{lemma}\label{lem:Cz-coefficients}
Put
\[
 u_z=\frac{z}{1-z}\in\mathcal O^\times.
\]
Then
\begin{equation}\label{eq:Cz-coefficient-expansion}
 C_z(T)
 =-\log(1-u_zT)
 =\sum_{n\ge1}\frac{u_z^n}{n}T^n.
\end{equation}
The series is rigid analytic on $\Disc^-$ and therefore is the Amice transform of a locally analytic distribution.  It is not a bounded measure on all of $\Zp$ in general: because $u_z$ is a unit, the coefficient valuations are
\[
 \valp\!\left(\frac{u_z^n}{n}\right)=-\valp(n),
\]
which are unbounded below along $n=p^r$.
\end{lemma}

\begin{proof}
The algebraic identity
\[
 \frac{1-z(1+T)}{1-z}=1-u_zT
\]
gives the logarithmic expansion.  For any $\lambda>0$,
\[
 -\valp(n)+\lambda n\longrightarrow+\infty,
\]
since $\valp(n)\le\log_p n$.  Theorem~\ref{thm:exact-Amice-criteria}(i) gives the distribution.  The failure of a uniform lower bound proves the final assertion by part (ii) of that theorem.
\end{proof}

\begin{definition}[Logarithmic Dirichlet distribution]\label{def:nu-z}
Let $\nu_z\in\Dla(\Zp,K)$ be the unique distribution with
\[
 \Ami_{\nu_z}(T)=C_z(T).
\]
For a finite presentation $\xi=\sum_z a_z[z]_{\mathrm{cyc}}$, set
\[
 \nu_\xi=\sum_z a_z\nu_z.
\]
\end{definition}

Let
\[
 \Theta=(1+T)\frac{d}{dT}.
\]
Multiplication of a distribution by the coordinate $x$ corresponds to applying $\Theta$ to its Amice transform.

\begin{theorem}\label{thm:nu-unit-measure}
One has
\begin{equation}\label{eq:xnu-mu}
 x\nu_z=\mu_z-\delta_0.
\end{equation}
Consequently,
\begin{equation}\label{eq:nu-unit-explicit}
 \nu_z^\times
 :=\one_{\Zp^\times}\nu_z
 =
 x^{-1}\mu_z^\times,
 \qquad
 \mu_z^\times:=\one_{\Zp^\times}\mu_z.
\end{equation}
In particular, $\nu_z^\times$ is an $\mathcal O$-valued bounded measure.  Hence $\nu_\xi^\times$ is a bounded $K$-valued measure for every finite presentation, and it is integral when all coefficients $a_z$ lie in $\mathcal O$.
\end{theorem}

\begin{proof}
Differentiate \eqref{eq:C-z-T}:
\[
 \Theta C_z(T)
 =
 \frac{z(1+T)}{1-z(1+T)}
 =
 \frac1{1-z(1+T)}-1.
\]
By Theorem~\ref{thm:koblitz-measure}, the right side is the Amice transform of $\mu_z-\delta_0$, which proves \eqref{eq:xnu-mu}.  Restriction to $\Zp^\times$ kills $\delta_0$, and multiplication by $x$ is invertible there.  Since $x^{-1}$ is an integral continuous function on $\Zp^\times$, \eqref{eq:nu-unit-explicit} is a bounded integral measure.
\end{proof}

\begin{proposition}\label{prop:nu-unit-positive-moments}
For every integer $n\ge1$,
\begin{equation}\label{eq:nu-unit-positive-moments}
 \int_{\Zp^\times}x^n\,d\nu_z(x)
 =
 \Li_{1-n,p}(z)
 -p^{n-1}\Li_{1-n,p}(z^p).
\end{equation}
For a presentation $\xi$, the corresponding unit moment is $S_n^{[p]}(\xi)$.
\end{proposition}

\begin{proof}
By \eqref{eq:nu-unit-explicit}, the left side is
\[
 \int_{\Zp^\times}x^{n-1}\,d\mu_z(x).
\]
For $n>1$, differentiation of \eqref{eq:koblitz-amice} gives the full moment $\Li_{1-n,p}(z)$, while \eqref{eq:koblitz-pclass} gives the $p\Zp$ contribution $p^{n-1}\Li_{1-n,p}(z^p)$.  When $n=1$, the corresponding full and zero-class masses are respectively $1/(1-z)$ and $1/(1-z^p)$; their difference is
\[
 \frac1{1-z}-\frac1{1-z^p}
 =
 \Li_{0,p}(z)-\Li_{0,p}(z^p).
\]
Thus the same formula holds for every $n\ge1$.
\end{proof}

\begin{example}\label{ex:koblitz-minus-one}
Because $p$ is odd, $(-1)^{p^m}=-1$ for every $m\ge0$.
Definition~\ref{def:koblitz-measure} therefore becomes the completely
explicit rule
\[
 \mu_{-1}(a+p^m\Zp)=\frac{(-1)^a}{2},
 \qquad
 \Ami_{\mu_{-1}}(T)=\frac1{2+T}.
\]
The logarithmic distribution has
\[
 C_{-1}(T)
 =\log\frac2{2+T}
 =-\log\left(1+\frac T2\right)
 =\sum_{n\ge1}\frac{(-1)^n}{n2^n}T^n.
\]
Its coefficient valuations are unbounded below along $n=p^r$, so it is
not a bounded measure on all of $\Zp$.  Nevertheless
Theorem~\ref{thm:nu-unit-measure} gives the bounded unit restriction
\[
 \nu_{-1}^{\times}=x^{-1}\one_{\Zp^\times}\mu_{-1}.
\]
Since $(-1)^p=-1$, its positive moments reduce to
\[
 \int_{\Zp^\times}x^n\,d\nu_{-1}(x)
 =(1-p^{n-1})\Li_{1-n,p}(-1).
\]
In particular, the first unit moment is zero.
\end{example}

\subsubsection{Character packets and the Dirichlet Euler factor}\label{subsec:local-character-packets}

Let $\chi$ be primitive and odd modulo $N$, with $p\nmid N$.  Retain
$\xi_{\chi,N}$ from \eqref{eq:character-packet}, and put
\[
 \nu_{\chi,N}=\nu_{\xi_{\chi,N}}.
\]

\begin{proposition}\label{prop:primitive-character-moments}
For $n\ge1$,
\begin{align}
 \int_{\Zp}x^n\,d\nu_{\chi,N}(x)
 &=
 G(\chi)L(1-n,\chi^{-1}),
 \label{eq:nu-primitive-full}
 \\
 \int_{\Zp^\times}x^n\,d\nu_{\chi,N}(x)
 &=
 G(\chi)
 \bigl(1-\chi^{-1}(p)p^{n-1}\bigr)
 L(1-n,\chi^{-1}).
 \label{eq:nu-primitive-unit}
\end{align}
\end{proposition}

\begin{proof}
First derive the finite Fourier identity.  For $|t|_p<1$, expand the
rational polylogarithms and apply the primitive Gauss identity
coefficientwise:
\[
 \sum_a\chi(a)\Li_{1-n,p}(t\zeta_N^a)
 =
 G(\chi)\sum_{r\ge1}\chi^{-1}(r)r^{n-1}t^r.
\]
For $n\ge1$, both sides extend as rational functions of $t$.  The right
side is regular at $t=1$: before applying
$(t\,d/dt)^{n-1}$ it is
\[
 \frac{\sum_{a=1}^{N}\chi^{-1}(a)t^a}{1-t^N},
\]
and the zero of the numerator at $t=1$ cancels the denominator because
$\chi$ is nontrivial.  Its value after continuation is
$L(1-n,\chi^{-1})$.  Setting $t=1$ therefore gives the finite Fourier
identity
\[
 \sum_a\chi(a)\Li_{1-n,p}(\zeta_N^a)
 =
 G(\chi)L(1-n,\chi^{-1}).
\]
Next compute the full moment of a single $\nu_z$.  The identity $x\nu_z=\mu_z-\delta_0$ gives, for $n\ge1$,
\begin{equation}\label{eq:nu-full-positive-moment-proof}
 \int_{\Zp}x^n\,d\nu_z(x)
 =
 \int_{\Zp}x^{n-1}\,d\mu_z(x)-\delta_0(x^{n-1}).
\end{equation}
If $n>1$, the last term is zero, and differentiation of
$\Ami_{\mu_z}(T)=(1-z(1+T))^{-1}$ gives
\[
 \int_{\Zp}x^{n-1}\,d\mu_z(x)=\Li_{1-n,p}(z).
\]
If $n=1$, the correction at the origin is essential:
\[
 \int_{\Zp}1\,d\mu_z-1
 =\frac1{1-z}-1
 =\frac z{1-z}
 =\Li_{0,p}(z).
\]
Thus \eqref{eq:nu-full-positive-moment-proof} equals
$\Li_{1-n,p}(z)$ for every $n\ge1$.  Summing with coefficients $\chi(a)$
and applying the finite Fourier identity proves
\eqref{eq:nu-primitive-full}.
Finally, multiplication by $p$ permutes the units modulo $N$, and the
substitution $b=ap$ gives
\[
 \sum_a\chi(a)\Li_{1-n,p}(\zeta_N^{ap})
 =
 \chi^{-1}(p)G(\chi)L(1-n,\chi^{-1}).
\]
Insert this and the first Fourier sum into
\eqref{eq:nu-unit-positive-moments}; this proves
\eqref{eq:nu-primitive-unit}.  The argument is restricted to $n\ge1$,
where $L(1-n,\chi^{-1})$ is an algebraic special value; in a
parity-mismatched case it is the trivial zero.  The weight-zero value of the bounded measure is defined by rigid analytic
continuation and is not identified with the complex number
$L(1,\chi^{-1})$.
\end{proof}

Let \(\chi_N\) be an odd character induced from a primitive odd
character \(\chi_f\) of conductor \(f>1\), write \(N=fg\), and assume
\(p\nmid N\).  Retain \(\xi_{\chi_N,N}\) from
\eqref{eq:character-packet}, and put
\[
 \nu_{\chi_N,N}:=\nu_{\xi_{\chi_N,N}}.
\]
For an integer \(d\), define the finite imprimitive Fourier correction by
\begin{equation}\label{eq:general-imprimitive-correction}
 A_{\chi_f,g}(d)
 =
 \sum_{\substack{u\mid g,\ u\mid d\\(g/u,f)=1}}
 u\,\mob(g/u)\chi_f(g/u)\chi_f^{-1}(d/u).
\end{equation}
For \(n\ge0\), define the corresponding Iwasawa moment factor by
\begin{equation}\label{eq:C-imprimitive-n}
 C^{\mathrm{Iw}}_{\chi_f,g}(n)
 =
 \sum_{\substack{u\mid g\\(g/u,f)=1}}
 \mob(g/u)\chi_f(g/u)u^n.
\end{equation}
Finally, put
\begin{equation}\label{eq:nu-imprimitive}
 \nu_{\chi_f,g}
 =
 \sum_{\substack{u\mid g\\(g/u,f)=1}}
 \mob(g/u)\chi_f(g/u)\,u_*\nu_{\chi_f,f}.
\end{equation}
Every divisor \(u\mid g\) is a \(p\)-adic unit, and no coprimality
assumption on \(f\) and \(g\) is imposed.

\begin{proposition}\label{prop:nu-imprimitive}
In $\Dla(\Zp,K)$ one has
\begin{equation}\label{eq:nu-imprimitive-exact}
 \nu_{\chi_N,N}=\nu_{\chi_f,g}.
\end{equation}
After restriction to units,
\begin{equation}\label{eq:nu-imprimitive-unit-exact}
 \nu_{\chi_N,N}^{\times}=\nu_{\chi_f,g}^{\times}.
\end{equation}
Consequently, for every $n\ge1$,
\begin{equation}\label{eq:nu-imprimitive-moment}
 \int_{\Zp}x^n\,d\nu_{\chi_N,N}(x)
 =
 G(\chi_f)C^{\mathrm{Iw}}_{\chi_f,g}(n)L(1-n,\chi_f^{-1}).
\end{equation}
\end{proposition}

\begin{proof}
Write $q=1+T$ and let $\Theta_q=q\,d/dq$.  Since every nontrivial character has zero total sum,
\begin{align}
 \Theta_q\Ami_{\nu_{\chi_N,N}}(T)
 &=\sum_{a\in(\Z/N\Z)^\times}
   \frac{\chi_N(a)}{1-\zeta_N^a q}.
 \label{eq:imprimitive-amice-derivative-left}
\end{align}
This derivative is a rational function of $q$ regular near $q=1$.  Expanding it at $q=0$ and applying Lemma~\ref{lem:ramanujan} gives
\begin{equation}\label{eq:imprimitive-amice-left}
 \Theta_q\Ami_{\nu_{\chi_N,N}}(T)
 =G(\chi_f)\sum_{r\ge1}A_{\chi_f,g}(r)q^r.
\end{equation}

By \eqref{eq:Amice-pushforward},
\[
 \Ami_{u_*\nu}(T)=\Ami_\nu(q^u-1).
\]
After applying $\Theta_q$, the primitive Gauss identity gives the rational-function expansion
\[
 \Theta_q\Ami_{u_*\nu_{\chi_f,f}}(T)
 =uG(\chi_f)\sum_{t\ge1}\chi_f^{-1}(t)q^{ut}.
\]
Therefore the coefficient of $q^r$ in
$\Theta_q\Ami_{\nu_{\chi_f,g}}(T)$ is
\[
 G(\chi_f)
 \sum_{\substack{u\mid g,\ u\mid r\\(g/u,f)=1}}
 u\,\mob(g/u)\chi_f(g/u)\chi_f^{-1}(r/u)
 =G(\chi_f)A_{\chi_f,g}(r).
\]
The derivatives therefore agree as rational functions.  Both Amice transforms vanish at $T=0$, so injectivity of the Amice transform gives \eqref{eq:nu-imprimitive-exact}.  Since every $u$ is a unit, restriction to $\Zp^\times$ commutes with $u_*$ and gives \eqref{eq:nu-imprimitive-unit-exact}.  Taking the $n$-th moment and applying Proposition~\ref{prop:primitive-character-moments} yields \eqref{eq:nu-imprimitive-moment}.
\end{proof}

\section{Cyclotomic Iwasawa pseudomeasures and local factorization}\label{sec:pseudomeasures}

\subsection{The bounded-denominator Iwasawa algebra and its weight components}\label{subsec:iwasawa-algebra}

Write
\[
 G=\Zp^\times=\Delta\times\Gamma,
 \qquad
 \Delta=\mu_{p-1},
 \qquad
 \Gamma=1+p\Zp.
\]
The integral Iwasawa algebra
\begin{equation}\label{eq:integral-Iwasawa-definition}
 \Iw(G)=\mathcal O[[G]]
\end{equation}
is the convolution algebra of bounded \(\mathcal O\)-valued measures on
\(G\).

\begin{definition}[Bounded-denominator Iwasawa algebra]\label{def:bounded-denominator-Iwasawa}
Its bounded-denominator extension is
\begin{equation}\label{eq:bounded-denominator-Iwasawa}
 \IwK(G)=\Iw(G)[1/p],
\end{equation}
and, in one variable,
\begin{equation}\label{eq:bounded-denominator-series}
 K[[X]]_{\bd}
 =\mathcal O[[X]][1/p]
 =\left\{\sum_{n\ge0}a_nX^n:\inf_n\valp(a_n)>-\infty\right\}.
\end{equation}
\end{definition}
Thus one power of \(p\) clears all denominators; this is stronger than
membership in \(K[[X]]\).

For \(u\in G\), let \(\delta_u\in\Iw(G)\) be the Dirac measure defined
above.  For every odd integer \(c>1\) prime to \(p\), put
\begin{equation}\label{eq:r-c}
 r_c=\delta_1-c\delta_c\in\Iw(G).
\end{equation}
Let \(\Sigma\) be the multiplicative set generated by the elements
\(r_c\).  The elements of \(\Sigma^{-1}\IwK(G)\) are called unit
pseudomeasures.

Choose a topological generator \(\gamma\) of \(\Gamma\).  After extending
scalars to contain the values of \(\teich\), the idempotents
\begin{equation}\label{eq:Delta-idempotents}
 e_j=\frac1{p-1}\sum_{\delta\in\Delta}
 \teich(\delta)^{-j}\delta_\delta,
 \qquad j\in\Z/(p-1)\Z,
\end{equation}
give the branch decomposition
\begin{equation}\label{eq:iwasawa-branch-decomposition}
 \IwK(G)\overset{\sim}{\longrightarrow}
 \prod_{j\in\Z/(p-1)\Z}K[[X]]_{\bd},
 \qquad
 \rho\longmapsto(I_{\rho,j}(X))_j.
\end{equation}
Put
\[
 \ell_\gamma(x)=\frac{\log_p\ang x}{\log_p\gamma}\in\Zp.
\]
For \(r_c\), the \(j\)-th Iwasawa-coordinate factor is
\begin{equation}\label{eq:r-c-Iwasawa-coordinate}
 I_{r_c,j}(X)
 =1-c\teich(c)^j(1+X)^{\ell_\gamma(c)}.
\end{equation}
Its weight-coordinate form is
\begin{align}
 r_{c,j}(s)
 &:=I_{r_c,j}(\gamma^s-1)\notag\\
 &=1-c\teich(c)^j\ang c^s
 =1-\teich(c)^{j+1}\ang c^{s+1}.
 \label{eq:r-c-symbol}
\end{align}
Every \(I_{r_c,j}\) is a nonzero power series, so every \(r_c\) is a
non-zero-divisor in the product of branch domains.

The Mellin branch of \(\rho\in\IwK(G)\) is
\begin{equation}\label{eq:branch-Mellin-Iwasawa-coordinate}
 \mathcal M_{\rho,j}(s)
 =I_{\rho,j}(\gamma^s-1)
 =\int_G\teich(x)^j\ang x^s\,d\rho(x).
\end{equation}
Arithmetic integers in a fixed residue class modulo \(p-1\) are dense in
\(\Zp\), so these values determine the branch.

\begin{definition}[Weight components]
\label{def:pseudomeasure-Mellin-branch}
For a locally analytic distribution \(D\) on \(G\), define
\begin{equation}\label{eq:distribution-Iwasawa-branch}
 I_{D,j}(X)
 =D\!\left(x\mapsto
 \teich(x)^j(1+X)^{\ell_\gamma(x)}\right).
\end{equation}
If \(D\) is bounded, then \(I_{D,j}\in K[[X]]_{\bd}\).  For
\(Q=A/r\in\Sigma^{-1}\IwK(G)\), define
\[
 I_{Q,j}(X)=\frac{I_{A,j}(X)}{I_{r,j}(X)}
\]
in the fraction field of the branch domain, and set
\begin{equation}\label{eq:pseudomeasure-Mellin-definition}
 \mathcal M_{Q,j}(s)=I_{Q,j}(\gamma^s-1)
\end{equation}
meromorphically.  This is independent of the representation of \(Q\): if
\(A/r=A'/r'\), then \(Ar'=A'r\), and the branch decomposition gives
\(I_{A,j}I_{r',j}=I_{A',j}I_{r,j}\).
\end{definition}

\begin{remark}\label{rem:bounded-denominator-not-all-series}
A quotient may define an element of \(K[[X]]\) while having unbounded
coefficient denominators.  Every removability and divisibility argument
below is therefore carried out in \(K[[X]]_{\bd}\), not merely in
\(K[[X]]\).
\end{remark}

\subsection{Half-shift regularization and multiplicative convolution}\label{subsec:half-regularizer}

Define
\begin{equation}\label{eq:Lambda-c}
 \Lambda_{1/2,c}(T)
 =
 \frac{(1+T)^{1/2}}{T}
 -c\frac{(1+T)^{c/2}}{(1+T)^c-1}.
\end{equation}
The fact that this series is integral, rather than merely rigid analytic, is essential for the measure construction.

\begin{theorem}[Integral half-shift measure]\label{thm:integral-half-measure}
For every odd $c>1$ with $p\nmid c$,
\begin{equation}\label{eq:Lambda-integral}
 \Lambda_{1/2,c}(T)\in\mathcal O[[T]].
\end{equation}
It is the Amice transform of an $\mathcal O$-valued bounded measure $\lambda_{1/2,c}$ on $\Zp$.  Its moments are
\begin{equation}\label{eq:lambda-total-moments}
 \int_{\Zp}x^n\,d\lambda_{1/2,c}(x)
 =
 (1-c^{n+1})\frac{B_{n+1}(1/2)}{n+1},
 \qquad n\ge0.
\end{equation}
\end{theorem}

\begin{proof}
Put $Z=(1+T)^{1/2}\in\mathcal O[[T]]$.  Since $c$ is odd,
\[
 \Lambda_{1/2,c}
 =
 \frac{Z}{(Z-1)(Z+1)}
 -
 \frac{cZ^c}{(Z-1)S_c(Z)(Z^c+1)},
\]
where
\[
 S_c(Z)=1+Z+\cdots+Z^{c-1}.
\]
Thus
\begin{equation}\label{eq:Lambda-Z-rational}
 \Lambda_{1/2,c}
 =
 \frac{Z S_c(Z)(Z^c+1)-cZ^c(Z+1)}
 {(Z-1)(Z+1)S_c(Z)(Z^c+1)}.
\end{equation}
The numerator vanishes at $Z=1$, hence is divisible by $Z-1$ in $\Z[Z]$.  After cancellation, every denominator factor has unit value at $Z=1$: namely $2$, $c$, and $2$.  Hence \eqref{eq:Lambda-integral} holds.

For the moments, start from the Bernoulli generating function
\[
 \frac{Ye^{xY}}{e^Y-1}
 =
 \sum_{m\ge0}B_m(x)\frac{Y^m}{m!}.
\]
After division by $Y$ and removal of its polar term, this becomes
\begin{equation}\label{eq:Bernoulli-regular-part}
 \frac{e^{xY}}{e^Y-1}-\frac1Y
 =
 \sum_{n\ge0}
 \frac{B_{n+1}(x)}{n+1}\frac{Y^n}{n!}.
\end{equation}
Set $T=e^Y-1$.  Applying \eqref{eq:Bernoulli-regular-part} first to
$x=1/2$ and $Y$, and then to $x=1/2$ and $cY$, gives
\begin{align*}
 \Lambda_{1/2,c}(e^Y-1)
 &=
 \left(\frac{e^{Y/2}}{e^Y-1}-\frac1Y\right)
 -c\left(\frac{e^{cY/2}}{e^{cY}-1}-\frac1{cY}\right)\\
 &=
 \sum_{n\ge0}
 (1-c^{n+1})\frac{B_{n+1}(1/2)}{n+1}
 \frac{Y^n}{n!}.
\end{align*}
Thus the polar terms cancel before coefficients are compared.
Comparison with the exponential moments of the Amice measure gives \eqref{eq:lambda-total-moments}.
\end{proof}

\begin{proposition}\label{prop:lambda-zero-class}
The restriction of $\lambda_{1/2,c}$ to $p\Zp$ is the push-forward $p_*\lambda_{1/2,c}$.  Equivalently,
\begin{equation}\label{eq:lambda-zero-amice}
 \frac1p\sum_{\varepsilon\in\mu_p}
 \Lambda_{1/2,c}\bigl(\varepsilon(1+T)-1\bigr)
 =
 \Lambda_{1/2,c}\bigl((1+T)^p-1\bigr).
\end{equation}
Consequently,
\begin{equation}\label{eq:lambda-unit-moments}
 \int_Gx^n\,d\lambda_{1/2,c}(x)
 =
 (1-p^n)(1-c^{n+1})
 \frac{B_{n+1}(1/2)}{n+1}.
\end{equation}
\end{proposition}

\begin{proof}
We may extend scalars to a finite field containing $\mu_p$: the natural map
on bounded measures is injective, so an identity proved after this extension
descends to the original coefficient field.
Write $q=1+T$ and choose the unique square root on $\mu_p$ given by the inverse of the squaring map.  If $t=q^{1/2}$, then
\[
 \frac1p\sum_{\varepsilon\in\mu_p}
 \frac{(\varepsilon q)^{1/2}}{\varepsilon q-1}
 =
 \frac1p\sum_{\eta\in\mu_p}
 \frac{\eta t}{\eta^2t^2-1}.
\]
Using
\[
 \frac{x}{x^2-1}
 =\frac12\left(\frac1{x-1}+\frac1{x+1}\right)
\]
and the finite partial-fraction identities
\[
 \frac1p\sum_{\eta\in\mu_p}\frac1{\eta t-1}
 =\frac1{t^p-1},
 \qquad
 \frac1p\sum_{\eta\in\mu_p}\frac1{\eta t+1}
 =\frac1{1+t^p},
\]
one obtains
\[
 \frac{q^{p/2}}{q^p-1}.
\]
For the second summand of $\Lambda_{1/2,c}$, the chosen square root gives
\[
 (\varepsilon q)^{c/2}=(\eta t)^c,
 \qquad
 (\varepsilon q)^c-1=(\eta t)^{2c}-1.
\]
Because $(c,p)=1$, the map $\eta\mapsto\eta^c$ permutes $\mu_p$; repeating the two partial-fraction sums with $t$ replaced by $t^c$ yields
\[
 \frac1p\sum_{\varepsilon\in\mu_p}
 \frac{(\varepsilon q)^{c/2}}{(\varepsilon q)^c-1}
 =\frac{q^{cp/2}}{q^{cp}-1}.
\]
Consequently the average of the complete regularizer is
\[
 \frac{q^{p/2}}{q^p-1}
 -c\frac{q^{cp/2}}{q^{cp}-1}
 =\Lambda_{1/2,c}(q^p-1),
\]
which is \eqref{eq:lambda-zero-amice}.  The left side is the Amice transform of the restriction to $p\Zp$, whereas the right side is the Amice transform of $p_*\lambda_{1/2,c}$; injectivity gives the measure identity.  Finally,
\[
 \int_{p\Zp}x^n\,d\lambda_{1/2,c}(x)
 =\int_{\Zp}(px)^n\,d\lambda_{1/2,c}(x)
 =p^n\int_{\Zp}x^n\,d\lambda_{1/2,c}(x).
\]
Subtract this from the total moment \eqref{eq:lambda-total-moments} to obtain \eqref{eq:lambda-unit-moments}.
\end{proof}

Write
\[
 \lambda_{1/2,c}^{\times}
 =\one_G\lambda_{1/2,c}\in\Iw(G).
\]

\subsubsection{Multiplicative convolution and the regularization identity}\label{subsec:regularized-convolution}

For distributions $D_1,D_2$ for which product push-forward is defined,
set
\begin{equation}\label{eq:multiplicative-convolution}
 (D_1\starx D_2)(f)
 =
 (D_1\otimes D_2)\bigl((x,y)\mapsto f(xy)\bigr).
\end{equation}
This is well defined when one factor is bounded and the other locally
analytic.  Its monomial moments multiply:
\begin{equation}\label{eq:convolution-moments}
 \int z^n\,d(D_1\starx D_2)(z)
 =
 \left(\int x^n\,dD_1(x)\right)
 \left(\int y^n\,dD_2(y)\right).
\end{equation}
Moreover,
\begin{equation}\label{eq:unit-convolution-restriction}
 \one_G(D_1\starx D_2)
 =
 (\one_GD_1)\starx(\one_GD_2).
\end{equation}
Consequently the two $p$-depletion factors in a character packet arise
separately from the half-shift and logarithmic measures.

Define the formal regularization operator on a series $F(T)$ by
\begin{equation}\label{eq:formal-Rc}
 (\operatorname{Reg}^{\mathrm{form}}_cF)(T)
 =
 F(T)-cF\bigl((1+T)^c-1\bigr).
\end{equation}
The operator on the left is defined coefficientwise on every formal series.
By \eqref{eq:Amice-pushforward}, whenever a distribution with Amice
transform $F$ exists, $\operatorname{Reg}^{\mathrm{form}}_cF$ is the
Amice-coordinate form of
convolution by $r_c=\delta_1-c\delta_c$.

For a finite presentation \(\xi\), define the regularized unit
measure
\begin{equation}\label{eq:regularized-unit-measure}
 N_{\xi,c}
 :=
 \lambda_{1/2,c}^{\times}\starx\nu_\xi^{\times}
 \in\IwK(G).
\end{equation}
If the coefficients of \(\xi\) are integral, then \(N_{\xi,c}\in\Iw(G)\).

\begin{theorem}[Regularized convolution]\label{thm:regularized-convolution}
For every finite cyclotomic presentation $\xi$,
\begin{equation}\label{eq:formal-regularized-convolution}
 \operatorname{Reg}^{\mathrm{form}}_cF_{\xi,p}
 =
 \Ami_{\lambda_{1/2,c}\starx\nu_\xi}.
\end{equation}
The right side is an actual locally analytic distribution on $\Zp$, and
its restriction to $G$ is $N_{\xi,c}$.
\end{theorem}

\begin{proof}
Substitute $T=e^Y-1$.  The left side becomes
\begin{align*}
 &F_{\xi,p}(e^Y-1)-cF_{\xi,p}(e^{cY}-1)\\
 &\quad=
 \sum_{n\ge1}
 (1-c^{n+1})\frac{B_{n+1}(1/2)}{n+1}
 S_n(\xi)\frac{Y^n}{n!}.
\end{align*}
By Theorem~\ref{thm:integral-half-measure}, the first factor in each
coefficient is the $n$-th moment of $\lambda_{1/2,c}$.  By
Definition~\ref{def:nu-z}, $S_n(\xi)$ is the $n$-th moment of $\nu_\xi$.
Equation \eqref{eq:convolution-moments} gives the same exponential series
for the Amice transform of the convolution.  The mass is zero on both
sides.  Formal substitution $T=e^Y-1$ is invertible, proving
\eqref{eq:formal-regularized-convolution}.

The whole-space convolution is locally analytic because one factor is
bounded and the other is locally analytic.  Equation
\eqref{eq:unit-convolution-restriction},
Theorem~\ref{thm:integral-half-measure}, and
Theorem~\ref{thm:nu-unit-measure} show that the unit restriction is
bounded, with the asserted integrality.
\end{proof}

\begin{corollary}\label{cor:regularized-unit-moments}
For $n\ge1$,
\begin{equation}\label{eq:regularized-unit-moments}
 \int_Gx^n\,dN_{\xi,c}(x)
 =
 (1-p^n)(1-c^{n+1})
 \frac{B_{n+1}(1/2)}{n+1}
 S_n^{[p]}(\xi).
\end{equation}
\end{corollary}

\begin{proof}
Multiply the two unit moment formulas \eqref{eq:lambda-unit-moments} and \eqref{eq:Snp-depleted} under multiplicative convolution.
\end{proof}

\subsection{The cyclotomic Iwasawa pseudomeasure and its exceptional branch}\label{subsec:cyclotomic-Iwasawa-pseudomeasure}

\begin{lemma}\label{lem:unit-moment-rigidity}
Two bounded \(K\)-valued measures on \(G\) are equal if all of their
monomial moments are equal.
\end{lemma}

\begin{proof}
The Stirling identities
\[
 x^n=\sum_{r=0}^nS(n,r)r!\binom xr,
 \qquad
 \binom xn=\frac1{n!}\sum_{r=0}^ns(n,r)x^r
\]
hold, where \(S(n,r)\) and \(s(n,r)\) are the Stirling numbers of the
second and first kind, respectively.  These formulas give triangular changes of
basis between monomials and the binomial
polynomials.  By Mahler's theorem \cite{Robert}, polynomials are dense in
\(C(\Zp,K)\).  Since \(G\) is clopen in \(\Zp\), every continuous function
on \(G\) extends by zero to \(\Zp\); hence restrictions of polynomials are
dense in \(C(G,K)\).  Equality of all monomial moments therefore gives
equality on a dense subspace of \(C(G,K)\), and boundedness extends the
equality to every continuous function.
\end{proof}

\begin{theorem}[Cross relation]\label{thm:cross-relation}
For any two admissible auxiliary integers $c,d$,
\begin{equation}\label{eq:cross-relation}
 r_d\starx N_{\xi,c}
 =
 r_c\starx N_{\xi,d}
 \qquad\text{in }\IwK(G).
\end{equation}
\end{theorem}

\begin{proof}
The $n$-th moment of convolution by
$r_d=\delta_1-d\delta_d$ is multiplied by $1-d^{n+1}$.  Thus
Corollary~\ref{cor:regularized-unit-moments} shows that both sides have the
moment
\[
 (1-d^{n+1})(1-c^{n+1})(1-p^n)
 \frac{B_{n+1}(1/2)}{n+1}S_n^{[p]}(\xi)
\]
for every $n\ge1$.  Both masses are zero, so the moments also agree at
$n=0$.  Lemma~\ref{lem:unit-moment-rigidity} now gives the equality of
bounded measures.
\end{proof}

\begin{definition}[Cyclotomic Iwasawa pseudomeasure]\label{def:cyclotomic-Iwasawa-pseudomeasure}
Define
\begin{equation}\label{eq:cyclotomic-Iwasawa-pseudomeasure}
 \IwPs_{\xi,p}
 =
 r_c^{-1}\starx N_{\xi,c}
 \in \Sigma^{-1}\IwK(G).
\end{equation}
Theorem~\ref{thm:cross-relation} shows that this element is independent of $c$.
\end{definition}

For a primitive packet and for a packet induced to modulus \(N\), write
\[
 \IwPs_{\chi,N,p}:=\IwPs_{\xi_{\chi,N},p},
 \qquad
 \IwPs_{\chi_N,N,p}:=\IwPs_{\xi_{\chi_N,N},p}.
\]
\begin{definition}[Half-shift zeta pseudomeasure]\label{def:half-zeta-pseudomeasure}
Define
\begin{equation}\label{eq:half-zeta-pseudomeasure}
 \Zhalf
 =
 r_c^{-1}\starx\lambda_{1/2,c}^{\times}
 \in \Sigma^{-1}\IwK(G).
\end{equation}
This element is independent of $c$.  Indeed, for two admissible
regularizers $c,c'$ and every $n\ge0$, the $n$-th moments of both
$r_{c'}\starx\lambda_{1/2,c}^{\times}$ and
$r_c\starx\lambda_{1/2,c'}^{\times}$ equal
\[
 (1-c'^{\,n+1})(1-c^{n+1})(1-p^n)
 \frac{B_{n+1}(1/2)}{n+1}.
\]
Both convolutions are bounded measures on $G$, so
Lemma~\ref{lem:unit-moment-rigidity} makes them equal.  Cross multiplication
in the localization then identifies the two quotients.
\end{definition}

\begin{theorem}[Canonical factorization]\label{thm:Iwasawa-pseudomeasure-factorization}
For every finite cyclotomic presentation,
\begin{equation}\label{eq:Iwasawa-pseudomeasure-factorization}
 \IwPs_{\xi,p}
 =
 \Zhalf\starx\nu_\xi^{\times}.
\end{equation}
This is an equality in $\Sigma^{-1}\IwK(G)$ and is linear in $\xi$.
\end{theorem}

\begin{proof}
By \eqref{eq:regularized-unit-measure},
\[
r_c^{-1}\starx N_{\xi,c}
 =
 \bigl(r_c^{-1}\starx\lambda_{1/2,c}^{\times}\bigr)
 \starx\nu_\xi^{\times}.
\]
The left-hand side is $\IwPs_{\xi,p}$, while the parenthesized factor on
the right-hand side is $\Zhalf$.
\end{proof}

The pseudomeasure $\IwPs_{\xi,p}$ realizes the fixed formal logarithmic
germ on the unit group: Theorem~\ref{thm:regularized-convolution}
identifies its bounded numerator with the restriction to $G$ of the
regularized formal germ.

\subsubsection{The unique exceptional weight}\label{subsec:unique-exceptional-weight}

Choose an odd positive integer $c_0>1$ whose residue class generates $(\Z/p^2\Z)^\times$.  Such a representative exists because adding the odd number $p^2$ changes parity without changing the residue class.  Then $\teich(c_0)$ generates $\Delta$ and $\ang{c_0}$ topologically generates $\Gamma$.

To prove independence of the auxiliary integer, we localize by all
\(r_c\).  For branchwise analysis, the cross relation gives the
representative
\begin{equation}\label{eq:one-regularizer-representative}
 \IwPs_{\xi,p}
 =
 r_{c_0}^{-1}\starx N_{\xi,c_0}.
\end{equation}
Thus \(\IwPs_{\xi,p}\) already lies in the image of
\(\IwK(G)[r_{c_0}^{-1}]\to \Sigma^{-1}\IwK(G)\), and it suffices to analyze
the divisor of \(r_{c_0}\).

\begin{theorem}\label{thm:unique-exceptional-branch}
For $j\ne-1\pmod{p-1}$, the element $r_{c_0}$ is a unit on the $j$-th Iwasawa branch.  On the branch $j=-1$,
\begin{equation}\label{eq:exceptional-denominator}
 r_{c_0,-1}(s)
 =
 1-\ang{c_0}^{s+1}
\end{equation}
has exactly one zero on $s\in\Zp$, namely $s=-1$, and the zero is simple.
Therefore the branches with $j\ne-1$ are bounded measures, while the
branch $j=-1$ can have only a simple pole, at the point $s=-1$.
\end{theorem}

\begin{proof}
If $j\ne-1$, then reduction of \eqref{eq:r-c-symbol} modulo the maximal ideal gives
\[
 1-\teich(c_0)^{j+1}\ne0,
\]
so the branch power series is a unit.  If $j=-1$, a zero of \eqref{eq:exceptional-denominator} satisfies $\ang{c_0}^{s+1}=1$.  A topological generator of the torsion-free group $\Gamma$ has no nontrivial $\Zp$-power equal to $1$, hence $s=-1$.  Its derivative there is
\[
 -\log_p\ang{c_0}\ne0,
\]
so the zero is simple.  The assertion for $\IwPs_{\xi,p}$ follows from Definition~\ref{def:cyclotomic-Iwasawa-pseudomeasure}, since its numerator is a bounded measure.
\end{proof}

\begin{example}
\label{ex:exceptional-branch-p-five}
Take $p=5$ and $c_0=3$, whose class generates
$(\Z/25\Z)^\times$.  The exceptional index is
$j=3\equiv-1\pmod4$, and its symbol is
\[
 r_{3,3}(s)=1-\ang3^{\,s+1},
\]
with one simple zero at $s=-1$.  For $j=0,1,2$, the reduction of the
constant term is $1-\teich(3)^{j+1}\ne0$ in $\F_5$, so those three branch
symbols are units.
\end{example}

\subsection{The half-zeta factor and negative moments}\label{subsec:zeta-identification}

Let
\[
 \Zclass\in \Sigma^{-1}\IwK(G)
\]
denote the classical Iwasawa zeta pseudomeasure.  In the
Mazur--Iwasawa construction, every admissible $c$ regularizes it to a
bounded integral measure:
\begin{equation}\label{eq:classical-zeta-regularization}
 r_c\starx\Zclass\in\Iw(G),
 \qquad
 \int_Gx^n\,d(r_c\starx\Zclass)(x)
 =
 -(1-c^{n+1})(1-p^n)\frac{B_{n+1}}{n+1}
\end{equation}
for $n\ge0$; see \cite{Iwasawa,Washington}.  For \(j\in\Z/(p-1)\Z\), define
\[
 \zeta_{p,j}(u):=\mathcal M_{\Zclass,j}(-u),
\]
and write \(\zeta_p:=\zeta_{p,-1}\) for the exceptional branch.  Equivalently,
its branch normalization is
\begin{equation}\label{eq:zeta-p-interpolation}
 \mathcal M_{\Zclass,j}(n)
 =
 \zeta_{p,j}(-n)
 =
 -(1-p^n)\frac{B_{n+1}}{n+1}
\end{equation}
for all $n\ge0$ with $n\equiv j\pmod{p-1}$.  On the exceptional component it has the standard simple pole corresponding to $\zeta_p(u)$ at $u=1$, normalized by
\begin{equation}\label{eq:zeta-p-residue-normalization}
 \Res_{u=1}\zeta_p(u)=1-p^{-1}.
\end{equation}

\begin{theorem}[Identification with the Iwasawa zeta pseudomeasure]\label{thm:half-zeta-identification}
In the localized Iwasawa algebra,
\begin{equation}\label{eq:half-zeta-identification}
 \Zhalf
 =
 (\delta_1-\delta_{2^{-1}})\starx\Zclass.
\end{equation}
Equivalently, on the $j$-th branch,
\begin{equation}\label{eq:half-zeta-branch}
 Z_{1/2,j}(s)
 :=\mathcal M_{\Zhalf,j}(s)
 =
 \bigl(1-\teich(2)^{-j}\ang2^{-s}\bigr)
 \zeta_{p,j}(-s).
\end{equation}
At every arithmetic point $n\ge0$, $n\equiv j$, one has
\begin{equation}\label{eq:half-zeta-interpolation}
 Z_{1/2,j}(n)
 =
 (1-p^n)\frac{B_{n+1}(1/2)}{n+1}.
\end{equation}
Moreover,
\begin{equation}\label{eq:half-zeta-residue}
 \Res_{s=-1}Z_{1/2,-1}(s)
 =1-p^{-1}.
\end{equation}
\end{theorem}

\begin{proof}
At an arithmetic point $n\equiv j$, the Mellin value of $r_c$ is $1-c^{n+1}$.  Dividing \eqref{eq:lambda-unit-moments} by this factor gives \eqref{eq:half-zeta-interpolation}.  Since
\[
 B_{n+1}(1/2)=(2^{-n}-1)B_{n+1},
\]
the right side becomes
\[
 (1-2^{-n})\zeta_{p,j}(-n).
\]
The Mellin transform of $\delta_1-\delta_{2^{-1}}$ at the same arithmetic
point is $1-2^{-n}$.  Fix an admissible $c$.  Multiplication by $r_c$
sends the left side of \eqref{eq:half-zeta-identification} to
$\lambda_{1/2,c}^{\times}$.  By
\eqref{eq:classical-zeta-regularization}, it also sends the proposed right
side to a bounded measure.  The Bernoulli identity above shows that these
two bounded measures have the same monomial moments for every $n\ge0$.
Lemma~\ref{lem:unit-moment-rigidity} therefore makes them equal.  Since
$r_c$ is a non-zero-divisor, cancellation in the
localization proves \eqref{eq:half-zeta-identification}.

For the residue, set $u=-s$.  Near $s=-1$, the factor in front of $\zeta_p(-s)$ tends to
\[
 1-2=-1,
\]
while $u-1=-(s+1)$.  The two minus signs cancel, and \eqref{eq:zeta-p-residue-normalization} gives \eqref{eq:half-zeta-residue}.
\end{proof}

\subsubsection{Negative weights and Coleman polylogarithms}\label{subsec:negative-weights}

For $r\ge1$, define the $p$-depleted Coleman polylogarithm
\begin{equation}\label{eq:p-depleted-polylog}
 \Li_{r,p}^{[p]}(z)
 =
 \Li_{r,p}(z)-p^{-r}\Li_{r,p}(z^p).
\end{equation}

\begin{theorem}[Coleman--Koblitz negative moments]\label{thm:koblitz-negative-moments}
For every prime-to-$p$ root $z\ne1$ and every integer $r\ge1$,
\begin{equation}\label{eq:koblitz-negative-moments}
 \int_Gx^{-r}\,d\mu_z(x)
 =
 \Li_{r,p}^{[p]}(z).
\end{equation}
Consequently, for $r\ge0$,
\begin{equation}\label{eq:nu-negative-moments}
 \mathcal M_{\nu_z^{\times},-r}(-r)
 =
 \Li_{r+1,p}^{[p]}(z).
\end{equation}
\end{theorem}

\begin{proof}
For $|z|_p<1$, the unit part of the convergent series is
\[
 \sum_{\substack{n\ge1\\p\nmid n}}\frac{z^n}{n^r}
 =
 \Li_r(z)-p^{-r}\Li_r(z^p).
\]
The left side is the Mellin integral of the Koblitz measure, as follows from
its finite Riemann sums.  For the boundary points used here, the
Coleman--Koblitz formula recalled in \cite[Prop.~2.1]{Besser} gives directly
\[
 \int_{\Zp^\times}x^{-r}\,d\mu_z(x)
 =
 \Li_{r,p}(z)-p^{-r}\Li_{r,p}(z^p)
\]
for $|z|_p=|1-z|_p=1$.  A nontrivial root of unity of order prime to
$p$ lies in this domain.  Thus the formula is evaluated on its own residue
disc.  Equation \eqref{eq:nu-negative-moments} follows from
$\nu_z^\times=x^{-1}\mu_z^\times$.
\end{proof}

\begin{remark}
\label{rem:cyclosyntomic-moment-bridge}
For \(d=p\), the coefficient formula for the first \(q\)-polylogarithm in
\cite[Thm.~C]{BouisGazda} formally specializes at \(q=1\) to
\begin{equation}\label{eq:BG-q1-weight-one}
 \left.\Li_1^{(p)}(z)_q\right|_{q=1}
 =\sum_{\substack{k\ge1\\p\nmid k}}\frac{z^k}{k}
 =\Li_1(z)-p^{-1}\Li_1(z^p).
\end{equation}
For a root of unity of order prime to \(p\), Coleman continuation gives the
corresponding \(p\)-adic value.  Theorem~\ref{thm:koblitz-negative-moments}
then yields
\begin{equation}\label{eq:weight-one-two-moment-bridge}
 \int_Gx^{-1}\,d\mu_z(x)=\Li_{1,p}^{[p]}(z),
 \qquad
 \int_Gx^{-2}\,d\mu_z(x)=\Li_{2,p}^{[p]}(z).
\end{equation}
Since \(\nu_z^\times=x^{-1}\mu_z^\times\), the exceptional value considered here is
\[
 \mathcal M_{\nu_z^\times,-1}(-1)
 =\int_Gx^{-1}\,d\nu_z(x)
 =\Li_{2,p}^{[p]}(z).
\]
Thus the weight-one and weight-two depleted Coleman values occur as adjacent
negative moments of the same measure.
\end{remark}

\begin{definition}[Depleted Coleman regulator]\label{def:depleted-regulator}
For a cyclotomic presentation $\xi=\sum_z a_z[z]_{\mathrm{cyc}}$, define
\begin{equation}\label{eq:depleted-regulator}
 \Regdep(\xi)
 =
 \sum_z a_z
 \left(D_p(z)-p^{-2}D_p(z^p)\right).
\end{equation}
Since $\log_pz=0$ for every root in the support, this is equivalently the sum of $\Li_{2,p}^{[p]}(z)$.
\end{definition}

\begin{corollary}\label{cor:character-depleted-regulator}
For a primitive odd packet,
\begin{equation}\label{eq:character-depleted-regulator}
 \mathcal M_{\nu_{\chi,N}^{\times},-1}(-1)
 =
 \Regdep(\xi_{\chi,N})
 =
 \bigl(1-p^{-2}\chi^{-1}(p)\bigr)
 D_p(\eta_{\chi,N}).
\end{equation}
\end{corollary}

\begin{proof}
The first equality is \eqref{eq:nu-negative-moments} at $r=1$, summed with the character coefficients.  In the second term of \eqref{eq:depleted-regulator}, substitute $b=ap$ modulo $N$; this introduces $\chi^{-1}(p)$.  The remaining sum is the Coleman regulator of the packet class.
\end{proof}

\subsection{Exceptional residues and boundedness}\label{subsec:exceptional-regulator}

\begin{theorem}[Exceptional residue]\label{thm:exceptional-residue}
For every finite cyclotomic presentation,
\begin{equation}\label{eq:exceptional-residue}
 \Res_{s=-1}
 \mathcal M_{\IwPs_{\xi,p},-1}(s)
 =
 (1-p^{-1})\Regdep(\xi).
\end{equation}
For a primitive odd packet this specializes to
\begin{equation}\label{eq:character-exceptional-residue}
 \Res_{s=-1}
 \mathcal M_{\IwPs_{\chi,N,p},-1}(s)
 =
 (1-p^{-1})
 \bigl(1-p^{-2}\chi^{-1}(p)\bigr)
 D_p(\eta_{\chi,N}).
\end{equation}
\end{theorem}

\begin{proof}
By Theorem~\ref{thm:Iwasawa-pseudomeasure-factorization}, the exceptional component is the product of the half-zeta component and the analytic Mellin transform of $\nu_\xi^\times$.  The first factor has residue $1-p^{-1}$ by \eqref{eq:half-zeta-residue}, and the second takes the value $\Regdep(\xi)$ at $s=-1$ by Theorem~\ref{thm:koblitz-negative-moments}.  Their product gives \eqref{eq:exceptional-residue}; Corollary~\ref{cor:character-depleted-regulator} gives the character formula.
\end{proof}

The residue in the weight coordinate differs from the polar coefficient in the Iwasawa coordinate.  Put
\begin{equation}\label{eq:residue-coordinate-notation}
 R_s(\xi):=(1-p^{-1})\Regdep(\xi),
 \qquad
 X=\gamma^s-1,
 \qquad
 X_0=\gamma^{-1}-1.
\end{equation}
Since
\begin{equation}\label{eq:X-coordinate-Jacobian}
 \left.\frac{dX}{ds}\right|_{s=-1}
 =\gamma^{-1}\log_p\gamma,
\end{equation}
a Laurent expansion
\[
 F(X)=\frac{A_X}{X-X_0}+B(X)
\]
has
\begin{equation}\label{eq:X-pole-vs-s-residue}
 \Res_{s=-1}F(\gamma^s-1)
 =\frac{A_X}{\gamma^{-1}\log_p\gamma},
 \qquad
 A_X=(\gamma^{-1}\log_p\gamma)R_s(\xi).
\end{equation}
Thus the $X$-polar coefficient and the $s$-residue differ by the displayed Jacobian.

\begin{lemma}\label{lem:finite-part-coordinate-change}
Put $t=s+1$, and let
\[
 u=at+bt^2+O(t^3),
 \qquad a\ne0.
\]
If a meromorphic branch has expansion
\[
 F(s)=\frac Rt+C_t+O(t),
\]
then its expansion in the parameter $u$ is
\begin{equation}\label{eq:finite-part-coordinate-change}
 F=\frac{aR}{u}+
 \left(C_t+\frac baR\right)+O(u).
\end{equation}
For $u=X-X_0$ with $X=\gamma^s-1$ and
$X_0=\gamma^{-1}-1$, this becomes
\begin{equation}\label{eq:X-finite-part-shift}
 C_X=C_s+\frac12\log_p(\gamma)\,R.
\end{equation}
Thus $R$ is the polar coefficient with respect to the distinguished
weight parameter $t=s+1$.  Under $u=at+O(t^2)$, the polar coefficient of
the function becomes $aR$.  What is invariant under reparametrization is
the residue of the associated meromorphic differential:
\[
 \Res_{t=0}F(t)\,dt
 =
 \Res_{u=0}F(t(u))\frac{dt}{du}\,du
 =R.
\]
The finite part has the coordinate-change law
\eqref{eq:finite-part-coordinate-change}.  In particular, a numerical
finite part is meaningful only after the local parameter has been
specified.
\end{lemma}

\begin{proof}
Formal inversion gives
\[
 t=\frac ua-\frac b{a^3}u^2+O(u^3),
 \qquad
 \frac1t=\frac au+\frac ba+O(u).
\]
Substitution proves \eqref{eq:finite-part-coordinate-change}.  Since
\[
 X-X_0
 =\gamma^{-1}\log_p(\gamma)t
 +\frac12\gamma^{-1}\log_p(\gamma)^2t^2+O(t^3),
\]
one has $b/a=\frac12\log_p(\gamma)$, giving
\eqref{eq:X-finite-part-shift}.
\end{proof}

\begin{remark}\label{rem:pole-regulator-meaning}
The singular exponential in the GSWZ completed logarithm cancels the
Coleman regulator term.  On unit weight space the same regulator appears as
the exceptional residue, so the pole is removable exactly when the residue
vanishes.
\end{remark}

\subsubsection{Frobenius depletion and regulator vanishing}\label{subsec:Frobenius-depletion}

When the coefficients are fixed by Frobenius, the depleted regulator is
obtained from the Coleman regulator by an invertible Frobenius operator.

\begin{theorem}[Invertibility of the Frobenius depletion operator]\label{thm:Frobenius-depletion-invertible}
Let $K/\Qp$ be unramified of residue degree $h$, let $\phi$ be the
Frobenius lift normalized by \(\phi(z)=z^p\) on Teichm\"uller roots, and
suppose the coefficients $a_z$ of
\(\xi=\sum_z a_z[z]_{\mathrm{cyc}}\) are fixed by $\phi$.  Let
\(\eta_\xi=\operatorname{cl}(\xi)\) be its coefficient-extended local
Bloch class.  Then Coleman equivariance \cite{Coleman} gives
\begin{equation}\label{eq:Frobenius-depletion-operator}
 \Regdep(\xi)
 =(1-p^{-2}\phi)D_p(\eta_\xi).
\end{equation}
The $\Qp$-linear operator $1-p^{-2}\phi$ is invertible, with
\begin{equation}\label{eq:Frobenius-depletion-inverse}
 (1-p^{-2}\phi)^{-1}
 =\frac{1+p^{-2}\phi+\dots+p^{-2(h-1)}\phi^{h-1}}
 {1-p^{-2h}}.
\end{equation}
Consequently
\begin{equation}\label{eq:depletion-zero-equivalence}
 \Regdep(\xi)=0
 \quad\Longleftrightarrow\quad
 D_p(\eta_\xi)=0.
\end{equation}
\end{theorem}

\begin{proof}
For a prime-to-$p$ Teichm\"uller root, Coleman Frobenius equivariance is
\[
 \phi(D_p(z))=D_p(z^p).
\]
Because every coefficient is fixed by $\phi$,
\[
 \phi(D_p(\eta_\xi))
 =\sum_z a_zD_p(z^p),
\]
which proves \eqref{eq:Frobenius-depletion-operator}.  Since $\phi^h=1$ on the unramified field,
\begin{align*}
 &(1-p^{-2}\phi)
 (1+p^{-2}\phi+\dots+p^{-2(h-1)}\phi^{h-1})\\
 &\qquad=1-p^{-2h}\phi^h=1-p^{-2h}.
\end{align*}
The scalar $1-p^{-2h}$ is nonzero and hence invertible in $K$, proving \eqref{eq:Frobenius-depletion-inverse} and the equivalence of kernels.
\end{proof}

\begin{remark}\label{rem:Frobenius-coefficient-twists}
For coefficient values not fixed by Frobenius, the packet formula becomes
\[
 \Regdep(\xi_{\chi,N})
 =(1-p^{-2}\chi^{-1}(p))D_p(\eta_{\chi,N}).
\]
\end{remark}

\subsubsection{Removability in the bounded-measure algebra}\label{subsec:measure-removability}

The removability argument is local on the Teichm\"uller components.  The
nonexceptional denominators are units; on the exceptional component the
denominator is an integral unit times a single linear factor.  The following two lemmas give the required divisibility statement.

\begin{lemma}\label{lem:bounded-denominator-division}
Let $x_0\in p\mathcal O$ and let
\[
 A(X)=\sum_{n\ge0}a_nX^n\in K[[X]]_{\bd}.
\]
Then
\begin{equation}\label{eq:bounded-denominator-divisibility}
 \frac{A(X)}{X-x_0}\in K[[X]]_{\bd}
 \quad\Longleftrightarrow\quad
 A(x_0)=0.
\end{equation}
If $A(x_0)=0$, the quotient coefficients are explicitly
\begin{equation}\label{eq:bounded-denominator-quotient-coefficients}
 \frac{A(X)}{X-x_0}
 =\sum_{n\ge0}b_nX^n,
 \qquad
 b_n=\sum_{m\ge n+1}a_mx_0^{m-n-1}.
\end{equation}
\end{lemma}

\begin{proof}
Choose $C$ with $\valp(a_m)\ge C$ for every $m$.  Since $\valp(x_0)>0$, the series defining $A(x_0)$ and every $b_n$ converges.  Moreover every summand in $b_n$ has valuation at least $C$, so $\valp(b_n)\ge C$ uniformly in $n$.  Thus the quotient in \eqref{eq:bounded-denominator-quotient-coefficients} belongs to $K[[X]]_{\bd}$.  Direct multiplication gives
\[
 (X-x_0)\sum_{n\ge0}b_nX^n=A(X)-A(x_0),
\]
so $A(x_0)=0$ proves the forward construction.  Conversely, if $A=(X-x_0)B$ with $B\in K[[X]]_{\bd}$, evaluation at $x_0$ converges and gives $A(x_0)=0$.
\end{proof}

\begin{lemma}\label{lem:exceptional-integral-unit-factor}
Let $c_0$ and $\gamma$ be the generators fixed in Theorem~\ref{thm:unique-exceptional-branch}, and put
\[
 a_{c_0}
 =\frac{\log_p\ang{c_0}}{\log_p\gamma}\in\Zp^\times,
 \qquad
 X_0=\gamma^{-1}-1.
\]
Then
\begin{equation}\label{eq:exceptional-integral-unit-factor}
 I_{r_{c_0},-1}(X)=U_{c_0}(X)(X-X_0),
 \qquad
 U_{c_0}(X)\in\mathcal O[[X]]^\times.
\end{equation}
\end{lemma}

\begin{proof}
Both $\ang{c_0}$ and $\gamma$ are topological generators of $\Gamma$, so $\ang{c_0}=\gamma^{a_{c_0}}$ with $a_{c_0}\in\Zp^\times$.  Since $1+X=\gamma^s$,
\[
 I_{r_{c_0},-1}(X)
 =1-\ang{c_0}(1+X)^{a_{c_0}}.
\]
Set $W=X-X_0$.  The identity $1+X=\gamma^{-1}(1+\gamma W)$ gives
\[
 I_{r_{c_0},-1}(X)
 =1-(1+\gamma W)^{a_{c_0}}
 =-\sum_{n\ge1}\binom{a_{c_0}}n\gamma^nW^n.
\]
Consequently
\[
 U_{c_0}(X)
 =-\sum_{n\ge1}\binom{a_{c_0}}n
 \gamma^n(X-X_0)^{n-1}.
\]
Its coefficients in the shifted variable $W$ are integral, and its constant coefficient is
\[
 -a_{c_0}\gamma\in\mathcal O^\times.
\]
Since \(W=0\) corresponds to \(X=X_0\), this also gives
\[
 U_{c_0}(X_0)=-a_{c_0}\gamma\in\mathcal O^\times.
\]
Because $X_0\in p\mathcal O$, translation by $X_0$ induces a continuous $\mathcal O$-algebra automorphism between $\mathcal O[[W]]$ and $\mathcal O[[X]]$.  Hence $U_{c_0}(X)$ is a unit of $\mathcal O[[X]]$.
\end{proof}

\begin{theorem}[Measure-removability criterion]\label{thm:measure-removability}
The following are equivalent.
\begin{enumerate}[label=(\roman*)]
\item The cyclotomic Iwasawa pseudomeasure $\IwPs_{\xi,p}$ belongs to $\IwK(G)$.
\item Its exceptional component belongs to $K[[X]]_{\bd}$, hence is a bounded $K$-valued measure.
\item $\Regdep(\xi)=0$.
\end{enumerate}
When these conditions fail, $\IwPs_{\xi,p}$ has a genuine simple pole at $(j,s)=(-1,-1)$.
\end{theorem}

\begin{proof}
Theorem~\ref{thm:unique-exceptional-branch} places every nonexceptional
component in $K[[X]]_{\bd}$.  On the exceptional component, take
$X=\gamma^s-1$ and $X_0=\gamma^{-1}-1\in p\mathcal O$.  By
Lemma~\ref{lem:exceptional-integral-unit-factor}, the regularizer is an
integral unit times $X-X_0$.  If $A(X)$ denotes the bounded numerator, then
Lemma~\ref{lem:bounded-denominator-division} shows that
$A/I_{r_{c_0},-1}$ has bounded denominators if and only if $A(X_0)=0$.
Since the denominator has a simple zero, this is equivalent to the
vanishing of the residue.  Theorem~\ref{thm:exceptional-residue} identifies
that residue with $(1-p^{-1})\Regdep(\xi)$, which gives the stated
equivalences and the pole assertion.
\end{proof}

\subsection[Bloch-class principal parts and local K3]{Bloch-class principal parts and local \(K_3\)}
\label{subsec:principal-part-descent}

The full pseudomeasure may depend on the chosen presentation, whereas its
singular part depends only on the Bloch class.  The quotient
\begin{equation}\label{eq:principal-part-quotient}
 \mathrm{PP}_K(G)
 :=
 \Sigma^{-1}\IwK(G)\big/\IwK(G)
\end{equation}
records this singular part modulo bounded measures; write
\(\operatorname{pp}\) for the quotient map.

Assume in this subsection that \(K/\Qp\) is unramified of residue degree
\(h\) and contains the roots in the support and coefficients.  Let \(\phi\) be
the Frobenius lift normalized by \(\phi(z)=z^p\) on Teichm\"uller roots.
It induces an automorphism \(\phi_B\) of \(B(K)\).
Extend it \(K\)-linearly, with the coefficient factor held fixed:
\begin{equation}\label{eq:cyclotomic-Frobenius-operator}
 \Phi_{\mathrm{cyc}}
 :=
 \phi_B\otimes\operatorname{id}_K:
 B(K)_K\longrightarrow B(K)_K.
\end{equation}
Thus
\[
 \Phi_{\mathrm{cyc}}\bigl(a[z]\bigr)=a[z^p]
\]
for a prime-to-\(p\) Teichm\"uller root \(z\), and
\(\Phi_{\mathrm{cyc}}^h=1\).  Let
\[
 D_{p,K}:B(K)_K\longrightarrow K
\]
denote the coefficient-linear extension of the Coleman regulator and
define the class-level depleted regulator
\begin{equation}\label{eq:class-level-depleted-regulator}
 D_{p,K}^{[p]}(\eta)
 :=
 D_{p,K}(\eta)
 -p^{-2}D_{p,K}\bigl(\Phi_{\mathrm{cyc}}\eta\bigr).
\end{equation}
\begin{theorem}[Bloch-class principal part]\label{thm:Bloch-principal-part}
Let
\[
 \xi=\sum_z a_z[z]_{\mathrm{cyc}},
 \qquad
 \eta_\xi=\operatorname{cl}_{K,K}(\xi)\in B(K)_K.
\]
Then
\begin{align}
 \Regdep(\xi)
 &=D_{p,K}^{[p]}(\eta_\xi),
 \label{eq:depleted-regulator-class-descent}\\
 \operatorname{pp}\bigl(\IwPs_{\xi,p}\bigr)
 &=
 D_{p,K}^{[p]}(\eta_\xi)\,
 \operatorname{pp}(\Zhalf)
 \quad\text{in }\mathrm{PP}_K(G).
 \label{eq:Iwasawa-principal-part-class-formula}
\end{align}
Consequently the principal part of \(\IwPs_{\xi,p}\) depends only on the
coefficient-extended Bloch class, even though the full pseudomeasure
depends on the chosen presentation.  Equivalently, if
\(\operatorname{cl}_{K,K}(\xi)=\operatorname{cl}_{K,K}(\xi')\), then
\begin{equation}\label{eq:equal-class-bounded-difference}
 \IwPs_{\xi,p}-\IwPs_{\xi',p}\in\IwK(G).
\end{equation}
The image of the resulting class-level map is contained in the
one-dimensional principal-part line generated by
\(\operatorname{pp}(\Zhalf)\); it equals that line whenever
\(D_{p,K}^{[p]}\) is nonzero on the cyclotomic Bloch subspace.
\end{theorem}

\begin{proof}
The field automorphism \(\phi\) acts functorially on the pre-Bloch group,
the Bloch group, and the Coleman regulator.  Since \(\phi(z)=z^p\) on
the roots in the support while \(\Phi_{\mathrm{cyc}}\) leaves the written
coefficients fixed,
\[
 D_{p,K}\bigl(\Phi_{\mathrm{cyc}}\eta_\xi\bigr)
 =\sum_z a_zD_p(z^p).
\]
Subtracting \(p^{-2}\) times this identity from
\(D_{p,K}(\eta_\xi)=\sum_z a_zD_p(z)\) proves
\eqref{eq:depleted-regulator-class-descent}.  In particular, the depleted regulator factors through the Bloch class.

It remains to identify the complete principal part.  On the exceptional
component write
\[
 Z(X)=\mathcal M_{\Zhalf,-1}(s),
 \qquad
 B_\xi(X)=\mathcal M_{\nu_\xi^\times,-1}(s),
 \qquad X=\gamma^s-1.
\]
The function \(B_\xi(X)\) belongs to \(K[[X]]_{\bd}\), and
Theorem~\ref{thm:koblitz-negative-moments} gives
\[
 B_\xi(X_0)=\Regdep(\xi),
 \qquad X_0=\gamma^{-1}-1.
\]
The factorization theorem gives
\[
 \mathcal M_{\IwPs_{\xi,p},-1}(s)=Z(X)B_\xi(X).
\]
By Lemma~\ref{lem:bounded-denominator-division},
\[
 \frac{B_\xi(X)-B_\xi(X_0)}{X-X_0}\in K[[X]]_{\bd}.
\]
The exceptional component of \(Z(X)\) has a denominator equal to an
integral unit times \(X-X_0\), by
Lemma~\ref{lem:exceptional-integral-unit-factor}; hence
\[
 Z(X)\bigl(B_\xi(X)-B_\xi(X_0)\bigr)\in K[[X]]_{\bd}.
\]
Every nonexceptional component of \(Z(X)\) is already bounded by
Theorem~\ref{thm:unique-exceptional-branch}.  Therefore
\[
 \IwPs_{\xi,p}-\Regdep(\xi)\Zhalf\in\IwK(G).
\]
Together with \eqref{eq:depleted-regulator-class-descent}, this proves
\eqref{eq:Iwasawa-principal-part-class-formula}.

If two presentations have the same class, their values under
\(D_{p,K}^{[p]}\) agree, so subtraction proves
\eqref{eq:equal-class-bounded-difference}.  Finally,
\(\operatorname{pp}(\Zhalf)\ne0\), because its exceptional component has
nonzero residue \(1-p^{-1}\) by
\eqref{eq:half-zeta-residue}.  Hence the displayed line is
one-dimensional.
\end{proof}

The theorem is stable under finite extension of the coefficient field:
after \(K\hookrightarrow E\), extend
\(\Phi_{\mathrm{cyc}}=\phi_B\otimes\operatorname{id}\) \(E\)-linearly
and base-change the measure algebra.  The extension \(E/K\) itself need
not be unramified, because Frobenius acts only on the support factor.
Define the cyclotomic part of the coefficient-extended Bloch group by
\[
 B_{\mathrm{cyc}}(K)_K
 :=
 \operatorname{im}\bigl(
 \operatorname{cl}_{K,K}:\CycPres(K;K)\to B(K)_K
 \bigr).
\]
Theorem~\ref{thm:Bloch-principal-part} gives a well-defined \(K\)-linear
map
\begin{equation}\label{eq:class-level-principal-part-map}
 \overline{\IwPs}_{p}:
 B_{\mathrm{cyc}}(K)_K
 \longrightarrow\mathrm{PP}_K(G),
 \qquad
 \eta\longmapsto
 D_{p,K}^{[p]}(\eta)\operatorname{pp}(\Zhalf).
\end{equation}

\begin{corollary}
\label{cor:principal-part-detects-K3}
Suppose \(p>3\), and let \(\xi\) have coefficients in \(\Zp\), so that
they are fixed by \(\phi\).  Then
\begin{equation}\label{eq:principal-part-K3-detection}
 \operatorname{pp}\bigl(\IwPs_{\xi,p}\bigr)=0
 \quad\Longleftrightarrow\quad
 \kappa_{K,p}(\eta_\xi)=0
 \ \text{in }K_3(K;\Zp).
\end{equation}
For a primitive odd character packet, after coefficient scalar extension
to a field \(E\) containing the character values, with
\(\Phi_{\mathrm{cyc}}\) acting trivially on \(E\),
\[
 \Phi_{\mathrm{cyc}}\eta_{\chi,N}
 =\chi^{-1}(p)\eta_{\chi,N},
\]
and hence
\begin{equation}\label{eq:packet-principal-part-class-formula}
 \operatorname{pp}\bigl(\IwPs_{\chi,N,p}\bigr)
 =
 \bigl(1-p^{-2}\chi^{-1}(p)\bigr)
 D_{p,E}(\eta_{\chi,N})\operatorname{pp}(\Zhalf).
\end{equation}
\end{corollary}

\begin{proof}
For Frobenius-fixed coefficients,
Theorem~\ref{thm:Frobenius-depletion-invertible} identifies
\(D_{p,K}^{[p]}(\eta_\xi)\) with
\((1-p^{-2}\phi)D_p(\eta_\xi)\), and this vanishes exactly when
\(D_p(\eta_\xi)=0\).  The \(p\)-completed Bloch--Suslin comparison and
\cite[Thm.~9]{GSWZ} identify the latter condition with
\(\kappa_{K,p}(\eta_\xi)=0\).  Theorem~\ref{thm:Bloch-principal-part}
and the nonvanishing of \(\operatorname{pp}(\Zhalf)\) prove
\eqref{eq:principal-part-K3-detection}.

For the packet, reindexing \(b=ap\) in its defining sum gives
\[
 \Phi_{\mathrm{cyc}}
 \sum_a\chi(a)[\zeta_N^a]
 =
 \sum_b\chi(bp^{-1})[\zeta_N^b]
 =
 \chi^{-1}(p)\eta_{\chi,N}.
\]
Insert this eigenvalue in
\eqref{eq:class-level-depleted-regulator} and apply
\eqref{eq:Iwasawa-principal-part-class-formula}.
\end{proof}

\begin{corollary}[Principal-part realization of local \(K_3\)]
\label{cor:K3-principal-part-realization}
Let \(p>3\) and let \(K/\Qp\) be unramified, with valuation ring
\(\mathcal O_K\).  Let \(\phi\) denote the Frobenius lift on \(K\) and
its functorial action on \(K_3(K;\Zp)\).  Then
\begin{equation}\label{eq:K3-principal-part-realization}
 J_{K,p}:K_3(K;\Zp)
 \overset{\sim}{\longrightarrow}
 \mathcal O_K\operatorname{pp}(\Zhalf),
 \qquad
 x\longmapsto
 \bigl(1-p^{-2}\phi\bigr)D_p(x)\,
 \operatorname{pp}(\Zhalf)
\end{equation}
is a \(\Zp\)-linear isomorphism.  If \(\xi\) has coefficients in
\(\Zp\), then
\begin{equation}\label{eq:presentation-compatible-K3-realization}
 J_{K,p}\bigl(\kappa_{K,p}(\eta_\xi)\bigr)
 =
 \operatorname{pp}\bigl(\IwPs_{\xi,p}\bigr).
\end{equation}
Conversely, every element of the lattice on the right of
\eqref{eq:K3-principal-part-realization} is represented by the
pseudomeasure principal part of a finite \(\Zp\)-linear cyclotomic
presentation.
\end{corollary}

\begin{proof}
By \cite[Thm.~9]{GSWZ}, the Coleman regulator is a
\(\Zp\)-linear isomorphism
\[
 D_p:K_3(K;\Zp)\overset{\sim}{\longrightarrow}p^2\mathcal O_K.
\]
On writing \(D_p(x)=p^2a\), the scalar in
\eqref{eq:K3-principal-part-realization} becomes
\[
 (1-p^{-2}\phi)(p^2a)=(p^2-\phi)a.
\]
The \(\Zp\)-linear endomorphism \(p^2-\phi\) of \(\mathcal O_K\) is an
automorphism: indeed
\[
 p^2-\phi=-\phi(1-p^2\phi^{-1}),
\]
and the second factor has inverse
\(\sum_{r\ge0}p^{2r}\phi^{-r}\) on \(\mathcal O_K\).
Since \(\operatorname{pp}(\Zhalf)\ne0\), multiplication by this
principal part identifies \(\mathcal O_K\) with
\(\mathcal O_K\operatorname{pp}(\Zhalf)\).  Therefore \(J_{K,p}\) is an isomorphism.

Regulator and Bloch--Suslin functoriality identify the scalar in
\(J_{K,p}(\kappa_{K,p}(\eta_\xi))\) with
\(D_{p,K}^{[p]}(\eta_\xi)\), so
\eqref{eq:presentation-compatible-K3-realization} is
\eqref{eq:Iwasawa-principal-part-class-formula}.  Finally,
\cite[Thm.~9]{GSWZ} also says that \(K_3(K;\Zp)\) is generated over
\(\Zp\) by the classes of prime-to-\(p\) roots of unity in \(K\).
The admissible class map of Lemma~\ref{lem:admissible-class-map}
therefore supplies a finite cyclotomic presentation of every class,
proving the final assertion.
\end{proof}

\begin{remark}\label{rem:presentation-dependence}
This additive Iwasawa-space statement is parallel to
\cite[Cor.~3.10]{GSWZ}, where a change of root-of-unity presentation changes
the corresponding GSWZ section by a Habiro unit, although the two arise
from different constructions.
\end{remark}

\subsection{Character components and local factorizations}\label{subsec:KL-comparison}

Let $\chi$ be primitive and odd modulo $N$, with $p\nmid N$.  Since $\nu_{\chi,N}^{\times}$ is bounded, define
\begin{equation}\label{eq:KL-branch-definition}
 L^{(p)}_{\chi^{-1},j}(s)
 =G(\chi)^{-1}\mathcal M_{\nu_{\chi,N}^{\times},j}(s).
\end{equation}
We use the reflected-branch normalization of the Kubota--Leopoldt
function, with the Teichm\"uller twist made explicit; see
\cite{KubotaLeopoldt,Washington}.

\begin{definition}[Kubota--Leopoldt normalization]\label{def:standard-KL-convention}
Let $\theta$ be a primitive Dirichlet character of conductor prime to $p$,
and let $j\in\Z/(p-1)\Z$.  We normalize the reflected branch
$L_p(1-s,\theta\teich^j)$ by
\begin{equation}\label{eq:standard-KL-interpolation}
 L_p(1-n,\theta\teich^j)
 =\bigl(1-\theta(p)p^{n-1}\bigr)L(1-n,\theta)
\end{equation}
for every integer $n\ge1$ with $n\equiv j\pmod{p-1}$.  When
$\theta(-1)=(-1)^j$, these are the critical-parity interpolation values.
If the parities differ, the matching complex values vanish and the
corresponding Kubota--Leopoldt branch is identically zero.  The Euler factor
is evaluated using the primitive character $\theta$, the primitive complex
associate of the finite-order character on $\Zp^\times$.
\end{definition}

\begin{theorem}[Kubota--Leopoldt comparison]\label{thm:KL-branch}
The function \eqref{eq:KL-branch-definition} is rigid analytic on every weight disc and satisfies
\begin{equation}\label{eq:KL-branch-interpolation}
 L^{(p)}_{\chi^{-1},j}(n)
 =\bigl(1-\chi^{-1}(p)p^{n-1}\bigr)L(1-n,\chi^{-1})
\end{equation}
for every $n\ge1$ with $n\equiv j\pmod{p-1}$.  With Definition~\ref{def:standard-KL-convention}, one has the identity of analytic functions
\begin{equation}\label{eq:KL-identification}
 L^{(p)}_{\chi^{-1},j}(s)
 =L_p(1-s,\chi^{-1}\teich^j).
\end{equation}
\end{theorem}

\begin{proof}
Analyticity follows from boundedness of $\nu_{\chi,N}^{\times}$.  At an arithmetic point $n\equiv j$, the Mellin integrand is
\[
 \teich(x)^j\ang x^n
 =\teich(x)^{j-n}x^n=x^n,
\]
so Proposition~\ref{prop:primitive-character-moments} gives \eqref{eq:KL-branch-interpolation}.  Definition~\ref{def:standard-KL-convention}, applied with the primitive character $\theta=\chi^{-1}$, gives the same value for $L_p(1-n,\chi^{-1}\teich^j)$.  When $j$ has mismatched parity, both sides are identically zero by the dense set of trivial-zero values; when the parity matches, the same dense arithmetic set gives the nonzero interpolation branch.  Rigid analytic uniqueness proves \eqref{eq:KL-identification} in either case.
\end{proof}

\subsubsection{Primitive and imprimitive local factorizations}\label{subsec:local-factorizations}

Retain the packet abbreviations introduced after
Definition~\ref{def:cyclotomic-Iwasawa-pseudomeasure}.

\begin{theorem}[Primitive local factorization]\label{thm:primitive-local-factorization}
Let $\chi$ be primitive and odd modulo $N$, with $p\nmid N$.  In the localized Iwasawa algebra,
\begin{equation}\label{eq:primitive-pseudomeasure-factorization}
 \IwPs_{\chi,N,p}
 =
 \Zhalf\starx\nu_{\chi,N}^{\times}.
\end{equation}
On Mellin branches,
\begin{equation}\label{eq:primitive-local-factorization}
 \mathcal M_{\IwPs_{\chi,N,p},j}(s)
 =
 G(\chi)
 Z_{1/2,j}(s)
 L^{(p)}_{\chi^{-1},j}(s).
\end{equation}
\end{theorem}

\begin{proof}
The pseudomeasure identity is Theorem~\ref{thm:Iwasawa-pseudomeasure-factorization} for the character presentation.  Multiplicative Mellin transforms turn convolution into multiplication, and equation~\eqref{eq:KL-branch-definition} gives the second formula.
\end{proof}

\begin{corollary}\label{cor:primitive-unit-moments}
For $n\ge1$ with $n\equiv j\pmod{p-1}$,
\begin{align}
 \mathcal M_{\IwPs_{\chi,N,p},j}(n)
 &=
 G(\chi)
 (1-p^n)\frac{B_{n+1}(1/2)}{n+1}
 \notag\\
 &\qquad\times
 \bigl(1-\chi^{-1}(p)p^{n-1}\bigr)
 L(1-n,\chi^{-1}).
 \label{eq:primitive-unit-moment-full}
\end{align}
\end{corollary}

\begin{proof}
Evaluate \eqref{eq:primitive-local-factorization} and use Theorems~\ref{thm:half-zeta-identification} and~\ref{thm:KL-branch}.
\end{proof}

For an imprimitive packet induced from $\chi_f$ to $N=fg$, define the finite Iwasawa branch factor
\begin{equation}\label{eq:imprimitive-Mellin-factor}
 C^{\mathrm{Iw}}_{\chi_f,g,j}(s)
 =
 \sum_{\substack{u\mid g\\(g/u,f)=1}}
 \mob(g/u)\chi_f(g/u)
 \teich(u)^j\ang u^s.
\end{equation}

\begin{theorem}[Imprimitive local factorization]\label{thm:imprimitive-local-factorization}
For the induced packet,
\begin{equation}\label{eq:imprimitive-pseudomeasure-factorization}
 \IwPs_{\chi_N,N,p}
 =
 \Zhalf\starx\nu_{\chi_f,g}^{\times},
\end{equation}
and
\begin{equation}\label{eq:imprimitive-local-factorization}
 \mathcal M_{\IwPs_{\chi_N,N,p},j}(s)
 =
 G(\chi_f)
 Z_{1/2,j}(s)
 C^{\mathrm{Iw}}_{\chi_f,g,j}(s)
 L^{(p)}_{\chi_f^{-1},j}(s).
\end{equation}
At an arithmetic point $n\equiv j$, the finite factor is $C^{\mathrm{Iw}}_{\chi_f,g}(n)$ from \eqref{eq:C-imprimitive-n}.
\end{theorem}

\begin{proof}
Use Proposition~\ref{prop:nu-imprimitive} in Theorem~\ref{thm:Iwasawa-pseudomeasure-factorization}.  The Mellin transform of $u_*\nu$ is multiplied by $\teich(u)^j\ang u^s$, which gives \eqref{eq:imprimitive-Mellin-factor}.  At a matching integer this multiplier is $u^n$.
\end{proof}

\section{Cyclotomic refinements and applications}\label{sec:fourier-refinement}

\subsection{Pre-Bloch distribution and refinement of presentations}\label{subsec:support-nonuniqueness}

Over an algebraically closed field, the constant root-of-unity term in
Suslin's distribution relation vanishes integrally in the pre-Bloch group.

\begin{theorem}[Integral pre-Bloch distribution relation]\label{thm:Bloch-distribution}
Let $F$ be a characteristic-zero field, let $n\ge2$, and let
$z\in(F^{\mathrm{alg}})^\times$.  In
$\mathcal P(F^{\mathrm{alg}})$, using the extended convention $[1]=0$,
one has
\begin{equation}\label{eq:Bloch-distribution}
 [z^n]
 =n\sum_{\varepsilon^n=1}[\varepsilon z].
\end{equation}
Equivalently, if $w^n=z$, then
\begin{equation}\label{eq:Bloch-support-refinement}
 [z]
 =n\sum_{\varepsilon^n=1}[\varepsilon w].
\end{equation}
Consequently both identities remain valid after tensoring with any
commutative $\Z$-algebra $A$.
\end{theorem}

\begin{proof}
Suslin's integral distribution relation is
\cite[Rem.~5.1, p.~198]{Suslin}
\[
 [z^n]
 =
 n\left(
   \sum_{\varepsilon^n=1}[\varepsilon z]
   -\sum_{\varepsilon^n=1}[\varepsilon]
 \right),
\]
with the standard extended convention $[1]=0$.  Since
$F^{\mathrm{alg}}$ is algebraically closed of characteristic zero, the
same remark proves that $\mathcal P(F^{\mathrm{alg}})$ is uniquely
divisible, hence torsion-free.  By \cite[Lem.~1.2, p.~181]{Suslin}, the
element $[u]+[u^{-1}]$ is $2$-torsion, and therefore
$[u^{-1}]=-[u]$ in $\mathcal P(F^{\mathrm{alg}})$.  Pair every
non-self-inverse $n$-th root $\varepsilon$ with
$\varepsilon^{-1}$.  The remaining roots are $1$, whose symbol is zero
by convention, and possibly $-1$; torsion-freeness also gives $[-1]=0$.
Thus the constant sum vanishes integrally, proving
\eqref{eq:Bloch-distribution}.  Over $\C$, the calculation
\[
 \sum_{\varepsilon^n=1}\Li_2(\varepsilon z)
 =\sum_{r\ge1}\frac{z^r}{r^2}
   \sum_{\varepsilon^n=1}\varepsilon^r
 =\frac1n\Li_2(z^n)
\]
checks the normalization of the factor $n$.  Applying
\eqref{eq:Bloch-distribution} to $w$ with $w^n=z$ gives
\eqref{eq:Bloch-support-refinement}.
\end{proof}

\begin{corollary}\label{cor:p-adic-Bloch-distribution}
For every prime $p$, the identities
\eqref{eq:Bloch-distribution} and
\eqref{eq:Bloch-support-refinement} hold in
$\mathcal P(F^{\mathrm{alg}})\otimes\Zp$.  If the original support has
order prime to $p$ and $p\nmid n$, every refined root also has order prime
to $p$.
\end{corollary}

\begin{proof}
Tensor the integral identity with $\Zp$.  If $u=\varepsilon w$ is a
refined root and $t=\ord(u)$, then
\[
 \ord(z)=\ord(u^n)=\frac{t}{(t,n)}.
\]
If $p\mid t$ while $p\nmid n$, the right side would be divisible by $p$,
contrary to the hypothesis on the original support.
\end{proof}

\begin{remark}\label{rem:Bloch-distribution-torsion}
The pre-Bloch distribution identity is integral for every prime.  The
condition $p>3$ enters only through the later GSWZ and local $K_3$
arguments, while $p\nmid n$ ensures that refined roots remain of order
prime to $p$.
\end{remark}

For an odd integer \(n>1\) prime to \(p\), and a prime-to-\(p\)
root of unity \(z\) with \(z^n\ne1\), define
\begin{equation}\label{eq:Delta-n-z-definition}
 \Delta_{n,z}
 =
 [z^n]_{\mathrm{cyc}}
 -n\sum_{\varepsilon^n=1}[\varepsilon z]_{\mathrm{cyc}}.
\end{equation}

\begin{proposition}
\label{prop:analytic-Bloch-distribution}
For \(n\) and \(z\) as above, one has
\begin{align}
 \nu_{\Delta_{n,z}}^\times
 &=r_n\starx\nu_{z^n}^\times,
 \label{eq:distribution-logarithmic-measure}\\
 \IwPs_{\Delta_{n,z},p}
 &=N_{[z^n]_{\mathrm{cyc}},n}
 =\lambda_{1/2,n}^{\times}\starx\nu_{z^n}^{\times}
 \in\Iw(G).
 \label{eq:distribution-bounded-remainder}
\end{align}
\end{proposition}

\begin{proof}
With \(q=1+T\), the product identity
\[
 \prod_{\varepsilon^n=1}(1-\varepsilon zq)=1-z^nq^n
\]
and the normalization \(C_u(0)=0\) give an identity of Amice transforms
\[
 \sum_{\varepsilon^n=1}C_{\varepsilon z}(T)
 =
 C_{z^n}\bigl((1+T)^n-1\bigr).
\]
By \eqref{eq:Amice-pushforward}, this is the distribution identity
\[
 \sum_{\varepsilon^n=1}\nu_{\varepsilon z}
 =n_*\nu_{z^n},
\]
where \(n_*\) is multiplicative push-forward.  Therefore
\[
 \nu_{\Delta_{n,z}}
 =
 \nu_{z^n}-n\,n_*\nu_{z^n}
 =
 (\delta_1-n\delta_n)\starx\nu_{z^n}
 =
 r_n\starx\nu_{z^n}.
\]
Since \(n\in\Zp^\times\), restriction to \(G\) commutes with \(n_*\),
which proves \eqref{eq:distribution-logarithmic-measure}.  Now use
\[
 r_n\starx\Zhalf=\lambda_{1/2,n}^{\times}
\]
and Theorem~\ref{thm:Iwasawa-pseudomeasure-factorization}:
\begin{align*}
 \IwPs_{\Delta_{n,z},p}
 &=
 \Zhalf\starx r_n\starx\nu_{z^n}^{\times}\\
 &=
 \lambda_{1/2,n}^{\times}\starx\nu_{z^n}^{\times}
 =
 N_{[z^n]_{\mathrm{cyc}},n}.
\end{align*}
The last measure is integral by
Theorem~\ref{thm:regularized-convolution}.
\end{proof}

\begin{corollary}\label{cor:support-not-intrinsic}
Let a nonzero cyclotomic Bloch class be represented by a finite sum
$\xi\in\mathcal P(F^{\mathrm{alg}})\otimes A$ with $\partial\xi=0$, where
either $A=\Q$, or $A=\Zp$.  In the $\Zp$ case assume that every root in
the original support has order prime to $p$, and take $p\nmid n$.  For
each term choose $w^n=z$ and make the explicit replacement
\begin{equation}\label{eq:explicit-support-refinement-map}
 a_z[z]\longmapsto
 n a_z\sum_{\varepsilon^n=1}[\varepsilon w].
\end{equation}
The resulting sum is another representative of the same class.  Define
the least common support order by
\[
 M(\xi)=\operatorname{lcm}\{\ord(z):a_z\ne0\}.
\]
Rationally, the replacement can enlarge $M(\xi)$ by any prescribed finite
product of primes.  Over $\Zp$, it can enlarge $M(\xi)$ by any prescribed
finite product of primes different from $p$.
\end{corollary}

\begin{proof}
First combine equal roots in the support, so every displayed coefficient is
nonzero.  Apply \eqref{eq:explicit-support-refinement-map} to each
remaining root, using Theorem~\ref{thm:Bloch-distribution}.  The resulting
sum is equal to $\xi$, so it has the same zero boundary and represents the
same Bloch class.  In the $\Zp$ case,
Corollary~\ref{cor:p-adic-Bloch-distribution} proves that every new root
remains prime to $p$.  Because
the original support excludes $1$, no term $\varepsilon w$ equals $1$:
otherwise $z=(\varepsilon w)^n=1$.  For distinct original roots the sets
of $n$-th roots are disjoint, since raising an element to the $n$-th power
recovers the original root.  Hence no cancellation between different
roots in the support is introduced by the refinement.

Let $M=M(\xi)$ be the least common order of the support, and fix a
prime $\ell$ (with $\ell\ne p$ in the $\Zp$ case), and choose a support
root $z$ of exact order $r$ for which
$v_\ell(r)=v_\ell(M)$.  Apply the refinement with $n=\ell$.  If
$\ell\nmid r$, choose the unique $w\in\mu_r$ with $w^\ell=z$; then every
nontrivial $\varepsilon\in\mu_\ell$ gives a root $\varepsilon w$ of exact
order $\ell r$.  If $\ell\mid r$, choose any solution
$w\in\mu_{\ell r}$ of $w^\ell=z$.  Such solutions exist because the
$\ell$-power map $\mu_{\ell r}\to\mu_r$ is surjective, and every such
$w$ has exact order $\ell r$: from
\[
 \ord(w^\ell)
 =\frac{\ord(w)}{(\ord(w),\ell)}
 =r
\]
and $\ord(w)\mid\ell r$, comparison of $\ell$-adic valuations and of
the prime-to-$\ell$ parts forces $\ord(w)=\ell r$.
In either case the refined support contains a nonzero term of order
$\ell r$, so the new least common support order has
$\ell$-adic valuation at least $v_\ell(M)+1$.  Repeating this step adds
any prescribed multiplicity of $\ell$; repeating it for the finitely
many prescribed primes proves both assertions.
\end{proof}

\begin{corollary}
\label{cor:refinement-bounded-change}
Assume \(K/\Qp\) is unramified and contains the original and refined
prime-to-\(p\) support.  Let \(\xi^{(n)}\) be obtained from \(\xi\) by
any finite sequence of the replacements
\eqref{eq:explicit-support-refinement-map}, with every refinement degree
prime to \(p\).  Then
\begin{equation}\label{eq:refinement-bounded-pseudomeasure-change}
 \IwPs_{\xi^{(n)},p}-\IwPs_{\xi,p}\in\IwK(G).
\end{equation}
In particular, the support order can be enlarged as in
Corollary~\ref{cor:support-not-intrinsic} without changing the
principal-part class or the exceptional residue.

If a simultaneous refinement uses the same odd degree \(n\) for every
term, so that
\[
 \xi=\sum_z a_z[z]_{\mathrm{cyc}},
 \qquad
 \xi^{(n)}
 =\sum_zna_z\sum_{\varepsilon^n=1}
 [\varepsilon w_z]_{\mathrm{cyc}},
 \qquad w_z^n=z,
\]
then the bounded difference is explicitly
\begin{equation}\label{eq:explicit-refinement-bounded-remainder}
 \IwPs_{\xi^{(n)},p}-\IwPs_{\xi,p}
 =
 -\sum_z a_zN_{[z]_{\mathrm{cyc}},n}.
\end{equation}
\end{corollary}

\begin{proof}
The integral pre-Bloch distribution relation shows that every refinement
has the same coefficient-extended Bloch class as the original
presentation.  Equation~\eqref{eq:equal-class-bounded-difference} now
gives \eqref{eq:refinement-bounded-pseudomeasure-change}.  Equality of
residues also follows directly from
\eqref{eq:Iwasawa-principal-part-class-formula}.  For an odd simultaneous
refinement, the difference contributed by the term \(a_z[z]\) is
\(-a_z\Delta_{n,w_z}\).  Apply
\eqref{eq:distribution-bounded-remainder} with \(w_z^n=z\) and sum over
the support to obtain
\eqref{eq:explicit-refinement-bounded-remainder}.
\end{proof}

\subsection{Finite-jet interpolation and rigidity}
\label{subsec:refinement-rigidity}

The preceding formula treats one distribution relation at a time.  Finite
linear combinations of refinements interpolate the component multipliers
simultaneously.

\begin{definition}[Cyclotomic refinement operators]
\label{def:cyclotomic-refinement-operators}
Let \(m\ge1\) be prime to \(p\).  After passing to a field containing
the \(m\)-th roots of the support, define
\[
 R_m[z]_{\mathrm{cyc}}
 :=
 m\sum_{u^m=z}[u]_{\mathrm{cyc}},
\]
and extend \(R_m\) linearly to finite cyclotomic
presentations.  Thus \(R_1\) is the identity and
\[
 R_mR_n=R_{mn}.
\]
\end{definition}

\begin{proposition}[Exact covariance under cyclotomic refinement]
\label{prop:exact-refinement-covariance}
Let \(p\) be odd, let \(K/\Qp\) be finite unramified with valuation ring
\(\mathcal O_K\), and let
\(\xi\in\CycPres(K;\mathcal O_K)\).  Suppose that its support has
order prime to \(p\).  Let \(m\ge1\) be prime to \(p\), and let \(L/K\)
be a finite unramified extension containing all roots occurring in
\(R_m\xi\).  After scalar extension to \(L\), one has
\begin{align}
 \operatorname{cl}_{L,\mathcal O_K}(R_m\xi)
 &=
 \operatorname{cl}_{L,\mathcal O_K}(\xi),
 \label{eq:refinement-class-covariance}\\
 \nu_{R_m\xi}^{\times}
 &=
 (m\delta_m)\starx\nu_\xi^{\times},
 \label{eq:refinement-nu-covariance}\\
 \IwPs_{R_m\xi,p}
 &=
 (m\delta_m)\starx\IwPs_{\xi,p}.
 \label{eq:refinement-pseudomeasure-covariance}
\end{align}
In particular,
\[
 \IwPs_{R_m\xi-\xi,p}
 =
 (m\delta_m-\delta_1)\starx\IwPs_{\xi,p}.
\]
\end{proposition}

\begin{proof}
The local coefficient ring has \(2\) invertible.  Hence the constant
root-of-unity term in Suslin's distribution relation vanishes after
tensoring with \(\mathcal O_K\): pair every non-self-inverse
\(\varepsilon\in\mu_m\) with \(\varepsilon^{-1}\), use
\([\varepsilon^{-1}]=-[\varepsilon]\), and note that the remaining
symbols \([1]\) and, when present, \([-1]\) vanish.  The distribution
relation over \(L\) therefore gives
\[
 [z]
 =
 m\sum_{u^m=z}[u]
 \qquad\text{in }\mathcal P(L)\otimes_{\Z}\mathcal O_K.
\]
This gives \eqref{eq:refinement-class-covariance}.

For the analytic identity, put \(q=1+T\).  For one support root,
the normalized logarithmic Amice transforms satisfy
\begin{align*}
 \Ami_{\nu_{R_m[z]}}(T)
 &=
 m\sum_{u^m=z}
 \log\frac{1-u}{1-uq}\\
 &=
 m\log
 \frac{\prod_{u^m=z}(1-u)}
      {\prod_{u^m=z}(1-uq)}\\
 &=
 m\log\frac{1-z}{1-zq^m}\\
 &=
 m\,\Ami_{\nu_z}(q^m-1).
\end{align*}
By \eqref{eq:Amice-pushforward}, the last expression is the Amice
transform of \(m\,m_*\nu_z=(m\delta_m)\starx\nu_z\).  Linearity proves
the identity on all of \(\Zp\).  Since \(m\in\Zp^\times\), restriction
to \(G=\Zp^\times\) commutes with \(m_*\), proving
\eqref{eq:refinement-nu-covariance}.  Finally,
Theorem~\ref{thm:Iwasawa-pseudomeasure-factorization} gives
\[
 \IwPs_{R_m\xi,p}
 =
 \Zhalf\starx(m\delta_m)\starx\nu_\xi^\times
 =
 (m\delta_m)\starx\IwPs_{\xi,p},
\]
which is \eqref{eq:refinement-pseudomeasure-covariance}.
\end{proof}

\begin{theorem}[Finite-jet control by refinements]
\label{thm:refinement-interpolation}
Assume that $K/\Qp$ is finite unramified, with valuation ring
$\mathcal O_K$, and that $K$ contains the prime-to-$p$ support of
$\xi\in\CycPres(K;\mathcal O_K)$.  For
\[
 e=0,\ldots,p-2,
 \qquad
 j_e\equiv e-1\pmod{p-1},
\]
center the $j_e$-th branch at $s=e-1$ by
\[
 Y=(1+p)^{s-(e-1)}-1.
\]
Fix $N\ge0$ and elements
\[
 v_{e,q}\in\mathcal O_K,
 \qquad
 0\le e\le p-2,
 \quad
 0\le q\le N,
 \qquad
 v_{0,0}=1.
\]
Then the following hold.
\begin{enumerate}[label=(\roman*)]
\item There exist finitely many integers $m$ prime to $p$, coefficients
$c_m\in\mathcal O_K$, and a finite unramified extension $L/K$ containing
the refined support such that, for
\[
 C=\sum_m c_mR_m,
 \qquad
 \xi'=C\xi,
\]
one has
\begin{equation}\label{eq:refinement-coefficient-sum}
 \sum_m c_m=1
\end{equation}
and
\begin{equation}\label{eq:refinement-same-Bloch-class}
 \operatorname{cl}_{L,\mathcal O_K}(\xi')
 =
 \operatorname{cl}_{L,\mathcal O_K}(\xi).
\end{equation}

\item On every centered component,
\begin{equation}\label{eq:refinement-branch-factorization}
 \mathcal M_{\IwPs_{\xi',p},j_e}(s)
 =
 H_{C,e}(Y)\,
 \mathcal M_{\IwPs_{\xi,p},j_e}(s),
\end{equation}
where
\begin{equation}\label{eq:refinement-branch-multiplier}
 H_{C,e}(Y)
 =
 \sum_m c_m m^e(1+Y)^{\ell_m},
 \qquad
 \ell_m=\frac{\log_p\ang m}{\log_p(1+p)}\in\Zp.
\end{equation}

\item The multiplier jets satisfy
\begin{equation}\label{eq:prescribed-refinement-jets}
 [Y^q]H_{C,e}(Y)=v_{e,q}
 \qquad
 (0\le e\le p-2,\ 0\le q\le N).
\end{equation}
Thus the normalized family $\sum_m c_m=1$ imposes no finite-jet
compatibility condition beyond $H_{C,0}(0)=1$.
\end{enumerate}
\end{theorem}

\begin{proof}
First derive the component multiplier.
Proposition~\ref{prop:exact-refinement-covariance} and linearity give
\[
 \IwPs_{C\xi,p}
 =
 \left(\sum_m c_m m\delta_m\right)
 \starx\IwPs_{\xi,p}.
\]
On the branch \(j_e=e-1\), the Mellin multiplier contributed by
\(m\delta_m\) is
\begin{align*}
 m\teich(m)^{e-1}\ang m^s
 &=
 m^e\ang m^{\,s-(e-1)}\\
 &=
 m^e(1+Y)^{\ell_m}.
\end{align*}
This gives \eqref{eq:refinement-branch-factorization} and
\eqref{eq:refinement-branch-multiplier}.  In particular,
\begin{equation}\label{eq:refinement-Mahler-coefficients}
 [Y^q]H_{C,e}
 =
 \sum_m c_m m^e\binom{\ell_m}{q}.
\end{equation}

Choose \(r\ge1\) such that \(p^r>N\), and put \(N_r=p^r-1\).
For every
\[
 (a,b)\in
 \F_p^\times\times\{0,\ldots,p^r-1\},
\]
choose a positive integer \(m_{a,b}\), prime to \(p\), such that
\begin{equation}\label{eq:refinement-coordinate-choice}
 m_{a,b}\equiv a\pmod p,
 \qquad
 \ell_{m_{a,b}}\equiv b\pmod{p^r}.
\end{equation}
These choices exist because
\[
 (\Z/p^{r+1}\Z)^\times
 \longrightarrow
 \F_p^\times\times\Z/p^r\Z,
 \qquad
 m\longmapsto
 \bigl(\widebar m,\ell_m\bmod p^r\bigr),
\]
is a bijection.

Consider the square matrix
\[
 E_{(e,q),(a,b)}
 =
 m_{a,b}^e
 \binom{\ell_{m_{a,b}}}{q},
\]
indexed by
\[
 0\le e\le p-2,\quad0\le q\le N_r
\]
in the rows and by \((a,b)\) in the columns.  For \(q<p^r\),
the reduction of \(\binom{x}{q}\) modulo \(p\) depends only on
\(x\bmod p^r\).  Indeed,
\[
 (1+T)^{x+p^r}
 \equiv
 (1+T)^x(1+T^{p^r})
 \pmod p,
\]
so the coefficients below degree \(p^r\) are unchanged.  Therefore
\[
 E\bmod p
 =
 \bigl(a^e\bigr)_{\substack{0\le e\le p-2\\a\in\F_p^\times}}
 \otimes
 \left(\binom bq\right)_{\substack{0\le q\le N_r\\0\le b\le N_r}}.
\]
The first factor is the evaluation matrix of
\(1,X,\ldots,X^{p-2}\) on \(\F_p^\times\), hence is invertible.
The second factor is triangular with diagonal entries \(1\).
Consequently
\[
 \det(E)\in\Zp^\times.
\]

Extend the prescribed targets to \(q\le N_r\), for example by setting
the unused targets equal to zero.  Since \(E\) is invertible over
\(\Zp\), equation \eqref{eq:refinement-Mahler-coefficients} has a
unique solution
\[
 c_{a,b}\in\mathcal O_K.
\]
Set
\[
 C=\sum_{a,b}c_{a,b}R_{m_{a,b}}.
\]
The row \((e,q)=(0,0)\) of \(E\) consists entirely of \(1\)'s.  Thus
its equation is
\[
 \sum_{a,b}c_{a,b}=v_{0,0}=1,
\]
which proves \eqref{eq:refinement-coefficient-sum}.
Proposition~\ref{prop:exact-refinement-covariance} then proves
\eqref{eq:refinement-same-Bloch-class}.  Finally, only finitely many
prime-to-\(p\) roots have been introduced, so they lie in a common
finite unramified extension \(L/K\).
\end{proof}

For a nonzero bounded branch
\(F(Y)=\sum_{n\ge0}a_nY^n\in K[[Y]]_{\bd}\), write
\[
 \mu_{\mathrm{Iw}}(F)=\min_n\valp(a_n),
 \qquad
 \lambda_{\mathrm{Iw}}(F)=\min\{n:\valp(a_n)=\mu_{\mathrm{Iw}}(F)\}.
\]

\begin{corollary}[Independent control of nonexceptional
Weierstrass invariants]
\label{cor:refinement-independent-lambda-control}
Assume the hypotheses of
Theorem~\ref{thm:refinement-interpolation}.  For every nonzero
nonexceptional component \(e\), write
\((\mu_{\mathrm{Iw},e}(\xi),\lambda_{\mathrm{Iw},e}(\xi))\) for its Weierstrass invariants in the
coordinate \(Y\).  Given arbitrary integers
\[
 d_e\ge0,\qquad 1\le e\le p-2,
\]
there is, after a finite unramified base change, one integral
presentation \(\xi'\) of the same base-changed Bloch class such that
simultaneously, on every nonzero nonexceptional component,
\[
 \mu_{\mathrm{Iw},e}(\xi')=\mu_{\mathrm{Iw},e}(\xi),
 \qquad
 \lambda_{\mathrm{Iw},e}(\xi')=\lambda_{\mathrm{Iw},e}(\xi)+d_e.
\]

The construction may simultaneously be required to satisfy
\[
 H_{C,0}(Y)\equiv1\pmod{Y^{M+1}}
\]
for any prescribed \(M\).  Since the exceptional component has at most a
simple pole, this leaves its Laurent coefficients from \(Y^{-1}\)
through \(Y^{M-1}\) unchanged.
\end{corollary}

\begin{proof}
Choose \(N\ge\max_e d_e\) and prescribe
\[
 [Y^q]H_{C,e}=
 \begin{cases}
 1,&q=d_e,\\
 0,&0\le q\le N,\ q\ne d_e
 \end{cases}
 \qquad(1\le e\le p-2).
\]
Thus
\[
 H_{C,e}(Y)=Y^{d_e}+O(Y^{N+1}),
\]
so \(H_{C,e}\) has Weierstrass invariants \((0,d_e)\).
Weierstrass invariants are additive under multiplication of nonzero
integral power series, and
\eqref{eq:refinement-branch-factorization} proves the stated formulas.
For the final assertion, enlarge \(N\) if necessary and prescribe
\([Y^0]H_{C,0}=1\) and
\([Y^q]H_{C,0}=0\) for \(1\le q\le M\).
\end{proof}

\begin{corollary}[Finite Laurent-jet variation under cyclotomic refinements]
\label{cor:exceptional-residue-rigidity}
Write the exceptional component in the coordinate
\(Y=(1+p)^{s+1}-1\) as
\[
 \mathcal M_{\IwPs_{\xi,p},-1}(s)
 =
 a_{-1}Y^{-1}+a_0+a_1Y+\cdots.
\]
If \(a_{-1}\ne0\), then, after extending coefficients from
\(\mathcal O_K\) to its fraction field \(K\),
every finite Laurent jet
\[
 a_{-1}Y^{-1}+b_0+b_1Y+\cdots+b_MY^M
\]
with the same polar coefficient \(a_{-1}\) is attained by a normalized
finite linear refinement operator applied to \(\xi\).  These exceptional
jets may be prescribed simultaneously with arbitrary finite multiplier
jets on all nonexceptional components.

Consequently, within the family of normalized finite linear refinement
operators, after scalar extension to the fraction field, the exceptional
residue is the only rigid finite Laurent coefficient when it is nonzero.
\end{corollary}

\begin{proof}
A normalized finite linear refinement operator has exceptional multiplier
\[
 H_{C,0}(Y)=1+h_1Y+h_2Y^2+\cdots.
\]
In the coefficient of \(Y^n\) of
\(H_{C,0}\mathcal M_{\IwPs_{\xi,p},-1}\), the new variable
\(h_{n+1}\) occurs with coefficient \(a_{-1}\); all other terms depend
only on \(h_1,\ldots,h_n\).  Since \(a_{-1}\ne0\), the equations for
the desired coefficients \(b_0,\ldots,b_M\) can be solved recursively
over the fraction field.  After tensoring the interpolation system of
Theorem~\ref{thm:refinement-interpolation} from \(\mathcal O_K\) with its
fraction field \(K\), realize the resulting multiplier coefficients
together with the independently prescribed nonexceptional ones.
\end{proof}

\begin{remark}\label{rem:refinement-range}
Refinement may require a finite unramified base change, so the preserved
object is the corresponding base change of the original Bloch class and,
when applicable, of its local $K_3$-image.  Since refinement acts by
multiplication on each branch, existing zeros persist: a nonexceptional
$\lambda$-invariant may increase but cannot decrease, and a regular
exceptional branch retains its order and leading coefficient.  The
optimality statement is relative to the normalized affine span of the
operators $R_m$.
\end{remark}

\subsection{Comparison with Mellin transforms and GSWZ germs}\label{subsec:comparisons}

\subsubsection{Primitive complex Mellin comparison}

Let \(\chi\) be a primitive odd character modulo \(N\).  Retain the packet
\(\xi_{\chi,N}\) from \eqref{eq:character-packet} and the functions
\(f_{\xi_{\chi,N}}\) and \(L_{\xi_{\chi,N}}\) from
\eqref{eq:general-complex-basic-path}.  Explicitly,
\[
 f_{\xi_{\chi,N}}(\hbar)
 =\sum_k\chi(k)
 \log\bigl((e^{-\hbar}\zeta_N^k;e^{-2\hbar})_\infty\bigr),
 \qquad
 L_{\xi_{\chi,N}}(s)
 =\frac1{\Gamma(s)}\int_0^\infty
 f_{\xi_{\chi,N}}(\hbar)\hbar^{s-1}\,d\hbar.
\]

\begin{theorem}[Primitive half-shifted factorization]
\label{thm:primitive-complex-factorization}
For \(\Re(s)>1\),
\[
 L_{\xi_{\chi,N}}(s)
 =-G(\chi)(1-2^{-s})\zeta(s)L(s+1,\chi^{-1}).
\]
\end{theorem}

\begin{proof}
Absolute convergence permits expansion of the Pochhammer logarithm and
termwise Mellin integration.  The primitive Gauss identity converts the
root-of-unity sum into \(G(\chi)\chi^{-1}(r)\), while the half shift
restricts the outer variable to odd integers, producing
\((1-2^{-s})\zeta(s)\).  Convergence estimates and the imprimitive and radial extensions are proved
in Appendix~\ref{sec:complex}.
\end{proof}

\subsubsection{Compatibility with GSWZ germs}

The fixed-center construction is compatible with the GSWZ germ family.
The higher-center statement requires the global-image and root-presentation
hypotheses of \cite{GSWZ}.

\begin{proposition}
\label{prop:GSWZ-germ-comparison}
Assume the GSWZ hypotheses and use \(\eta_p\), \(\widehat\xi\), and
\(g_\zeta\) as in Appendix~\ref{sec:higher-gswz}.  At the center \(1\),
write
\[
 g_1(x)=\frac{A}{\log(1+x)}+h_1(x),
 \qquad A=D_p(\eta_p),
\]
and, for odd \(c>1\) prime to \(p\), put
\(x_c=(1+x)^c-1\).  Then the constant-normalized regularization satisfies
\[
 \bigl(g_1(x)-c g_1(x_c)\bigr)^0
 =F_{\widehat\xi,p}(x)-cF_{\widehat\xi,p}(x_c)
 =\Ami_{\lambda_{1/2,c}\starx\nu_{\widehat\xi}}(x),
\]
and its restriction to \(G\) is
\[
 N_{\widehat\xi,c}=r_c\starx\IwPs_{\widehat\xi,p}.
\]
More generally, if \(\zeta\) has exact order \(m\) and
\(d=(m,c)\), then
\[
 g_\zeta(q)-\frac{c}{d^2}g_{\zeta^c}(q^c)
\]
is the unique regularization of this form that cancels the common
\(A/(m^2\log q)\) pole for arbitrary \(A\), and this regularization
commutes with the logarithmic Frobenius difference.
\end{proposition}

\begin{proof}
At the fixed center the two polar terms cancel because
\(\log((1+x)^c)=c\log(1+x)\); the regular part is the formal germ of
Section~\ref{sec:local-germs}, and
Theorem~\ref{thm:regularized-convolution} identifies its Amice transform.
At a center of order \(m\), the target center has order \(m/d\), so equality
of polar coefficients forces the factor \(c/d^2\).  Power substitution and
Frobenius commute.  The all-center Dwork argument and the integrality statement are proved in
Appendix~\ref{sec:higher-gswz}.
\end{proof}

\subsection[The knot 5-2]{A local \(K_3\)-criterion for the knot \(5_2\)}\label{sec:52}

Let
\[
 F=\Q(\alpha),
 \qquad
 \alpha^3-\alpha^2+1=0,
\]
and consider the Bloch element
\[
 \beta_{5_2}=2[1-\alpha^2]+[1-\alpha]\in B(F).
\]
Let \(p>3\), let \(\bar\alpha\in\F_p\) be a simple root of
\(h(T)=T^3-T^2+1\), and let \(\alpha_{p,\bar\alpha}\in\Zp\) be its
Hensel lift.  Set
\[
 z_{p,\bar\alpha}=1-\alpha_{p,\bar\alpha}^2,
 \qquad
 a=1-\bar\alpha^2\in\F_p,
\]
and let \(\beta^B_{5_2,p,\bar\alpha}\) be the image of \(\beta_{5_2}\) in
\(B(\Qp)\otimes\Zp\), with \(\beta^K_{5_2,p,\bar\alpha}\) its image in
\(K_3(\Qp;\Zp)\).

Put \(g(X)=X^3-2X^2+3X-1\).  The identities
\[
 g(1-T^2)=-h(T)(T^3+T^2-1),
 \qquad
 g'(1-T^2)=T^2h'(T)+2h(T)
\]
show that \(a\) is a simple root of \(g\).  Moreover,
\(a\notin\{0,1,-2\}\): the first two exclusions follow directly from
\(h(\bar\alpha)=0\), while \(a=-2\) would force \(p=23\),
\(\bar\alpha=16\), and \(h'(\bar\alpha)=0\), contrary to simplicity.
For \(r=1,2\), define the finite-polylogarithm polynomial
\begin{equation}\label{eq:finite-polylog-polynomial}
 \li_{r,p}(X):=\sum_{k=1}^{p-1}k^{-r}X^k\in\F_p[X],
\end{equation}
and define
\begin{equation}\label{eq:52-Wp-definition}
 W_p(X)
 :=\li_{2,p}(X)+\frac{3\li_{1,p}(X)^2}{2(X+2)}
 \qquad\bigl(X\in\F_p\setminus\{-2\}\bigr).
\end{equation}

\begin{theorem}[Finite-polylogarithm criterion]
\label{thm:main-52-criterion}
One has
\[
 \beta^B_{5_2,p,\bar\alpha}=3[z_{p,\bar\alpha}]
 \quad\text{in }B(\Qp)\otimes\Zp,
\]
and
\[
 \frac{D_p(\beta^B_{5_2,p,\bar\alpha})}{p^2}
 \equiv-\frac{3W_p(a)}{1-a}\pmod p.
\]
Consequently,
\[
 \beta^K_{5_2,p,\bar\alpha}
 \text{ generates }K_3(\Qp;\Zp)
 \quad\Longleftrightarrow\quad
 W_p(a)\ne0.
\]
\end{theorem}

\begin{proof}
A direct boundary calculation shows that $\beta_{5_2}$ lies in the Bloch
group.  The cubic relation gives
\(1-\alpha=\alpha^{-2}=(1-(1-\alpha^2))^{-1}\), and the tetrahedral
cross-ratio identities identify the two shape symbols after inverting
\(6\).  Hence
\(\beta^B_{5_2,p,\bar\alpha}=3[z_{p,\bar\alpha}]\).

Write
\(\alpha_{p,\bar\alpha}=\teich(\bar\alpha)(1+pw)\) modulo \(p^2\).  The exact
identities
\(z_{p,\bar\alpha}=-\alpha_{p,\bar\alpha}^3\) and
\(1-z_{p,\bar\alpha}=\alpha_{p,\bar\alpha}^2\) give the residue-disc parameters
\(u=3w\) and \(v=2w\).  The finite logarithm relation yields
\(\li_{1,p}(a)=(a+2)w\).  Substitution in Besser's residue-disc
expansion \cite[Prop.~2.3]{Besser} for \(D_p\) gives the displayed
congruence.  Finally,
\(D_p:K_3(\Qp;\Zp)\overset{\sim}{\to}p^2\Zp\) identifies generation with
nonvanishing of the normalized first digit.  The residue-disc calculation
is given in Appendix~\ref{app:52-details}.
\end{proof}

\begin{example}
Modulo \(5\), one may take \(\bar\alpha=2\), hence \(a=2\).  Direct finite sums
give
\[
 \li_{1,5}(2)=4,
 \qquad
 \li_{2,5}(2)=1,
 \qquad
 W_5(2)=2.
\]
Therefore
\[
 5^{-2}D_5(\beta^B_{5_2,5,2})
 \equiv-\frac{3\cdot2}{1-2}\equiv1\pmod5,
\]
so the local knot class generates \(K_3(\Q_5;\Z_5)\).
\end{example}

\clearpage
\appendix

\section{Character decompositions, Kummer congruences, and Mellin transforms}
\label{app:character-kummer}\label{sec:complex}

\subsection{Odd Fourier decomposition}
\label{subsec:exact-order-modules}

Let \(M>2\), let \(K\) contain a primitive \(M\)-th root
\(\zeta_M\), and let \(E\) split the character group of
\((\Z/M\Z)^\times\).  For every odd character \(\chi\) modulo \(M\),
retain the packet \(\xi_{\chi,M}\) from \eqref{eq:character-packet} and
extend scalars to \(E\).  Character orthogonality gives
\begin{equation}\label{eq:odd-Fourier-basis}
 \CycPresM M(K;E)^-
 =\bigoplus_{\chi(-1)=-1}E\xi_{\chi,M}.
\end{equation}
If
\(
 \xi_M=\sum_a c_a[\zeta_M^a]_{\mathrm{cyc}}
\)
with \(c_{-a}=-c_a\), then
\begin{equation}\label{eq:odd-Fourier-coefficients}
 \xi_M
 =\frac1{\varphi(M)}
 \sum_{\chi(-1)=-1}
 \left(\sum_a c_a\chi^{-1}(a)\right)\xi_{\chi,M}.
\end{equation}
In particular,
\begin{equation}\label{eq:odd-difference}
 [\zeta_M^a]_{\mathrm{cyc}}-[\zeta_M^{-a}]_{\mathrm{cyc}}
 =\frac2{\varphi(M)}
 \sum_{\chi(-1)=-1}\chi^{-1}(a)\xi_{\chi,M}.
\end{equation}
Indeed,
\[
 \sum_{\chi(-1)=-1}\chi^{-1}(a)\chi(b)
 =\frac{\varphi(M)}2
 \bigl(\one_{b=a}-\one_{b=-a}\bigr).
\]
A presentation supported on roots whose orders divide \(M\) decomposes
first by exact order and then by \eqref{eq:odd-Fourier-coefficients}.

All constructions used in the paper are linear in the written
presentation.  Thus, if
\[
 \xi=\sum_{i=1}^r b_i\xi_{\chi_i,M_i},
\]
then, for every \((m,p)=1\),
\[
 F_{\xi,p,m}=\sum_i b_iF_{\xi_{\chi_i,M_i},p,m},
 \qquad
 \nu_\xi^\times=\sum_i b_i\nu_{\chi_i,M_i}^\times,
 \qquad
 \IwPs_{\xi,p}=\sum_i b_i\IwPs_{\chi_i,M_i,p}.
\]
Consequently,
\[
 \mathcal M_{\IwPs_{\xi,p},j}(s)
 =\sum_i b_i\mathcal M_{\IwPs_{\chi_i,M_i,p},j}(s),
\]
where each summand is given by the primitive or imprimitive local
factorization of Subsection~\ref{subsec:local-factorizations}.

\subsection{Kummer congruences}
\label{subsec:kummer-congruences}

Assume that the coefficients of \(\xi\) lie in \(\mathcal O_K\), and set
\begin{equation}\label{eq:regularized-arithmetic-values}
 A_{\xi,c}(n)
 =(1-p^n)(1-c^{n+1})
 \frac{B_{n+1}(1/2)}{n+1}S_n^{[p]}(\xi).
\end{equation}
By Corollary~\ref{cor:regularized-unit-moments}, these are the moments of the
integral measure \(N_{\xi,c}\).

\begin{corollary}\label{cor:regularized-Kummer}
Let \(a\ge1\), let \(n_1,\ldots,n_r\ge1\) be congruent modulo
\(p-1\), and let \(b_1,\ldots,b_r\in\mathcal O_K\).  If
\[
 \sum_{i=1}^r b_i x^{n_i}\in p^a\mathcal O_K
 \qquad\text{for every }x\in\Zp^\times,
\]
then
\[
 \sum_{i=1}^r b_iA_{\xi,c}(n_i)\in p^a\mathcal O_K.
\]
In particular,
\[
 n\equiv m\pmod{(p-1)p^{a-1}}
 \quad\Longrightarrow\quad
 A_{\xi,c}(n)\equiv A_{\xi,c}(m)\pmod{p^a}.
\]
\end{corollary}

\begin{proof}
Integrate the pointwise divisibility relation against \(N_{\xi,c}\).
For the final assertion, write \(x=\teich(x)\ang x\) and use
\(v_p(u^d-1)\ge v_p(d)+1\) for \(u\in1+p\Zp\).
\end{proof}

\subsection{Complex Mellin setup}
\label{subsec:complex-setup}

Throughout this appendix \(q=e^{-2\hbar}\) with \(\hbar>0\).  For a
function \(f\) in the domain of convergence, set
\begin{equation}
 \mathcal M\{f\}(s)
 =\frac1{\Gamma(s)}\int_0^\infty f(\hbar)\hbar^{s-1}\,d\hbar.
\end{equation}
If \(f(\hbar)=\sum_{n\ge1}a_ne^{-n\hbar}\) and the corresponding
Dirichlet series is absolutely convergent, then
\(
 \mathcal M\{f\}(s)=\sum_{n\ge1}a_nn^{-s}.
\)

For a finite cyclotomic presentation
\(
 \xi=\sum_{z\in S}a_z[z]_{\mathrm{cyc}}
\)
and an embedding of its coefficient field into \(\C\), define
\begin{align}
 \log\widehat\Psi_\xi(t,q)
 &=\sum_{z\in S}a_z\sum_{j\ge0}\log(1-tzq^j),
 \label{eq:general-presentation-Pochhammer-log}\\
 f_\xi(\hbar)
 &=\log\widehat\Psi_\xi(e^{-\hbar},e^{-2\hbar}),
 \qquad
 L_\xi(s)=\mathcal M\{f_\xi\}(s).
 \label{eq:general-complex-basic-path}
\end{align}
The logarithm is the power-series branch
\(\log(1-u)=-\sum_{r\ge1}u^r/r\).  Absolute and locally uniform
convergence follows from
\[
 \sum_{z\in S}|a_z|
 \sum_{r\ge1}\frac{|t|^r}{r(1-|q|^r)}<\infty.
\]
We use
\(
 ue^{xu}/(e^u-1)=\sum_{n\ge0}B_n(x)u^n/n!.
\)

\subsection{Absolute convergence}\label{subsec:complex-convergence}

\begin{lemma}\label{lem:complex-Pochhammer-bound}
Let $|z|=1$ and $\hbar>0$.  Put
\[
 P_z(\hbar)=\log\bigl((e^{-\hbar}z;e^{-2\hbar})_\infty\bigr),
\]
using the power-series branch of the logarithm.  Then
\begin{equation}\label{eq:complex-Pochhammer-absolute-majorant}
 |P_z(\hbar)|
 \le
 \sum_{r\ge1}\frac{e^{-r\hbar}}{r(1-e^{-2r\hbar})}.
\end{equation}
There are absolute constants $C_0,C_\infty>0$ such that
\begin{equation}\label{eq:complex-Pochhammer-endpoint-bounds}
 |P_z(\hbar)|\le \frac{C_0}{\hbar}
 \quad(0<\hbar\le1),
 \qquad
 |P_z(\hbar)|\le C_\infty e^{-\hbar}
 \quad(\hbar\ge1),
\end{equation}
independently of $z$.
\end{lemma}

\begin{proof}
For every $j\ge0$ the number $u=e^{-(2j+1)\hbar}z$ has modulus less than one, hence
\[
 \log(1-u)=-\sum_{r\ge1}\frac{u^r}{r}.
\]
Taking absolute values and summing first over $j$ gives \eqref{eq:complex-Pochhammer-absolute-majorant}.  If $0<\hbar\le1$, then $1-e^{-2r\hbar}\ge (1-e^{-2})\min(1,r\hbar)$.  Splitting the $r$-sum at $r\le \hbar^{-1}$ yields
\[
 \sum_{r\le \hbar^{-1}}\frac1{r^2\hbar}
 +\sum_{r>\hbar^{-1}}\frac{e^{-r\hbar}}r
 \ll \hbar^{-1}.
\]
Indeed, the second term has the explicit bound
\[
 \sum_{r>\hbar^{-1}}\frac{e^{-r\hbar}}r
 \le
 \hbar\frac{e^{-1}}{1-e^{-\hbar}}
 \ll1.
\]
For $\hbar\ge1$, use $1-e^{-2r\hbar}\ge1-e^{-2}$ and
\[
 \sum_{r\ge1}\frac{e^{-r\hbar}}r
 =-\log(1-e^{-\hbar})
 \le\frac{e^{-\hbar}}{1-e^{-1}}.
\]
All constants are independent of the unit-modulus parameter $z$.
\end{proof}

\begin{theorem}\label{thm:complex-absolute-domain}
Let $\xi=\sum_z a_z[z]_{\mathrm{cyc}}$ be a finite complex cyclotomic presentation and let $f_\xi$ be its logarithmic section.  Then
\[
 f_\xi(\hbar)=O_\xi(\hbar^{-1})\quad(\hbar\downarrow0),
 \qquad
 f_\xi(\hbar)=O_\xi(e^{-\hbar})\quad(\hbar\to\infty).
\]
Consequently the Mellin integral defining $L_\xi(s)$ converges absolutely for $\Re(s)>1$.  In that half-plane all logarithmic expansions, finite character sums, infinite sums, and Mellin integrals may be interchanged absolutely.
\end{theorem}

\begin{proof}
The endpoint estimates follow from Lemma~\ref{lem:complex-Pochhammer-bound} after multiplying by the finite coefficient sum $\sum_z|a_z|$.  For the interchange statement, let $\sigma=\Re(s)>1$.  The absolute Mellin integral of the fully expanded series is bounded by
\begin{align*}
 &\sum_z|a_z|\sum_{r\ge1}\sum_{j\ge0}\frac1r
 \int_0^\infty e^{-(2j+1)r\hbar}\hbar^{\sigma-1}\,d\hbar\\
 &\qquad=
 \Gamma(\sigma)
 \sum_z|a_z|
 \left(\sum_{r\ge1}r^{-\sigma-1}\right)
 \left(\sum_{j\ge0}(2j+1)^{-\sigma}\right),
\end{align*}
which is finite precisely for $\sigma>1$.  After multiplication by the
normalizing factor $1/|\Gamma(s)|$, the majorant is
\[
 \frac{\Gamma(\sigma)}{|\Gamma(s)|}
 \sum_z|a_z|
 \left(\sum_{r\ge1}r^{-\sigma-1}\right)
 \left(\sum_{j\ge0}(2j+1)^{-\sigma}\right).
\]
The gamma quotient is finite at every $s$ in this half-plane.  Tonelli's
theorem therefore applies to the absolute majorant, and Fubini's theorem
justifies every rearrangement in the stated domain.
\end{proof}

\begin{remark}\label{rem:complex-continuation-separate}
The Mellin integral is used for $\Re(s)>1$; outside this half-plane, all
values and residues refer to the meromorphic continuation of the resulting
$L$-function product.
\end{remark}

\subsection{Primitive cyclotomic sections}\label{subsec:primitive}

Let \(\chi_N\) be an odd Dirichlet character modulo \(N\), and put
\(\zeta_N=e^{2\pi i/N}\).  Retain the packet and Bloch class from
\eqref{eq:character-packet}, and abbreviate
\[
 \xi:=\xi_{\chi_N,N},
 \qquad
 \eta_{\chi_N}:=\eta_{\chi_N,N}.
\]
We use the same symbols after the fixed embedding
\(\Q(\zeta_N)\hookrightarrow\C\).

The logarithmic identity
\[
 \log(a;q)_\infty=-\sum_{r\ge1}\frac{a^r}{r(1-q^r)}
\]
is classical; in the notation of the local Habiro construction it is
\cite[Eq.~(46)]{GSWZ}.

The general definition specializes to
\begin{equation}\label{eq:f-xi}
f_\xi(\hbar):=\log\widehat\Psi_\xi(e^{-\hbar},e^{-2\hbar})
=\sum_{k\in(\Z/N\Z)^\times}\chi_N(k)\,
\log\!\bigl((e^{-\hbar}\zeta_N^k;e^{-2\hbar})_\infty\bigr).
\end{equation}

\begin{lemma}\label{lem:gauss}
  
Assume $\chi_N$ is primitive modulo $N$ and odd, and use the Gauss-sum
normalization \eqref{eq:gauss-sum-normalization}.  We extend $\chi_N$ and
$\chi_N^\vee$ to $\Z$ by setting them equal to $0$ when $(d,N)>1$.
Then for every integer $d$,
\[
\sum\limits_{k\in(\Z/N\Z)^\times}\chi_N(k)\,\zeta_N^{kd}=G(\chi_N)\,\chi_N^\vee(d).
\]
\end{lemma}

\begin{proof}
If $(d,N)=1$, multiplication by $d^{-1}$ permutes $(\Z/N\Z)^\times$, so we may substitute $k\mapsto kd^{-1}$ inside the unit group:
\[
\sum\limits_{k\in(\Z/N\Z)^\times}\chi_N(k)\zeta_N^{kd}
=\sum\limits_{k\in(\Z/N\Z)^\times}\chi_N(kd^{-1})\zeta_N^{k}
=\chi_N^\vee(d)\sum\limits_{k\in(\Z/N\Z)^\times}\chi_N(k)\zeta_N^{k}
=G(\chi_N)\chi_N^\vee(d).
\]

If $(d,N)>1$, pick a prime $\ell\mid(d,N)$ and write $N=\ell N'$.
Since $\chi_N$ is primitive modulo $N$, there exists a unit $u\in(\Z/N\Z)^\times$ such that
$u\equiv 1\pmod{N'}$ but $\chi_N(u)\neq 1$.
Indeed, otherwise $\chi_N$ would be trivial on the kernel of reduction
$(\Z/N\Z)^\times\to(\Z/N'\Z)^\times$ and would descend to modulus
$N'$, contradicting that its conductor is $N$.
For every $k$ we have $\zeta_N^{u k d}=\zeta_N^{k d}$ (because $(u-1)kd$ is a multiple of $N$), hence
\[
S:=\sum\limits_{k\in(\Z/N\Z)^\times}\chi_N(k)\zeta_N^{kd}
=\sum\limits_{k\in(\Z/N\Z)^\times}\chi_N(uk)\zeta_N^{u k d}
=\chi_N(u)\,S.
\]
Thus $S=0$, which agrees with the right-hand side since $\chi_N^\vee(d)=0$.
\end{proof}

\begin{theorem}[Complex Mellin factorization for primitive packets]\label{thm:primitive-factor}
Assume that $\chi_N$ is primitive modulo $N$ and odd. Then
\[
f_\xi(\hbar)=\sum\limits_{n=1}^\infty a_n(\xi)\,e^{-n\hbar},
\qquad
a_n(\xi)=-G(\chi_N)\sum\limits_{\substack{d\mid n\\ n/d\ \mathrm{odd}}}\frac{\chi_N^\vee(d)}{d}.
\]
Moreover, for $\Re(s)>1$,
\[
L_\xi(s)=-G(\chi_N)\,\tilde\zeta(s)\,L(s+1,\chi_N^\vee),
\qquad
\tilde\zeta(s)=(1-2^{-s})\zeta(s)=\sum\limits_{\substack{m\ge1\\ m\ \mathrm{odd}}}m^{-s}.
\]
\end{theorem}

\begin{proof}
Expanding the logarithm and summing the geometric series gives
\[
 f_\xi(\hbar)
 =-
 \sum_{r\ge1}\sum_{j\ge0}\frac{e^{-(2j+1)r\hbar}}{r}
 \sum_k\chi_N(k)\zeta_N^{kr}.
\]
By Lemma~\ref{lem:gauss}, the inner sum is
$G(\chi_N)\chi_N^\vee(r)$.  Writing $n=(2j+1)r$ gives
\[
 a_n(\xi)
 =-G(\chi_N)
 \sum_{\substack{r\mid n\\ n/r\ \mathrm{odd}}}
 \frac{\chi_N^\vee(r)}{r}.
\]
For $\Re(s)>1$, termwise Mellin transformation then yields
\[
 L_\xi(s)
 =-G(\chi_N)L(s+1,\chi_N^\vee)
   \sum_{\substack{m\ge1\\m\ \mathrm{odd}}}m^{-s},
\]
which is the required factorization.
\end{proof}

\begin{remark}\label{rem:Fantini-Rella-comparison}
Fantini--Rella obtain $\zeta(s)L(s+1,\bar\chi)$ for the unshifted
transform \cite[Lem.~4.6 and Eqs.~(105), (111)]{FantiniRella}.  Here the
half shift restricts the outer Dirichlet variable to odd integers,
replacing $\zeta(s)$ by $(1-2^{-s})\zeta(s)$.
\end{remark}

\subsection{Imprimitive cyclotomic sections}\label{subsec:imprimitive}

Let $\chi_N$ be an odd Dirichlet character modulo $N$, induced from an odd primitive character $\chi_f$ of conductor $f$, and write
\[
 N=fg.
\]

Retain the finite correction \(A_{\chi_f,g}(d)\) from
\eqref{eq:general-imprimitive-correction}.  When \((f,g)=1\), it becomes
\[
 A_{\chi_f,g}(d)=\chi_f(g)\chi_f^\vee(d)C_g(d),
 \qquad
 C_g(d):=\sum_{\substack{v\;(\mathrm{mod}\;g)\\(v,g)=1}}e^{2\pi i vd/g},
\]
by the usual formula
\(C_g(d)=\sum_{u\mid(g,d)}u\,\mob(g/u)\).

\begin{lemma}\label{lem:ramanujan}
For every integer $d$ one has
\[
\sum_{k\in(\Z/N\Z)^\times}\chi_N(k)\,\zeta_N^{kd}
=
G(\chi_f)\,A_{\chi_f,g}(d).
\]
In particular, if $(f,g)=1$, then
\[
\sum_{k\in(\Z/N\Z)^\times}\chi_N(k)\,\zeta_N^{kd}
=
G(\chi_f)\chi_f(g)\chi_f^\vee(d)C_g(d).
\]
\end{lemma}

\begin{proof}
The condition $(k,N)=1$ is equivalent to $(k,f)=1$ and $(k,g)=1$.  Since $\chi_N$ is induced from $\chi_f$,
\[
\sum_{k\in(\Z/N\Z)^\times}\chi_N(k)\zeta_N^{kd}
=
\sum_{\substack{k\;(\mathrm{mod}\;fg)\\(k,f)=1}}
\chi_f(k)e^{2\pi i kd/(fg)}
\sum_{\substack{v\mid g\\ v\mid k}}\mu_{\mathrm{Mob}}(v).
\]
If $v$ has a prime factor in common with $f$, then the inner condition $v\mid k$ is incompatible with $(k,f)=1$.  Hence only divisors $v$ of $g$ with $(v,f)=1$ contribute.  Put $v=g/u$; then $u\mid g$ and $(g/u,f)=1$.  For such $u$ write $k=(g/u)b$.  The corresponding contribution is
\[
\mu_{\mathrm{Mob}}(g/u)\chi_f(g/u)
\sum_{\substack{b\;(\mathrm{mod}\;fu)\\(b,f)=1}}
\chi_f(b)e^{2\pi i bd/(fu)}.
\]
Write $b=a+ft$ with $a\in(\Z/f\Z)^\times$ and $0\le t<u$.  The sum over $t$ is $u$ if $u\mid d$ and is $0$ otherwise.  When $u\mid d$, the remaining sum over $a$ is
\[
u\sum_{a\in(\Z/f\Z)^\times}\chi_f(a)e^{2\pi i a(d/u)/f}
=
u\,G(\chi_f)\chi_f^\vee(d/u)
\]
by the primitive Gauss identity.  Summing over $u$ gives \eqref{eq:general-imprimitive-correction}.  If $(f,g)=1$, the identity
\[
\chi_f(g/u)\chi_f^\vee(d/u)=\chi_f(g)\chi_f^\vee(d)
\]
for $u\mid d$ gives the displayed Ramanujan-sum specialization.
\end{proof}

\begin{definition}[Finite imprimitive Euler correction]\label{def:Lg}
For $N=fg$ as above, set
\begin{equation}\label{eq:finite-imprimitive-euler-correction}
C^{\mathrm{Mell}}_{\chi_f,g}(s)
:=
\sum_{\substack{u\mid g\\(g/u,f)=1}}
\mu_{\mathrm{Mob}}(g/u)\chi_f(g/u)u^{-s}.
\end{equation}
Thus, when $(f,g)=1$,
\[
C^{\mathrm{Mell}}_{\chi_f,g}(s)
=
\chi_f(g)\sum_{u\mid g}\mu_{\mathrm{Mob}}(g/u)\chi_f^\vee(u)u^{-s}.
\]
\end{definition}

\begin{theorem}[Complex Mellin factorization for imprimitive packets]\label{thm:imprimitive-factor}
Let $\chi_N$ be induced from the odd primitive character $\chi_f$ of
conductor $f$, write $N=fg$, and put $\xi=\xi_{\chi_N,N}$.  Then, for
$\Re(s)>1$,
\[
L_\xi(s)=-G(\chi_f)\,\tilde\zeta(s)\,C^{\mathrm{Mell}}_{\chi_f,g}(s)\,L(s+1,\chi_f^\vee).
\]
If $(f,g)=1$, this is equivalently
\[
L_\xi(s)
=
-G(\chi_f)\chi_f(g)\tilde\zeta(s)L(s+1,\chi_f^\vee)
\sum_{u\mid g}\mu_{\mathrm{Mob}}(g/u)\chi_f^\vee(u)u^{-s}.
\]
\end{theorem}

\begin{proof}
The computation in Theorem~\ref{thm:primitive-factor} remains valid up to the finite character sum.  Lemma~\ref{lem:ramanujan} gives
\[
f_\xi(\hbar)
=
-G(\chi_f)\sum_{r\ge1}\sum_{j\ge0}
\frac{A_{\chi_f,g}(r)}{r}\,e^{-(2j+1)r\hbar}.
\]
Taking the normalized Mellin transform gives
\[
L_\xi(s)
=
-G(\chi_f)\tilde\zeta(s)
\sum_{r\ge1}\frac{A_{\chi_f,g}(r)}{r^{s+1}}
\qquad(\Re(s)>1).
\]
Using \eqref{eq:general-imprimitive-correction} and rearranging absolutely convergent sums,
\begin{align*}
\sum_{r\ge1}\frac{A_{\chi_f,g}(r)}{r^{s+1}}
&=
\sum_{\substack{u\mid g\\(g/u,f)=1}}
u\,\mu_{\mathrm{Mob}}(g/u)\chi_f(g/u)
\sum_{\substack{r\ge1\\u\mid r}}
\frac{\chi_f^\vee(r/u)}{r^{s+1}}  \\
&=
\sum_{\substack{u\mid g\\(g/u,f)=1}}
\mu_{\mathrm{Mob}}(g/u)\chi_f(g/u)u^{-s}
\sum_{t\ge1}\frac{\chi_f^\vee(t)}{t^{s+1}} \\
&=C^{\mathrm{Mell}}_{\chi_f,g}(s)L(s+1,\chi_f^\vee).
\end{align*}
\end{proof}

\begin{example}\label{ex:modulus-six-imprimitive-correction}
Let $\chi_3$ be the nontrivial character modulo $3$ and induce it to
modulus $N=6$, so $f=3$ and $g=2$.  The two admissible divisors in
\eqref{eq:finite-imprimitive-euler-correction} are $u=1,2$.  Since
\[
 \mu_{\mathrm{Mob}}(2)\chi_3(2)=(-1)(-1)=1,
 \qquad
 \mu_{\mathrm{Mob}}(1)\chi_3(1)=1,
\]
one obtains
\[
 C^{\mathrm{Mell}}_{\chi_3,2}(s)=1+2^{-s}.
\]
At the coefficient level, Lemma~\ref{lem:ramanujan} reads
\[
 A_{\chi_3,2}(r)
 =\chi_3(r)+2\one_{2\mid r}\chi_3(r/2).
\]
Substituting the finite factor into Theorem~\ref{thm:imprimitive-factor}
gives the figure-eight formula used in
Subsection~\ref{subsec:complex-examples}.
\end{example}

\begin{example}
\label{ex:modulus-eight-noncoprime-correction}
Let $\chi_4$ be the primitive odd character of conductor $f=4$ and
induce it to modulus $N=8$, so $g=2$ and $(f,g)\ne1$.  In
\eqref{eq:general-imprimitive-correction}, the divisor $u=1$ is excluded
because $(2,4)\ne1$, whereas $u=2$ contributes.  Hence
\[
 A_{\chi_4,2}(d)
 =
 \begin{cases}
 2\chi_4(d/2),&2\mid d,\\
 0,&2\nmid d,
 \end{cases}
 \qquad
 C^{\mathrm{Mell}}_{\chi_4,2}(s)=2^{-s}.
\]
\end{example}

\subsection[Residues at s=1 and dilogarithms]{Residues at $s=1$ and dilogarithms}\label{subsec:complex-reg}

The residue at $s=1$ is governed by dilogarithmic data.  Let $D(z)$ denote
the Bloch--Wigner dilogarithm.  Its coefficient-linear extension is
\[
 D_E=D\otimes1:
 B(\C)\otimes_\Z E\longrightarrow\R\otimes_\Z E.
\]
For the Bloch-class component of the chosen pair set
\[
D_E(\eta_{\chi_N})
:=\sum_{k\in(\Z/N\Z)^\times}\chi_N(k)\,D(\zeta_N^k).
\]
Since $|\zeta_N^k|=1$, one has $D(\zeta_N^k)=\Im(\Li_2(\zeta_N^k))$.

\begin{lemma}\label{lem:dilog-sum}
Let $\chi_N$ be induced from the odd primitive character $\chi_f$ of conductor $f$, and write $N=fg$.  Then
\[
\sum_{k\in(\Z/N\Z)^\times}\chi_N(k)\,\Li_2(\zeta_N^k)
=
G(\chi_f)\,C^{\mathrm{Mell}}_{\chi_f,g}(1)L(2,\chi_f^\vee).
\]
Equivalently, in the conductor-coprime case $(f,g)=1$, this equals
\[
G(\chi_f)\chi_f(g)
L(2,\chi_f^\vee)
\sum_{u\mid g}\mu_{\mathrm{Mob}}(g/u)\chi_f^\vee(u)u^{-1}.
\]
\end{lemma}

\begin{proof}
Expand $\Li_2(\zeta_N^k)=\sum_{n\ge1}\zeta_N^{kn}/n^2$ and interchange sums:
\[
\sum_k\chi_N(k)\Li_2(\zeta_N^k)
=
\sum_{n\ge1}\frac{1}{n^2}\sum_k\chi_N(k)\zeta_N^{kn}.
\]
Lemma~\ref{lem:ramanujan} gives
\[
\sum_k\chi_N(k)\Li_2(\zeta_N^k)
=
G(\chi_f)\sum_{n\ge1}\frac{A_{\chi_f,g}(n)}{n^2}.
\]
The computation in the proof of Theorem~\ref{thm:imprimitive-factor}, evaluated at $s=1$, identifies the last Dirichlet series with $C^{\mathrm{Mell}}_{\chi_f,g}(1)L(2,\chi_f^\vee)$.
\end{proof}

\begin{theorem}[Residue and regulator]\label{thm:complex-residue}
Let
\[
\xi=\sum_{k\in(\Z/N\Z)^\times}\chi_N(k)[\zeta_N^k]_{\mathrm{cyc}}
\in\CycPres(\Q(\zeta_N);E_{\chi_N}),
\]
and assume that $\chi_N$ is induced from an odd primitive character $\chi_f$ of conductor $f$, with $N=fg$.  Then
\[
\Res_{s=1}L_\xi(s)
=
-\frac12\,G(\chi_f)C^{\mathrm{Mell}}_{\chi_f,g}(1)L(2,\chi_f^\vee)
=
-\frac12\sum_{k\in(\Z/N\Z)^\times}\chi_N(k)\Li_2(\zeta_N^k).
\]
If, in addition, $\chi_N$ is real-valued, then
\[
\Res_{s=1}L_\xi(s)=-\frac{i}{2}\,D_E(\eta_{\chi_N}).
\]
\end{theorem}

\begin{proof}
By Theorem~\ref{thm:imprimitive-factor},
\[
L_\xi(s)=-G(\chi_f)\tilde\zeta(s)C^{\mathrm{Mell}}_{\chi_f,g}(s)L(s+1,\chi_f^\vee).
\]
Since $\Res_{s=1}\tilde\zeta(s)=1/2$ and the remaining factors are holomorphic at $s=1$, the first equality follows.  The second equality is Lemma~\ref{lem:dilog-sum}.

If $\chi_N$ is odd and real-valued, the terms indexed by $k$ and $-k$ show that
\[
\sum_k\chi_N(k)\Li_2(\zeta_N^k)
=i\sum_k\chi_N(k)D(\zeta_N^k)
=iD_E(\eta_{\chi_N}),
\]
because the real parts cancel and the imaginary parts add with the odd character sign.  Substitution gives the stated regulator formula.
\end{proof}

If $\chi_N$ is primitive and odd, then analytic continuation of
Theorem~\ref{thm:primitive-factor} also gives, for every $n\ge1$,
\[
L_\xi(1-2n)
=
2^{2n-1}\frac{B_{2n}(1/2)}{2n}
\sum_{k\in(\Z/N\Z)^\times}\chi_N(k)\,\Li_{2-2n}(\zeta_N^k).
\]
Indeed,
\[
L_\xi(1-2n)=\tilde\zeta(1-2n)\cdot \bigl(-G(\chi_N)L(2-2n,\chi_N^\vee)\bigr).
\]
Moreover
\[
\tilde\zeta(1-2n)=-2^{2n-1}\frac{B_{2n}(1/2)}{2n}
\]
and
\[
\sum_{k\in(\Z/N\Z)^\times}\chi_N(k)\Li_s(\zeta_N^k)=G(\chi_N)L(s,\chi_N^\vee)
\]
by analytic continuation of the primitive Gauss-sum identity.

\subsection{The figure-eight knot}\label{subsec:complex-examples}

Let \(\chi_{-3}\) be the primitive odd character of conductor \(3\), and
consider its induction to modulus \(6\):
\[
 \xi_{4_1}
 =\sum_{k\in(\Z/6\Z)^\times}
 \chi_{-3}(k)[\zeta_6^k]_{\mathrm{cyc}}.
\]
The imprimitive factorization gives
\[
 L_{\xi_{4_1}}(s)
 =-G(\chi_{-3})(1+2^{-s})(1-2^{-s})\zeta(s)
   L(s+1,\chi_{-3}^{-1}).
\]
The figure-eight complement is obtained by gluing two positively
oriented regular ideal tetrahedra, both of shape \(\zeta_6\),
\cite[Ex.~4.8 and Fig.~4.10]{Purcell}, so
\[
 \vol(S^3\setminus4_1)=2D(\zeta_6)
 =\frac{3\sqrt3}{2}L(2,\chi_{-3}^{-1}).
\]
Theorem~\ref{thm:complex-residue} therefore yields
\[
 \Res_{s=1}L_{\xi_{4_1}}(s)
 =-\frac{i}{2}\vol(S^3\setminus4_1).
\]

\subsection{Radial branches at roots of unity}\label{subsec:complex-higher-branches}

Let \(\lambda\) be a root of exact order \(m\), and choose a square-root
lift \(\widetilde\lambda\) as in Subsection~\ref{subsec:root-conventions}.
Put
\[
 q_\lambda(\hbar)=\lambda e^{-2\hbar/m},
 \qquad
 t_\lambda(\hbar)=\widetilde\lambda e^{-\hbar/m}.
\]
For a packet \(\xi\), define
\[
 f_{\xi,\lambda}(\hbar)
 =\log\widehat\Psi_\xi(t_\lambda(\hbar),q_\lambda(\hbar)),
 \qquad
 L_{\xi,\lambda}(s)=\mathcal M\{f_{\xi,\lambda}\}(s).
\]

Let \(\chi_N\) be induced from an odd primitive character \(\chi_f\) of
conductor \(f\), where \(N=fg\), and set
\[
 B(r)=
 \begin{cases}
  \chi_N^{-1}(r),&g=1,\\
  A_{\chi_f,g}(r),&g>1,
 \end{cases}
 \qquad
 C=G(\chi_f).
\]
Define
\[
 \Theta_\lambda(r,s)
 =\sum_{\substack{n\ge1\\ n\text{ odd}}}
   \frac{\widetilde\lambda^{nr}}{n^s}
 =\Li_s(\widetilde\lambda^r)-2^{-s}\Li_s(\widetilde\lambda^{2r}).
\]

\begin{proposition}[Radial Mellin expansion]
\label{prop:complex-higher-branch}
For \(\Re(s)>1\),
\begin{align}
 f_{\xi,\lambda}(\hbar)
 &=-C\sum_{r\ge1}\sum_{\substack{n\ge1\\n\text{ odd}}}
   \frac{B(r)\widetilde\lambda^{nr}}{r}e^{-nr\hbar/m},
 \label{eq:radial-log-expansion}\\
 L_{\xi,\lambda}(s)
 &=-C\,m^s\sum_{r\ge1}
   \frac{B(r)}{r^{s+1}}\Theta_\lambda(r,s).
 \label{eq:radial-Mellin-expansion}
\end{align}
At \(\lambda=1\), these formulas recover the primitive or imprimitive
factorization at the basic center.
\end{proposition}

\begin{proof}
Writing the Pochhammer index as \(mj+\ell\) gives
\[
 (t_\lambda\zeta_N^k;q_\lambda)_\infty
 =\prod_{\ell=0}^{m-1}
 (\widetilde\lambda^{2\ell+1}e^{-(2\ell+1)\hbar/m}
  \zeta_N^k;e^{-2\hbar})_\infty.
\]
After expansion, the integers \((2\ell+1)+2mj\) run through the positive
odd integers.  The primitive Gauss identity or its imprimitive correction
evaluates the finite sum over \(k\), giving
\eqref{eq:radial-log-expansion}.  The double series is absolutely
Mellin-integrable for \(\Re(s)>1\), and termwise integration gives
\eqref{eq:radial-Mellin-expansion}.
\end{proof}

\begin{theorem}[Finite Hurwitz--zeta continuation]
\label{thm:radial-meromorphic-continuation}
Let \(H\) be the order of \(\widetilde\lambda\), and put
\(L=\operatorname{lcm}(N,H)\).  Then
\begin{equation}\label{eq:radial-Hurwitz-continuation}
 L_{\xi,\lambda}(s)
 =-C\,m^sL^{-s-1}
 \sum_{a=1}^{L}
 B(a)\Theta_\lambda(a,s)
 \zeta\!\left(s+1,\frac aL\right).
\end{equation}
This formula extends \(L_{\xi,\lambda}\) meromorphically to \(\C\), with
possible simple poles only at \(s=0\) and \(s=1\).  The continuation
depends on the chosen lift \(\widetilde\lambda\).
\end{theorem}

\begin{proof}
The primitive function \(B\) is periodic modulo \(N\).  In the
imprimitive case this follows from the divisor formula for
\(A_{\chi_f,g}\): every divisor occurring in that formula divides
\(N=fg\), so replacing \(r\) by \(r+N\) preserves both the divisibility
conditions and the character values.  Since
\(\Theta_\lambda(r,s)\) depends on \(r\) only through
\(\widetilde\lambda^r\), the product \(B(r)\Theta_\lambda(r,s)\) is periodic
modulo \(L\).  Grouping \eqref{eq:radial-Mellin-expansion} by residue
classes modulo \(L\) gives \eqref{eq:radial-Hurwitz-continuation}.  Moreover,
\[
 \Li_s(\widetilde\lambda^a)
 =H^{-s}\sum_{b=1}^{H}
  \widetilde\lambda^{ab}\zeta\!\left(s,\frac bH\right).
\]
The two Hurwitz factors have possible simple poles only at \(s=1\) and
\(s=0\), respectively.  This proves the continuation and the stated pole
set.
\end{proof}

\section{Higher GSWZ germs and Frobenius gluing}\label{sec:higher-gswz}

Let \(k\) be a number field, let
\(\eta^{\mathrm{glob}}\in K_3(k)\), and fix an unramified local factor
\(K_p/\Qp\) at \(p>3\).  Denote the local image by
\[
 \eta_p\in K_3(K_p;\Zp).
\]
We project the notation of \cite{GSWZ} to this factor.  Its coefficient
ring is denoted by \(R_p^\wedge\); it is \(p\)-torsion-free and
\(p\)-adically complete, and it carries the Frobenius lift \(\phi_p\).

For every integer \(m\) prime to \(p\), the mod-\(m\) étale Chern
character used in \cite{GSWZ} associates to a \(K_3\)-class \(\theta\) a
\(\mu_m\)-torsor.  Following their unit notation, write this torsor as
\(\varepsilon_m(\theta)^{1/m}\).  For a root-unit class \([z]\), the
expression \(\varepsilon_m([z])^{1/m}\) denotes the corresponding
constant torsor factor.  Write
\(\mathcal L_p^{\mathrm{GSWZ}}(\eta^{\mathrm{glob}})\) for the object denoted
\(L_p(\eta^{\mathrm{glob}})\) in \cite{GSWZ}.  An invertible
\(\mathcal L_p^{\mathrm{GSWZ}}(\eta^{\mathrm{glob}})\)-section is a compatible
collection of local germs whose \(m\)-th component belongs to
\[
 \varepsilon_m(\eta^{\mathrm{glob}})^{1/m}
 \bigl(R_p^\wedge[\zeta_m]^\times
       +xK_p[\zeta_m][[x]]\bigr)
\]
and satisfies the Frobenius gluing relation below.

Using the local regulator theorem and the comparison
\eqref{eq:local-Bloch-K3-comparison}, write a projected root-unit
presentation as a finite sum
\[
 \widehat\xi=\sum_\zeta a_\zeta[\zeta]_{\mathrm{cyc}},
 \qquad a_\zeta\in\Zp,
 \qquad
 \kappa_{K_p,p}\bigl(\operatorname{cl}(\widehat\xi)\bigr)=\eta_p.
\]
It is obtained from a presentation satisfying the global-image condition
\cite[(205)--(206)]{GSWZ}.

At a center \(\zeta_m\) of order \(m\) prime to \(p\), write
\[
 q=\zeta_m+x=\zeta_m(1+T),
 \qquad
 Y=\log(1+T),
 \qquad
 q^{m/2}=(1+T)^{m/2}.
\]
For a prime-to-\(p\) root \(z\ne1\), the explicit GSWZ branch is
\begin{equation}\label{eq:gswz-single-branch}
 \Psi_{[z],p,m}(x)
 =\exp\!\left(-\frac{\Li_{2,p}(z)}{m^2\log_pq}\right)
  \varepsilon_m([z])^{1/m}
  \qpoch{q^{m/2}z}{q^m}^{1/m}.
\end{equation}
The torsor factor is constant in the local coordinate.  The completed
logarithm of an invertible section has the form
\begin{equation}\label{eq:GSWZ-completed-log}
 \log\widehat f_m(x)
 =\frac{D_p(\eta_p)}{m^2\log q}+\log f_m(x)
\end{equation}
and satisfies
\begin{equation}\label{eq:GSWZ-gluing}
 \log\!\left(
  \frac{\phi_p\widehat f(q^p)}{\widehat f(q)^p}
 \right)
 \in\prod_{(m,p)=1}\frac p xR_p^\wedge[\zeta_m][[x]].
\end{equation}
Here \(\phi_p(\zeta_m)=\zeta_m^p\) and
\(x_p=(\zeta_m+x)^p-\zeta_m^p\).

The following is \cite[Thm.~10, Lem.~3.6, and Cor.~3.7]{GSWZ} after
projection to the chosen local factor.

\begin{theorem}[Garoufalidis--Scholze--Wheeler--Zagier]
\label{thm:GSWZ-section}
Assume the global-image and full-product root-presentation hypotheses of
\cite[(205)--(206)]{GSWZ}.  The product of
\eqref{eq:gswz-single-branch} with exponents \(a_\zeta\) from the
projected presentation \(\widehat\xi\) is an invertible
\(\mathcal L_p^{\mathrm{GSWZ}}(\eta^{\mathrm{glob}})\)-section.  It has the unique
all-root extension of \cite[Lem.~3.6 and Cor.~3.7]{GSWZ} and satisfies
\eqref{eq:GSWZ-gluing}.
\end{theorem}

\begin{proposition}\label{prop:GSWZ-formal-normalization}
Let \(\widehat\xi\) be the projected integral root presentation above.  If
\[
 \log\widehat f_m(T)
 =\frac{D_p(\eta_p)}{m^2Y}+h_m(T),
\]
then
\begin{equation}\label{eq:GSWZ-formal-branch-compatibility}
 h_m(T)-h_m(0)=F_{\widehat\xi,p,m}(T).
\end{equation}
\end{proposition}

\begin{proof}
The Laurent expansion of the Pochhammer factor in
\eqref{eq:gswz-single-branch} is
\eqref{eq:higher-formal-coefficient-expansion}.  The torsor contributes
only a constant, while the regular coefficients are those of
\eqref{eq:normalized-single-branch-log}.  Subtracting the constant term
gives the formula.
\end{proof}

\subsection{Order-corrected regularization}
\label{subsec:regularized-Dwork}

Use the all-root extension to write \(g_\zeta\) for the completed germ at
each prime-to-\(p\) root \(\zeta\).  If \(\zeta\) has exact order \(m\)
and \(c\) is prime to \(p\), put
\[
 d=(m,c),
 \qquad
 m_c=\ord(\zeta^c)=m/d,
 \qquad
 x_c=(\zeta+x)^c-\zeta^c,
\]
and define
\begin{equation}\label{eq:order-corrected-regularizer}
 (\Regord_cg)_\zeta(x)
 =g_\zeta(x)-\frac{c}{d^2}g_{\zeta^c}(x_c).
\end{equation}
The logarithmic Frobenius difference is
\begin{equation}\label{eq:Frobenius-difference}
 (\nabla_pg)_\zeta(x)
 =(\phi_pg_\zeta)(x_p)-pg_\zeta(x).
\end{equation}

\begin{proposition}[Forced order correction]
\label{prop:order-corrected-pole-cancellation}
Suppose that the polar part at a center of exact order \(m\) is
\[
 P_\zeta(q)=\frac{A}{m^2\log q},
\]
where \(A\) is independent of the center.  Then
\(
 \Regord_cP=0.
\)
Among expressions
\(
 g_\zeta(q)-\lambda_{m,c}g_{\zeta^c}(q^c),
\)
the coefficient \(\lambda_{m,c}=c/d^2\) is the unique one that cancels
the pole for every \(A\).  It is a \(p\)-adic unit when \(p\nmid mc\).
\end{proposition}

\begin{proof}
The target pole equals
\[
 \frac{A}{(m/d)^2\log(q^c)}
 =\frac{A}{(m/d)^2c\log q}.
\]
Multiplication by \(c/d^2\) gives \(A/(m^2\log q)\), and equality of the
polar coefficients proves uniqueness.
\end{proof}

\begin{proposition}
\label{prop:order-corrected-Frobenius-commutation}
For every branch collection on which the two operators are defined,
\begin{equation}\label{eq:regularization-Frobenius-commutes}
 \nabla_p\Regord_c=\Regord_c\nabla_p.
\end{equation}
\end{proposition}

\begin{proof}
The identities
\(
 (\zeta^p)^c=(\zeta^c)^p
\)
and
\(
 (q^p)^c=(q^c)^p
\)
show that the coordinate substitutions commute.  Since \((p,m)=1\), the
integer \(d=(m,c)\) is unchanged by Frobenius, and
\(\phi_p(c/d^2)=c/d^2\).  Expanding the two composites gives the same four
terms.
\end{proof}

Let \(p>2\), and let \(R_0\) be a \(p\)-torsion-free,
\(p\)-adically complete \(\Zp\)-algebra with a continuous Frobenius lift.
Extend Frobenius to cyclotomic coefficient rings by
\(\phi_p(\zeta)=\zeta^p\).  For every prime-to-\(p\) root \(\zeta\) and
every \(c\) prime to \(p\), assume that
\[
 R_0[\zeta^c]\longrightarrow R_0[\zeta]
\]
is Frobenius compatible.  With \(x=q-\zeta\) and
\(x_c=q^c-\zeta^c\), assume that power pullback sends
\[
 R_0[\zeta^c][[x_c]]
 \quad\text{into}\quad
 R_0[\zeta][[x]]
\]
and
\[
 \frac{p}{x_c}R_0[\zeta^c][[x_c]]
 \quad\text{into}\quad
 \frac{p}{x}R_0[\zeta][[x]].
\]

\begin{theorem}[Order-corrected Dwork gluing]\label{thm:higher-Dwork}
Assume the Frobenius and power-pullback hypotheses above.  Suppose
\begin{equation}\label{eq:abstract-completed-polar-form}
 \log\widehat f_\zeta(q)
 =\frac{A}{m^2\log q}+h_\zeta(q),
 \qquad
 h_\zeta(q)\in R_0[\zeta][1/p][[q-\zeta]],
\end{equation}
where $A$ is independent of the center, and suppose
\[
 (\nabla_p\log\widehat f)_\zeta
 \in\frac p xR_0[\zeta][[x]].
\]
For $H=\Regord_c(\log\widehat f)$, one has
\begin{equation}\label{eq:regularized-Dwork-log}
 (\nabla_pH)_\zeta
 \in pR_0[\zeta][[x]].
\end{equation}
Put $H_\zeta^0=H_\zeta-H_\zeta(0)$.  Then
\begin{equation}\label{eq:regularized-Dwork-normalized}
 (\nabla_pH^0)_\zeta
 \in pxR_0[\zeta][[x]],
\end{equation}
and the associated constant-normalized Frobenius quotient belongs to
$1+pxR_0[\zeta][[x]]$.
\end{theorem}

\begin{proof}
Put \(R=R_0[\zeta]\).  Proposition~\ref{prop:order-corrected-pole-cancellation}
makes \(H_\zeta\) regular.  Since
\[
 x_c=x\bigl(c\zeta^{c-1}+O(x)\bigr)
\]
has unit linear coefficient, power pullback preserves
\(px^{-1}R[[x]]\).  By
Proposition~\ref{prop:order-corrected-Frobenius-commutation},
\[
 \nabla_pH=\Regord_c\nabla_p(\log\widehat f).
\]
The right side lies in \(px^{-1}R[[x]]\), whereas regularity places the
left side in \(R[1/p][[x]]\).  Hence
\[
 R[1/p][[x]]\cap px^{-1}R[[x]]=pR[[x]],
\]
which gives \eqref{eq:regularized-Dwork-log}.  Subtracting the constant
term gives \eqref{eq:regularized-Dwork-normalized}; exponentiation is
valid because \(p>2\).
\end{proof}

\begin{corollary}
\label{cor:GSWZ-order-corrected-gluing}
Under the hypotheses of Theorem~\ref{thm:GSWZ-section}, apply
Theorem~\ref{thm:higher-Dwork} with
\(R_0=R_p^\wedge\) and \(A=D_p(\eta_p)\).
\end{corollary}

\begin{proof}
Equation~\eqref{eq:GSWZ-completed-log} gives the required pole at the
distinguished centers.  The equivariant all-root extension preserves the
same coefficient at every center, and \eqref{eq:GSWZ-gluing} gives the
Frobenius hypothesis.  The remaining pullback conditions follow from the
unit linear coefficient of \(x_c\).
\end{proof}

At the fixed center, write
\[
 g_1(x)=\frac{A}{\log(1+x)}+h_1(x),
 \qquad A=D_p(\eta_p).
\]
For odd \(c>1\) prime to \(p\), set
\[
 x_c=(1+x)^c-1,
 \qquad
 H_{1,c}=g_1(x)-cg_1(x_c),
 \qquad
 U_{1,c}=\exp\bigl(H_{1,c}-H_{1,c}(0)\bigr).
\]

\begin{proposition}
\label{prop:fixed-higher-compatibility}
The function $H_{1,c}$ is regular and
\begin{equation}\label{eq:fixed-higher-logarithmic-compatibility}
 H_{1,c}-H_{1,c}(0)
 =F_{\widehat\xi,p}(x)-cF_{\widehat\xi,p}(x_c)
 =\Ami_{\lambda_{1/2,c}\starx\nu_{\widehat\xi}}(x).
\end{equation}
Its restriction to $\Zp^\times$ is
\begin{equation}\label{eq:fixed-higher-unit-compatibility}
 N_{\widehat\xi,c}=r_c\starx\IwPs_{\widehat\xi,p}.
\end{equation}
Moreover,
\begin{equation}\label{eq:fixed-higher-Dwork-quotient}
 \frac{(\phi_pU_{1,c})(x_p)}{U_{1,c}(x)^p}
 \in1+pxR_p^\wedge[[x]].
\end{equation}
\end{proposition}

\begin{proof}
The two polar terms cancel because
\(\log((1+x)^c)=c\log(1+x)\).  Proposition~\ref{prop:GSWZ-formal-normalization}
and Theorem~\ref{thm:regularized-convolution} give
\eqref{eq:fixed-higher-logarithmic-compatibility}; restriction to the unit
group gives \eqref{eq:fixed-higher-unit-compatibility}.  The final
assertion follows from
Corollary~\ref{cor:GSWZ-order-corrected-gluing} at the fixed center.
\end{proof}

\begin{lemma}\label{lem:Dwork-coefficient-criterion}
Let \(R\) be a \(p\)-torsion-free, \(p\)-adically complete
\(\Zp\)-algebra with Frobenius lift, and put \(x_p=(1+x)^p-1\).  If
\[
 B(x)\in1+xR[1/p][[x]],
 \qquad
 \frac{(\phi_pB)(x_p)}{B(x)^p}\in1+pxR[[x]],
\]
then
\[
 B(x)\in1+xR[[x]]
 \quad\Longleftrightarrow\quad
 [x]B(x)\in R.
\]
\end{lemma}

\begin{proof}
Write \(B=1+\sum_{n\ge1}b_nx^n\) and compare coefficients in
\((\phi_pB)(x_p)=B^pQ\), where \(Q\in1+pxR[[x]]\).  Assuming
\(b_1,\ldots,b_{n-1}\in R\), the coefficient of \(x^n\) gives
\[
 b_n-p^{n-1}\phi_p(b_n)\in R.
\]
For \(n\ge2\), iteration and \(p\)-adic completeness imply
\(b_n\in R\).  The converse is immediate.
\end{proof}

\begin{corollary}
\label{cor:fixed-center-integrality}
In the setting of Proposition~\ref{prop:fixed-higher-compatibility},
assume that the coefficient ring contains the support and coefficients of
\(\widehat\xi\).  Then
\[
 [x]U_{1,c}(x)
 =\frac{c^2-1}{24}
  \sum_z a_z\frac{z}{1-z},
 \qquad
 U_{1,c}(x)\in1+xR_p^\wedge[[x]].
\]
\end{corollary}

\begin{proof}
The linear term of
\(F_{\widehat\xi,p}(x)-cF_{\widehat\xi,p}(x_c)\) is the displayed quantity, since
\(B_2(1/2)=-1/12\).  It is integral because \(p>3\) and each
\(1-z\) is a unit.  Apply
Lemma~\ref{lem:Dwork-coefficient-criterion} to
\eqref{eq:fixed-higher-Dwork-quotient}.
\end{proof}

\section[Local calculations for the knot 5-2]{Local calculations for the knot $5_2$}\label{app:52-details}

\subsection{Shape class and selected degree-one places}

Retain \(F\), \(\alpha\), and \(\beta_{5_2}\) from
Subsection~\ref{sec:52}, and put
\[
 z_1=1-\alpha^2,
 \qquad
 z_2=1-\alpha.
\]
Thus \(\beta_{5_2}=2[z_1]+[z_2]\).
The cubic relation gives
\[
 z_2=\alpha^{-2}=(1-z_1)^{-1},
 \qquad
 z_1=-\alpha^3,
 \qquad
 1-z_1=\alpha^2.
\]
Moreover,
\[
 \partial\beta_{5_2}
 =2(-\alpha^3)\wedge\alpha^2+\alpha^{-2}\wedge\alpha
 =4(-1)\wedge\alpha=0,
\]
so \(\beta_{5_2}\in B(F)\).

For \(p=5,7,11,23\), choose the simple roots
\[
 \bar\alpha_p=2,4,9,15,
\]
respectively, of \(X^3-X^2+1\) modulo \(p\), and let
\(\alpha_p\in\Zp\) be their Hensel lifts.  Set
\[
 z_{1,p}=1-\alpha_p^2,
 \qquad
 z_{2,p}=1-\alpha_p.
\]
The required reductions are
\begin{center}
\small
\begin{tabular}{cccc}
\toprule
$p$ & $\alpha_p\bmod p^2$ & $z_{1,p}\bmod p^2$ & $z_{2,p}\bmod p^2$\\
\midrule
$5$  & $17$ & $12$  & $9$\\
$7$  & $4$  & $34$  & $46$\\
$11$ & $97$ & $30$  & $25$\\
$23$ & $38$ & $144$ & $492$\\
\bottomrule
\end{tabular}
\end{center}
At \(p=23\), the factor corresponding to \(15\) is simple even though
the discriminant is divisible by \(23\).

After inverting \(6\), Lemma~1.2 gives
\([z^{-1}]=-[z]\), while Lemmas~1.3 and 1.5 give
\([1-z]=-[z]\) \cite[Lemmas~1.2, 1.3, and 1.5, p.~181]{Suslin}.
Applying these relations successively yields
\[
 [z]=\left[\frac1{1-z}\right]
 =\left[1-\frac1z\right]
 =-[1-z]=-[z^{-1}]
 =-\left[\frac z{z-1}\right].
\]
Hence, for each selected place,
\begin{equation}\label{eq:52-three-shape}
 [z_{2,p}]=[z_{1,p}],
 \qquad
 \beta^B_{5_2,p}=3[z_{1,p}]
 \quad\text{in }B(\Qp)\otimes\Zp,
\end{equation}
and
\begin{equation}\label{eq:52-regulator-equality}
 D_p(\beta^B_{5_2,p})=3D_p(z_{1,p}).
\end{equation}

\subsection{Cyclotomic comparison classes}

Choose
\[
 (f_5,\lambda_5)=(4,2),
 \quad
 (f_7,\lambda_7)=(3,2),
 \quad
 (f_{11},\lambda_{11})=(5,3),
 \quad
 (f_{23},\lambda_{23})=(11,2),
\]
where \(\lambda_p\) has order \(f_p\) in \(\F_p^\times\), and let
\(\zeta_{f_p}=\teich(\lambda_p)\).  Define
\[
 \vartheta^B_p
 =\operatorname{cl}\bigl(
 [\zeta_{f_p}]_{\mathrm{cyc}}
 -[\zeta_{f_p}^{-1}]_{\mathrm{cyc}}
 \bigr),
\]
with \(K_3\)-images \(\vartheta^K_p\) and \(\beta^K_{5_2,p}\).
The local regulator theorem of \cite[Thm.~9]{GSWZ} gives
\begin{equation}\label{eq:GSWZ-regulator-isomorphism}
 D_p:K_3(\Qp;\Zp)\overset{\sim}{\longrightarrow}p^2\Zp.
\end{equation}
Evaluating \eqref{eq:finite-polylog-polynomial} at $a\in\F_p$,
Besser's congruence \cite[Cor.~2.2]{Besser} gives
\[
 p^{-2}D_p(\teich(a))
 \equiv\frac{\li_{2,p}(a)}{a-1}\pmod p.
\]
The values
\[
 \li_{2,5}(2)=1,
 \quad
 \li_{2,7}(2)=3,
 \quad
 \li_{2,11}(3)=6,
 \quad
 \li_{2,23}(2)=10
\]
show that every \(\vartheta^K_p\) is a generator.  Hence there is a
unique \(c_p\in\Zp\) such that
\begin{equation}\label{eq:52-scalar-ratio}
 \beta^K_{5_2,p}=c_p\vartheta^K_p,
 \qquad
 c_p=\frac{3D_p(z_{1,p})}{2D_p(\teich(\lambda_p))}.
\end{equation}

\subsection{Finite-polylogarithm criterion}

\begin{lemma}\label{lem:52-residue-disc-reduction}
Let \(p>3\), let \(z,1-z\in\Zp^\times\), and put
\(a=\bar z\in\F_p\setminus\{0,1\}\).  If
\[
 z=\teich(a)(1+pu),
 \qquad
 1-z=\teich(1-a)(1+pv)
 \pmod{p^2},
\]
then
\begin{equation}\label{eq:52-first-digit-formula}
 \frac{D_p(z)}{p^2}
 \equiv
 -\frac{\li_{2,p}(a)}{1-a}
 -\frac{\li_{1,p}(a)}{1-a}u
 +\frac{a}{2(1-a)}u^2
 +\frac12uv
 \pmod p.
\end{equation}
\end{lemma}

\begin{proof}
Apply Besser's root-of-unity reduction and residue-disc expansion
\cite[Cor.~2.2 and Prop.~2.3]{Besser} to the Teichm\"uller point
\(\teich(a)\), and use
\(
 D_p(z)=\Li_{2,p}(z)+\tfrac12\log_p(z)\log_p(1-z).
\)
\end{proof}

Put
\[
 h(T)=T^3-T^2+1,
 \qquad
 g(X)=X^3-2X^2+3X-1.
\]
Retain \(\bar\alpha\), \(\alpha_{p,\bar\alpha}\),
\(z_{p,\bar\alpha}\), \(a\), and \(W_p\) from
Subsection~\ref{sec:52}.
The identities
\[
 z_{p,\bar\alpha}=-\alpha_{p,\bar\alpha}^3,
 \qquad
 1-z_{p,\bar\alpha}=\alpha_{p,\bar\alpha}^2
\]
give \(u=3w\), \(v=2w\) when
\(
 \alpha_{p,\bar\alpha}=\teich(\bar\alpha)(1+pw).
\)
The finite logarithm relation yields
\(
 \li_{1,p}(a)=(a+2)w.
\)
Substitution in \eqref{eq:52-first-digit-formula} gives
\begin{equation}\label{eq:52-all-place-first-digit}
 \frac{D_p(z_{p,\bar\alpha})}{p^2}
 \equiv-\frac{W_p(a)}{1-a}\pmod p.
\end{equation}
Together with \eqref{eq:52-three-shape}, this proves
\begin{equation}\label{eq:52-all-place-generator-criterion}
 \beta^K_{5_2,p,\bar\alpha}\text{ generates }K_3(\Qp;\Zp)
 \quad\Longleftrightarrow\quad W_p(a)\ne0.
\end{equation}

A root-free form of the criterion is obtained as follows.  Set
\[
 H_p(X)=\gcd(g(X),X^p-X),
 \qquad
 H_p^{\mathrm{simp}}(X)
 =\frac{H_p(X)}{\gcd(H_p(X),g'(X))},
\]
and
\[
 \widetilde W_p(X)
 =(X+2)\li_{2,p}(X)+\frac32\li_{1,p}(X)^2.
\]
The number of simple degree-one places at which the class fails to generate
is
\begin{equation}\label{eq:52-root-free-count}
 \deg\gcd\bigl(H_p^{\mathrm{simp}}(X),\widetilde W_p(X)\bigr).
\end{equation}

\subsection{The four selected places}

For \(z=z_{1,p}\), write \(a_p=\bar z\) and determine \(u_p,v_p\) as in
Lemma~\ref{lem:52-residue-disc-reduction}.  The resulting data are
\begin{table}[ht]
\centering
\small
\begin{tabular}{crrrrr}
\toprule
$p$ & $a_p$ & $\teich(a_p)\bmod p^2$
& $\teich(1-a_p)\bmod p^2$ & $u_p$ & $v_p$\\
\midrule
$5$  & $2$ & $7$   & $24$  & $3$  & $2$\\
$7$  & $6$ & $48$  & $30$  & $2$  & $6$\\
$11$ & $8$ & $118$ & $81$  & $10$ & $3$\\
$23$ & $6$ & $466$ & $501$ & $13$ & $1$\\
\bottomrule
\end{tabular}

\medskip

\begin{tabular}{crrr}
\toprule
$p$ & $\li_{1,p}(a_p)$ & $\li_{2,p}(a_p)$
& $p^{-2}D_p(z_{1,p})\bmod p$\\
\midrule
$5$  & $4$ & $1$  & $2$\\
$7$  & $3$ & $0$  & $2$\\
$11$ & $4$ & $9$  & $1$\\
$23$ & $4$ & $19$ & $9$\\
\bottomrule
\end{tabular}
\caption{Residue-disc data for the four selected places.}
\label{tab:52-first-digit-data}
\end{table}
\FloatBarrier

Equations \eqref{eq:52-regulator-equality} and
\eqref{eq:GSWZ-regulator-isomorphism} show that both
\(\beta^K_{5_2,p}\) and \(\vartheta^K_p\) are generators.  Reducing
\eqref{eq:52-scalar-ratio} modulo \(p\) gives
\[
 c_5\equiv3\pmod5,
 \qquad
 c_7\equiv1\pmod7,
 \qquad
 c_{11}\equiv6\pmod{11},
 \qquad
 c_{23}\equiv14\pmod{23}.
\]

\subsection{Unramified scalar extension}

Let \(K_{p,d}/\Qp\) be the unramified extension of degree \(d\), let
\[
 E_p^{\mathrm{char}}
 =\Qp\bigl(\chi(a):\chi\bmod f_p,\ a\in(\Z/f_p\Z)^\times\bigr),
 \qquad
 L_{p,d}=K_{p,d}E_p^{\mathrm{char}},
\]
and extend all presentations and Iwasawa algebras to \(L_{p,d}\).  Write
\[
 \xi_{5_2,p,d}
 =c_p\bigl([\zeta_{f_p}]_{\mathrm{cyc}}
 -[\zeta_{f_p}^{-1}]_{\mathrm{cyc}}\bigr),
 \qquad
 \IwPs_{5_2,p,d}=\IwPs_{\xi_{5_2,p,d},p}.
\]
The odd Fourier formula and linearity give
\begin{equation}\label{eq:52-pseudomeasure-packet}
 \IwPs_{5_2,p,d}
 =\frac{2c_p}{\varphi(f_p)}
 \sum_{\substack{\chi\bmod f_p\\\chi(-1)=-1}}
 \IwPs_{\chi,f_p,p}.
\end{equation}
Since \(f_p\mid p-1\), Frobenius fixes the support roots.  Naturality of
the regulator and Theorem~\ref{thm:exceptional-residue} therefore give
\begin{equation}\label{eq:52-exceptional-residue}
 \Res_{s=-1}\mathcal M_{\IwPs_{5_2,p,d},-1}(s)
 =(1-p^{-1})(1-p^{-2})D_p(\beta^B_{5_2,p}).
\end{equation}


\bigskip
\noindent
\textsc{Honghuai Fang}\\
Institute for Theoretical Sciences, Westlake University,\\
No.~600 Dunyu Road, Xihu District, Hangzhou, Zhejiang 310030, China\\
\textit{Email address:}
\href{mailto:fanghonghuai@westlake.edu.cn}{\nolinkurl{fanghonghuai@westlake.edu.cn}}

\medskip
\noindent
\textsc{Zekun Chen}\\
Morningside Center of Mathematics, Chinese Academy of Sciences,\\
No.~55 Zhongguancun East Road, Haidian District, Beijing 100190, China\\
\textit{Email address:}
\href{mailto:chenzekun@amss.ac.cn}{\nolinkurl{chenzekun@amss.ac.cn}}

\end{document}